\documentclass[12pt]{amsart}
\usepackage[top=4.2cm, bottom=4cm, left=2.4cm, right=2.4cm]{geometry}
\usepackage[utf8]{inputenc}
\usepackage[USenglish]{babel}
\usepackage[T1]{fontenc} 
\usepackage{mathrsfs}
\usepackage{mathtools}
\usepackage{amsmath}
\usepackage{amssymb,epsfig}
\usepackage{caption}
\usepackage{subcaption}
\usepackage{amsthm}
\usepackage[absolute]{textpos}
\usepackage{mathtools,emptypage}
\usepackage[bookmarks=true]{hyperref}
\usepackage{xcolor}
\hypersetup{
    colorlinks,
    linkcolor={red!50!black},
    citecolor={blue!50!black},
    urlcolor={blue!80!black},
}
\usepackage{todonotes}
\usepackage{leftindex}

\mathtoolsset{showonlyrefs} 

\makeatletter
\newtheorem*{rep@theorem}{\rep@title}
\newcommand{\newreptheorem}[2]{%
\newenvironment{rep#1}[1]{%
 \def\rep@title{#2 \ref{##1}}%
 \begin{rep@theorem}}%
 {\end{rep@theorem}}}
\makeatother

\newreptheorem{theorem}{Theorem}

\newcommand{\N}{\mathbb{N}}
\newcommand{\Q}{\mathbb{Q}}
\newcommand{\R}{\mathbb{R}}

\newcommand{\sfd}{{\sf d}}

\renewcommand{\d}{{\mathrm d}}

\newcommand{\restr}[1]{\lower3pt\hbox{$|_{#1}$}}

\newcommand{\limi}{\varliminf}
\newcommand{\lims}{\varlimsup}
                          
\newcommand{\X}{{\rm X}}
\newcommand{\Y}{{\rm Y}}

\newcommand{\lip}{{\rm lip}}
\newcommand{\Lip}{{\rm Lip}}
\newcommand{\loc}{{\rm loc}}

\newcommand{\mm}{\mathfrak m}

\newcommand{\F}{{\rm F}}

\newcommand{\width}{{\rm width}}

\renewcommand{\SS}{{\rm SS}}

\makeatletter
\newcommand{\mytag}[2]{%
  \text{#1}%
  \@bsphack
  \begingroup
    \@onelevel@sanitize\@currentlabelname
    \edef\@currentlabelname{%
      \expandafter\strip@period\@currentlabelname\relax.\relax\@@@%
    }%
    \protected@write\@auxout{}{%
      \string\newlabel{#2}{%
        {\color{black}#1}%
        {\thepage}%
        {\@currentlabelname}%
        {\@currentHref}{}%
      }%
    }%
  \endgroup
  \@esphack
}
\makeatother

\def\Xint#1{\mathchoice
{\XXint\displaystyle\textstyle{#1}}%
{\XXint\textstyle\scriptstyle{#1}}%
{\XXint\scriptstyle\scriptscriptstyle{#1}}%
{\XXint\scriptscriptstyle\scriptscriptstyle{#1}}%
\!\int}
\def\XXint#1#2#3{{\setbox0=\hbox{$#1{#2#3}{\int}$ }
\vcenter{\hbox{$#2#3$ }}\kern-.6\wd0}}

\def\dashint{\Xint-}

\newtheorem{theorem}{Theorem}[section]

\newtheorem{corollary}[theorem]{Corollary}
\newtheorem{lemma}[theorem]{Lemma}
\newtheorem{proposition}[theorem]{Proposition}
\theoremstyle{definition}

\newtheorem{example}[theorem]{Example}

\newcounter{Counter}

\newtheorem{remark}[theorem]{Remark}

\newcommand{\Mod}{\textup{Mod}}
\newcommand{\cModepsilon}[2]{{\mathscr C}\text{-}{\rm Mod}_{#2}^{#1}}
\newcommand{\sfc}{{\sf c}}
\newcommand{\sfC}{{\sf C}}
\newcommand{\Chain}[1]{\mathscr{C}^{#1}}
\renewcommand{\sl}[1]{{\rm sl}_{#1}}
\newcommand{\UG}[1]{{\rm UG}^{#1}}
\newcommand{\LUG}[1]{{\rm LUG}^{#1}}

\newcommand{\WUG}[2]{{\rm WUG}_{#1}^{#2}}

\newcommand{\calL}[1]{\mathcal{L}^{#1}}
\newcommand{\HH}[2]{H^{1,#1}_{{\rm #2}}}
\renewcommand{\F}[1]{{\rm F}_{{\rm #1}}}
\newcommand{\relF}[1]{\tilde{\rm F}_{{\rm #1}}}
\title[Sobolev spaces via chains in metric measure spaces]{Sobolev spaces via chains in metric measure spaces}

\author[Emanuele Caputo]{Emanuele Caputo}\address[Emanuele Caputo]{Mathematics Institute, Zeeman Building, University of Warwick, Coventry, CV4 7AL, United Kingdom}\email{emanuele.e.caputo@jyu.fi}
\author[Nicola Cavallucci]{Nicola Cavallucci}\address[Nicola Cavallucci]{
Institute of Mathematics, EPFL, Station 8, 1015 Lausanne, Switzerland}\email{n.cavallucci23@gmail.com}

\keywords{Sobolev spaces, metric measure spaces, Poincar\'{e} inequality, chain upper gradients}
\subjclass[2020]{46E36, 30L99, 49J52}

\begin{document}
\begin{abstract}
    We define the chain Sobolev space on a possibly non-complete metric measure space in terms of chain upper gradients. In this context, $\varepsilon$-chains are a finite collection of points with distance at most $\varepsilon$ between consecutive points. They play the role of discrete versions of curves. Chain upper gradients are defined accordingly and the chain Sobolev space is defined by letting the size parameter $\varepsilon$ going to zero. In the complete setting, we prove that the chain Sobolev space is equal to the classical notions of Sobolev spaces in terms of relaxation of upper gradients or of the local Lipschitz constant of Lipschitz functions. 
    The proof of this fact is inspired by a recent technique developed by Eriksson-Bique in \cite{Eriksson-Bique-Density-in-energy}.
    In the possible non-complete setting, we prove that the chain Sobolev space is equal to the one defined via relaxation of the local Lipschitz constant of Lipschitz functions, while in general they are different from the one defined via upper gradients along curves. 
    We apply the theory developed in the paper to prove equivalent formulations of the Poincar\'{e} inequality in terms of pointwise estimates involving $\varepsilon$-upper gradients, lower bounds on modulus of chains connecting points and 
    size of separating sets measured with the Minkowski content in the non-complete setting.
    Along the way, we discuss the notion of weak $\varepsilon$-upper gradients and asymmetric notions of integral along chains.
\end{abstract}

\maketitle

\tableofcontents

\section{Introduction}

A fundamental research direction in analysis on metric spaces is the development of calculus with Sobolev functions and Lipschitz functions defined on metric measure spaces $(\X,\sfd,\mm)$.
After the work of Shanmuganlingam \cite{Shanmu00}, the Sobolev seminorm of a Borel function $u\colon \X \to \R$ is the infimum of the $L^p(\X)$-norms of all upper gradients of $u$, where a function $g\colon \X \to [0,+\infty]$ is an upper gradient of $u$ if the following weak version of fundamental theorem of calculus holds for every rectifiable curve $\gamma\colon [0,1] \to \X$:
\begin{equation*}
    |u(\gamma_1)-u(\gamma_0)| \le \int_0^1 g(\gamma_t)|\dot{\gamma}_t|\,\d t.
\end{equation*}
We refer the reader to the classical textbook \cite{HKST15}.

Another classical approach in the subject concerning the case $p > 1$ is based on a relaxation procedure of appropriate functionals with respect to the $L^p(\X)$-topology, playing the role of the Dirichlet energy. In Section \ref{subsec:relaxation} we will recall these relaxation procedures in detail.
On one side, Cheeger in \cite{Cheeger99} considered the relaxation of the $L^p(\X)$-norm of upper gradients. Here we denote the corresponding Banach space by $\HH{p}{curve}(\X)$. On the other hand, Ambrosio, Gigli and Savar\'{e} in \cite{AmbrosioGigliSavare11-3, AmbrosioGigliSavare14} studied the relaxation of the $L^p(\X)$-norm of the local Lipschitz constant $\lip\,u$ of Lipschitz functions $u$. The corresponding Banach space is denoted by $\HH{p}{AGS}(\X)$. For details on these spaces we refer to Section \ref{sec:Cheeger-AGS}.

Other approaches are available, like the ones defined via integrations along test plans \cite{AmbrosioGigliSavare11-3, AmbrosioGigliSavare14}, but this approach will not be used in this work. We refer the reader to the recent survey \cite{AmbIkoLucPas2024}. 

The main results of \cite{AmbrosioGigliSavare11-3, AmbrosioDiMarino14, Eriksson-Bique-Density-in-energy, LucicPasqualetto2024} show that the spaces $\HH{p}{curve}(\X)$ and $\HH{p}{AGS}(\X)$ coincide for $p\ge 1$, if the metric space $(\X,\sfd)$ is complete. However, for non complete metric spaces they can be different, see Example \ref{ex:AGS_different_curve}.

This work comes from the following question: is it possible to show that $\HH{p}{AGS}(\X)$ is equal to a space obtained via relaxation in terms of a suitable notion of upper gradients when the metric space $(\X,\sfd)$ is not assumed to be complete?

The answer is affirmative and leads to an alternative definition of Sobolev or $BV$ spaces, expressed in terms of chains instead of curves. We recall that a $\varepsilon$-chain, for $\varepsilon > 0$, is a finite collection of points $\sfc = \{q_i\}_{i=0}^N$ such that $\sfd(q_i,q_{i+1})\le \varepsilon$ for every $i=0,\ldots,N-1$.

The integration along rectifiable paths into Shanmugalingam's definition is replaced by integration along $\varepsilon$-chains. With this analogy in mind, a function $g$ is a $\varepsilon$-upper gradient of $u$ provided that
\begin{equation*}
    |u(q_N)-u(q_0)| \le \sum_{i=0}^{N-1}\frac{g(q_i) + g(q_{i+1})}{2}\sfd(q_i,q_{i+1}) =:\int_{\sfc} g
\end{equation*}
for every $\varepsilon$-chain $\sfc = \{q_i\}_{i=0}^N$. This can be seen as a discrete analogue of the integral along a rectifiable path, see Proposition \ref{prop:UGeta_subset_UG} for a precise statement. The set of all $\varepsilon$-upper gradients of $u$ is denoted by $\UG{\varepsilon}(u)$. 
The corresponding functional is
\begin{equation}
        \F{\Chain{}} \colon L^{p}(\X) \to  [0,+\infty], \quad u \mapsto \lim_{\varepsilon \to 0} \inf\left\{ \| g \|_{L^p(\X)}\,:\, g\in \UG{\varepsilon}(u)\right\}.
\end{equation}
The Banach space obtained by relaxation of $\F{\Chain{}}$ is denoted by $\HH{p}{\Chain{}}(\X)$.

The first result shows the equivalence of the spaces introduced so far, if the metric space is complete.

\begin{theorem}
\label{theo-intro:equivalences-of-all-norms-complete-spaces}
    Let $(\X,\sfd,\mm)$ be a metric measure space such that $(\X,\sfd)$ is complete. Then
    \begin{equation}
        \HH{p}{\Chain{}}(\X) = \HH{p}{AGS}(\X) = \HH{p}{curve}(\X)
    \end{equation}
    and
    \begin{equation}
        \|u \|_{\HH{p}{\Chain{}}(\X)} = \|u \|_{\HH{p}{AGS}(\X)} = \|u \|_{\HH{p}{curve}(\X)}
    \end{equation}
    for every $u\in L^p(\X)$.
\end{theorem}

The second result, which answers the aforementioned question, establishes the equality between $\HH{p}{AGS}(\X)$ and $\HH{p}{\Chain{}}(\X)$ also for non complete metric spaces. It is obtained from Theorem \ref{theo-intro:equivalences-of-all-norms-complete-spaces} and the fact that $\HH{p}{AGS}(\X)=\HH{p}{AGS}(\bar{\X})$ and $\HH{p}{\Chain{}}(\X)=\HH{p}{\Chain{}}(\bar{\X})$, where $\bar{\X}$ is the metric completion of $\X$, see Proposition \ref{prop:AGS_completion} and Theorem \ref{theo:natural_isometries_chains}.

\begin{theorem}
\label{theo-intro:equivalences-of-AGS-Chain}
    Let $(\X,\sfd,\mm)$ be a metric measure space, not necessarily complete. Then
    \begin{equation}
        \HH{p}{\Chain{}}(\X) = \HH{p}{AGS}(\X)
    \end{equation}
    and
    \begin{equation}
        \|u \|_{\HH{p}{\Chain{}}(\X)} = \|u \|_{\HH{p}{AGS}(\X)}
    \end{equation}
    for every $u\in L^p(\X)$.
\end{theorem}

Other possible notions of integral along chains are considered in the paper, leading to the definition of the same space.

\subsection{The proof of the main results}

The proof of Theorem \ref{theo-intro:equivalences-of-all-norms-complete-spaces}
is inspired by the approximation method developed by Eriksson-Bique in \cite{Eriksson-Bique-Density-in-energy}. It is not apriori clear the relation between $\HH{p}{\Chain{}}(\X)$ and $\HH{p}{curve}(\X)$. Therefore we introduce an auxiliary space, that we denote as $\HH{p}{\Chain{},\, \Lip}(\X)$ and it is defined as the domain of finiteness of the lower semicontinuous envelop of the following energy
\begin{equation}
        \F{\Chain{},\, \Lip} \colon L^{p}(\X) \to  [0,+\infty], \quad u \mapsto \begin{cases}
            \underset{\varepsilon\to 0}{\lim}\inf\left\{ \| g \|_{L^p(\X)}\,:\, g\in \UG{\varepsilon}(u) \cap {\rm Lip}(\X) \right\} &\text{if } u\in \Lip(\X),\\
            +\infty &\text{otherwise}.
        \end{cases}
\end{equation}

The space is normed with the sum of the $L^p$-norm and the relaxation of the energy functional above.
The functions for which $\F{\Chain{},\,\Lip}$ is finite play the role of regular functions, being Lipschitz and with Lipschitz upper gradients. They form a regular class of functions for which one could hope to get density in energy in $\HH{p}{curve}(\X)$, in full generality.

On one side, one easily gets 
\begin{equation}
\label{eq:trivial_inclusion_of_spaces}
    \HH{p}{\Chain{},\Lip}(\X)\subseteq\HH{p}{curve}(\X) \qquad\text{and}\qquad \HH{p}{\Chain{},\Lip}(\X)\subseteq\HH{p}{\Chain{}}(\X)
\end{equation}
and that the inclusions are $1$-Lipschitz.

The proof of Theorem \ref{theo-intro:equivalences-of-all-norms-complete-spaces} is achieved into two steps, respectively proving that the reverse inclusions in \eqref{eq:trivial_inclusion_of_spaces} hold and are $1$-Lipschitz.

To do both, we follow the mentioned approximation scheme in \cite{Eriksson-Bique-Density-in-energy}, that we briefly recall, in a simplified form, for reader's convenience in the case of the proof of the inclusion $\HH{p}{curve}(\X) \subseteq \HH{p}{\Chain{},\Lip}(\X)$. For every given $u\in L^p(\X)$, proceed using the following steps.
\begin{itemize}
    \item[(Step 1)] Reduce the proof to the case where $u$ is bounded, with bounded support and nonnegative.
    \item[(Step 2)] For every upper gradient along curves $g$, define Lipschitz functions $g_j$ that converge to $g$ pointwise and in $L^p(\X)$.
    \item[(Step 3)] Define the functions 
    $$u_j(x) := \inf\left\{ u(q_0) + \int_\sfc g_j \,:\, \sfc = \{q_i\}_{i=0}^N \text{ is a }\frac{1}{j}\text{-chain such that } q_N = x\right\}.$$
    Prove that $u_j$ is Lipschitz, has bounded support and it has $g_j$ as $\frac{1}{j}$-upper gradient.
    \item[(Step 4)] Conclude the proof showing that $u_j$ converges to $u$ in $L^p(\X)$ via a contradiction argument. The contradiction is obtained by violating the fact that $g$ is an upper gradient of $u$.
\end{itemize}

So, in order to prove that $\HH{p}{curve}(\X) \subseteq \HH{p}{\Chain{},\Lip}(\X)$ one could follow this scheme with minor modifications. However, there are technical issues in adapting the proof to show that $\HH{p}{\Chain{}}(\X) \subseteq \HH{p}{\Chain{},\Lip}(\X)$, that we briefly present in a moment. That is why we will give a different proof of $\HH{p}{curve}(\X) \subseteq \HH{p}{\Chain{},\Lip}(\X)$ too, that can be easily adapted to show $\HH{p}{\Chain{}}(\X) \subseteq \HH{p}{\Chain{},\Lip}(\X)$.

The main problem is related to the reduction in Step 1.
As stated, this is a consequence of the well known locality property of the minimal $p$-weak upper gradient along curves, which can be derived from a Leibniz rule. The same locality property for $p$-weak $\varepsilon$-upper gradients does not hold, and it is also not clear if it is true in a limit sense for $\varepsilon$ going to zero (see Remark \ref{rem:comments_between_CC_EB}). However, a weak Leibniz rule for $\varepsilon$-chain upper gradients holds (see Proposition \ref{prop:Leibniz-rule-chains}). Using it we are able to reduce the proof, also for $\HH{p}{\Chain{}}(\X)$, to bounded functions with bounded support, but \emph{not necessarily nonnegative}.

This difference creates two additional difficulties in the scheme sketched above. The first one is that the approximating functions $u_j$ are not necessarily with bounded support. This requires an additional cutoff argument.
The second difference involves the core of the proof, namely the contradiction argument in Step 4. We need to analyze separately three different cases, one of which is the only one that needed to be considered in \cite{Eriksson-Bique-Density-in-energy}. For more details we refer to Step 8 of the proof of Theorem \ref{theo:density_energy_complete}.

Moreover, the adaptation of the proof to the inclusion $\HH{p}{\Chain{}}(\X) \subseteq \HH{p}{\Chain{},\Lip}(\X)$ requires an additional care in the contradiction argument in Step 4. This is due to the fact that every $\varepsilon$-upper gradient, that takes extended values, satisfies the upper gradient inequality along many curves but not necessarily along all of them, see Proposition \ref{prop:UGeta_subset_UG}.

\subsection{Application to characterizations of Poincar\'{e} inequality}
A standard consequence of theorems like Theorem \ref{theo-intro:equivalences-of-AGS-Chain} is the equivalence between apriori different formulations of the Poincar\'{e} inequality. Tailoring such a discussion in our case, we have the equivalence between the Poincar\'{e} inequality formulated for couples $(u, \lip\, u)$ with $u\in \Lip(\X)$ and $(u,g)$ where $u$ is Borel and $g \in \UG{\varepsilon}(u)$ for some $\varepsilon >0$, see Corollary \ref{cor:PI_lip=chain}. This result holds also in non-complete metric space. While in the complete setting the conditions above are also equivalent to the Poincar\'{e} inequality formulated for couples $(u,g)$ where $u$ is Borel and $g$ is an upper gradient of $u$, this is no more true if the metric measure space is not complete, see Remark \ref{rem:example_not_equivalence_poincare}.

The Poincaré inequality on complete and doubling metric measure spaces can be characterized in at least three ways: Heinonen's pointwise estimates in \cite{Hei01}, Keith's modulus estimates in \cite{Kei03} and via energy of separating sets for $p=1$ in \cite{CaputoCavallucci2024}.
We reinterpret these characterizations also for non-complete spaces in terms of chains respectively in Sections \ref{sec:pt_estimates_chains}, \ref{sec:keith_chain} and \ref{sec:separating_sets_chains}. To treat the analogue of Keith's estimate we need to introduce the notion of modulus of a family of chains and to study its basic properties. This is done in Section \ref{sec:weak_eps_upper_gradient}.

Interestingly, the approach via chains allows to improve a result concerning upper gradients along curves by relaxing the assumptions on the metric measure space. Indeed in Proposition \ref{prop:pt_PI_equivalent_conditions_A_p_connectedness} and Remark \ref{rem:drop_quasiconvexity} we prove the equivalence of the validity of the following Heinonen's pointwise estimates at a fixed couple of points $x,y \in \X$:
\begin{equation*}
    \vert u(x) - u(y) \vert ^p \le C\sfd(x,y)^{p-1}\int (\lip\,u)^p \,\d\mm_{x,y}^L\qquad \text{for all }u \in {\rm Lip}(\X)
\end{equation*}
and
\begin{equation*}
    \vert u(x) - u(y) \vert ^p \le C\sfd(x,y)^{p-1}\int g^p \,\d\mm_{x,y}^L\qquad\text{for all $u$ Borel and $g$ upper gradient of $u$.}
\end{equation*}
Here $\mm_{x,y}^L$ is the measure that is absolutely continuous with respect to the reference measure with density being the truncated Riesz potential with poles $x$ and $y$ (see Section \ref{sec:pt_estimates_chains}).
This result was proved by the authors in \cite[Theorem A.3]{CaputoCavallucci2024II}, under the additional assumption of local quasiconvexity of the metric space. The use of chains allows to remove this additional assumption. It is likely that similar improvements can be performed in other situations.
\subsection{Structure of the paper}
Section \ref{sec:preliminaries} contains general facts about measure theory, curves and chains on metric spaces. In Section \ref{sec:Cheeger-AGS} we recall the definition of Sobolev spaces via a relaxation approach. In Section \ref{sec:chain_upper_gradients} we define chain upper gradients and we derive basic properties and relations with the classical upper gradients along curves. Section \ref{sec:weak_eps_upper_gradient} introduces the notions of modulus of a family of chains and of weak chain upper gradient. Similarities and differences with the theory of weak upper gradients along curves are shown. In Section \ref{sec:chain_sob_space} we define the Sobolev spaces via chain upper gradients. In Section \ref{sec:proof_main_result} we prove the main results, Theorems \ref{theo-intro:equivalences-of-all-norms-complete-spaces} and \ref{theo-intro:equivalences-of-AGS-Chain}.
Section \ref{sec:PI} contains equivalent formulations of Poincar\'{e} inequality in the possibly non-complete setting in terms of chains.
\subsection{Acknowledgments}
The first author is supported by the European Union’s Horizon 2020 research and innovation programme (Grant
agreement No. 948021).
\section{Preliminaries}
\label{sec:preliminaries}

Let $(\X,\sfd,\mm)$ be a metric measure space. With this term, we mean that $(\X,\sfd)$ is a separable metric space which is not necessarily complete and $\mm$ is a non-trivial outer measure which is Radon and finite on bounded sets. If $(\X,\sfd,\mm)$ is a metric measure space in the sense of \cite[p.62]{HKST15}, i.e it is a triple such that
$(\X,\sfd)$ is complete and separable and $\mm$ is a non-trivial outer measure which is Borel regular and finite on bounded sets, then $\mm$ is Radon by \cite[Proposition 3.3.44]{HKST15}, and so $(\X,\sfd,\mm)$ is a metric measure space in the sense above. This is false if $(\X,\sfd)$ is not complete.
If a triple $(\X,\sfd,\mm)$ is such $(\X,\sfd)$ is separable and $\mm$ is a non-trivial outer measure which is finite on bounded sets, then $\mm$ is Radon if and only if the space $(\bar{\X},\bar{\sfd},\bar{\mm})$ is again a metric measure space as defined above (cp. \cite[Proposition 3.3.46]{HKST15}: the proof still works with our assumption of finiteness on bounded set in place of local finiteness), where $(\bar{\X},\bar{\sfd})$ denotes the completion of $(\X,\sfd)$ and $\bar{\mm}$ is the outer measure $\bar{\mm}(E) := \mm(E\cap \X)$ for every $E\subseteq \bar{\X}$.

A metric measure space $(\X,\sfd,\mm)$ is said to be doubling if there exists $C_D \ge 1$ such that
\begin{equation}
    \label{eq:defin_doubling_measure}
    \mm(B_{2r}(x)) \le C_D \mm(B_r(x)) \qquad \text{ for all } x\in \X, r>0.
\end{equation}
If $(\X,\sfd,\mm)$ is doubling and $(\X,\sfd)$ is complete, then $(\X,\sfd)$ is a proper metric space, i.e. every closed bounded set is compact.

We denote by $\calL{p}(\X)$ the space of functions $u\colon \X \to \R$ such that $\int \vert u \vert^p \,\d \mm < \infty$, and by $L^p(\X)$ its quotient by the equivalence relation that identifies two functions if they agree $\mm$-almost everywhere. The $L^p(\X)$ norm will be denoted by $\Vert \cdot \Vert_{L^p(\X)}$. The class of Lipschitz functions on $\X$ is denoted by $\Lip(\X)$.

A function $u\colon \X \to \R$ is bounded if there exists $M \ge 0$ such that $\vert u \vert \le M$. It has bounded support if there exists a bounded subset $B\subseteq \X$ such that $u\equiv 0$ on $\X \setminus B$. The slope of $u\colon \X \to \mathbb{R}$ at size $\varepsilon > 0$ is defined as
\begin{equation*}
    \sl{\varepsilon} u(x):=\sup_{y \in B_\varepsilon(x)}\frac{|u(y)-u(x)|}{\sfd(y,x)}.
\end{equation*}
If $\varepsilon' < \varepsilon$ then $\sl{\varepsilon'}u \le \sl{\varepsilon} u$. The local Lipschitz constant of $u$ is defined as 
$$\lip\, u(x):=\lim_{\varepsilon \to 0} \sl{\varepsilon} u(x)= \inf_{\varepsilon > 0} \sl{\varepsilon} u(x).$$

\subsection{Relaxation of functionals}
\label{subsec:relaxation}
Let $1\le p < \infty$ and let $\F{}\colon \calL{p}(\X) \to [0,+\infty]$ be a functional such that 
\begin{itemize}
    \item[(a)] $\F{}(0) = 0$,
    \item[(b)] $\F{}(u+v) \le \F{}(u) + \F{}(v)$,
    \item[(c)] $\F{}(\lambda u) = \vert \lambda \vert \F{}(u)$
\end{itemize}
for every $u,v\in \calL{p}(\X)$ and $\lambda \in \R$. By definition, the relaxation of $\F{}$ is the biggest functional $\relF{} \colon \calL{p}(\X) \to [0,+\infty]$ which is lower semicontinuous with respect to the $L^p(\X)$-norm and such that $\relF{} \le \F{}$. A concrete description of $\relF{}$ is given by
\begin{equation}
    \label{eq:defin_relaxation_full_generality}
    \relF{}(u) = \inf\left\{ \limi_{j\to +\infty}\F{}(u_j)\, : \, u_j \underset{L^p(\X)}{\longrightarrow} u\right\}.
\end{equation}
By the definition above, $\relF{}$ induces a functional on $L^p(\X)$ which is lower semicontinuous and satisfies properties (a), (b) and (c) above. Therefore one can define the space
\begin{equation}
    \label{eq:defin_H_full_generality}
    \HH{p}{F}(\X) := \left\{ u \in L^p(\X)\,:\, \relF{}(u) < +\infty \right\}
\end{equation}
endowed with the norm
\begin{equation}
    \label{eq:defin_norm_H_full_generality}
    \Vert u \Vert_{\HH{p}{F}(\X)}^p := \Vert u \Vert_{L^p(\X)}^p + \relF{}(u)^p.
\end{equation}
The normed space $(\HH{p}{F}(\X), \Vert \cdot \Vert_{\HH{p}{F}(\X)})$ is a Banach space since $\relF{}$ is $L^p(\X)$-lower semicontinuous.
Indeed, given a $\HH{p}{F}(\X)$-Cauchy sequence $u_j$, that is also $L^p(\X)$-Cauchy, we can extract an $L^p(\X)$-limit $u$. By lower semicontinuity, $\F{}(u-u_k)\le \limi_{j\to +\infty} \F{}(u_j-u_k)$ for every $k$, thus
\begin{equation*}
    0 \le \lims_{k\to +\infty} \F{}(u-u_k)\le \lims_{k\to +\infty}\limi_{j\to +\infty} \F{}(u_j-u_k)=0
\end{equation*}
where the last equality follows by $\{u_j\}$ being $\HH{p}{F}(\X)$-Cauchy. Throughout the whole paper we will consider several functionals $\F{}$, where it can be readily checked that they always satisfy properties (a), (b) and (c) above.

\subsection{Curves}
A curve is a continuous function $\gamma \colon [a,b] \to \X$ for some $a,b \in \mathbb{R}$ with $a<b$. The starting and final point of $\gamma$ are respectively $\alpha(\gamma):=\gamma(a)$ and $\omega(\gamma) := \gamma(b)$.
The length of a curve $\gamma$ is defined as
\begin{equation}
    \label{eq:defin_length_curve}
    \ell(\gamma):=\sup \left\{ \sum_{i=0}^{N-1} \sfd(\gamma_{t_i},\gamma_{t_{i+1}})\,:\,a=t_0 < t_1 < \dots < t_N=b,\, N \in \mathbb{N}
    \right\}.
\end{equation}
A curve of finite length is called rectifiable. Every rectifiable curve $\gamma \colon [a,b] \to \X$ admits a reparametrization $s_\gamma \colon [0,\ell(\gamma)] \to [a,b]$ by arc-length. This means that the curve $\gamma' := \gamma \circ s_\gamma \colon [0,\ell(\gamma)] \to \X$ satisfies $\ell(\gamma'\restr{[0,t]}) = t$ for every $t\in [0,\ell(\gamma)]$. 
Given two points $x,y\in \X$, we denote by $\Gamma_{x,y}$ the set of rectifiable curves $\gamma$ with $\alpha(\gamma)=x$ and $\omega(\gamma)=y$.

The integral of a Borel function $g\colon \X \to [0,+\infty]$ over a rectifiable curve $\gamma$ is
\begin{equation}
    \label{eq:defin_integral_over_curve}
    \int_\gamma g := \int_0^{\ell(\gamma)} g(\gamma(s_\gamma(t)))\,\d t.
\end{equation}

\subsection{Chains}
Let $(\X,\sfd)$ be a metric space and $\varepsilon > 0$.
A $\varepsilon$-chain is a finite collection of points $\{ q_i \}_{i=0}^N$ such that $\sfd(q_i,q_{i+1})\le \varepsilon$. The set of all $\varepsilon$-chains of $\X$ is denoted by $\Chain{\varepsilon}(\X)$. The set of all chains of $\X$ is $\Chain{}(\X) := \bigcup_{\varepsilon > 0} \Chain{\varepsilon}(\X)$. When the context is clear we simply write $\Chain{\varepsilon}$ and $\Chain{}$.
More generally, if $E$ is a subset of $\X$, we set $\Chain{}(E) := \{ \sfc = \{q_i\}_{i=0}^N \in \Chain{} \,:\, q_i \in E \text{ for some } 0\le i \le N\}$ and $\Chain{\varepsilon}(E) := \Chain{}(E) \cap \Chain{\varepsilon}$. Notice that the two definitions of $\Chain{\varepsilon}(\X)$ and $\Chain{}(\X)$ are consistent.

The first and last points of $\sfc = \{q_i\}_{i=0}^N \in \Chain{}$ are respectively $\alpha(\sfc) := q_0$ and $\omega(\sfc) = q_N.$
The concatenation of $\sfc = \{q_i\}_{i=0}^N \in \Chain{\varepsilon}, \sfc' =\{q_i'\}_{i=0}^{N'} \in \Chain{\varepsilon'}$ such that $\omega(\sfc) = \alpha(\sfc')$ is defined as
\begin{equation}
    \label{eq:defin_concatenation}
    \sfc\star \sfc' = \{q_0,\ldots,q_N = q_0',q_1',\ldots,q_{N'}'\}.
\end{equation}
Notice that $\sfc\star \sfc' \in \Chain{\varepsilon \vee \varepsilon'}$ and $\alpha(\sfc\star \sfc') = \alpha(\sfc)$, $\omega(\sfc\star \sfc') = \omega(\sfc')$. The inverse of a chain $\sfc = \{q_0,\ldots,q_N\} \in \Chain{}$ is $-\sfc := \{q_N,\ldots,q_0\}$.
If $x,y\in \X$ are two points then we set $\Chain{}_{x,y} := \{\sfc \in \Chain{} \, : \, \alpha(\sfc) = x, \omega(\sfc) = y\}$ and $\Chain{\varepsilon}_{x,y} := \Chain{}_{x,y} \cap \Chain{\varepsilon}$. 

A metric space $(\X,\sfd)$ is said to be $\varepsilon$-chain connected if $\Chain{\varepsilon}_{x,y} \neq \emptyset$ for every $x,y \in \X$. A metric space can be decomposed in $\varepsilon$-chain connected components in the following way. Given two points $x,y\in \X$ we say that $x \sim_\varepsilon y$ if and only if $\Chain{\varepsilon}_{x,y} \neq \emptyset$. This defines an equivalence relation on $\X$. Such a relation partitions $\X$ into a family of sets $\{A_i\}_{i \in I}$. If $(\X,\sfd)$ is separable, it can be readily checked that the set of indices $I$ is countable. Moreover, every set $A_i$ is $\varepsilon$-chain connected: it is called a $\varepsilon$-chain connected component of $\X$. 
By definition, every $\varepsilon$-chain connected component is both open and closed.
Moreover, we have that 
\begin{equation}
\label{eq:chain_connected_components_positive_distance}
    \sfd(A_i,A_j) \ge \varepsilon,\qquad \text{if }i \neq j.
\end{equation}

Let $\varepsilon' \le \varepsilon$ and let $\X =\bigcup_{i\in I_\varepsilon} A_i^\varepsilon =\bigcup_{i\in I_{\varepsilon'}} A_i^{\varepsilon'}$, be the two decompositions where $\{A_i^\varepsilon\}_{i\in I_\varepsilon}$ and $\{A_i^{\varepsilon'}\}_{i \in I_{\varepsilon'}}$ are respectively the $\varepsilon$ and $\varepsilon'$-chain connected components of $\X$. Then, for every $i \in I_{\varepsilon'}$ there exists $j\in I_{\varepsilon}$ such that $A_i^{\varepsilon'} \subseteq A_j^{\varepsilon}$. The set of $\varepsilon$-chain connected components of $\X$ is denoted by $\Chain{\varepsilon}\textup{-cc}(\X)$.

Given $\sfc=\{ q_i \}_{i=0}^N \in \Chain{}$ and a function $g \colon \X \to [0,+\infty]$ we define
\begin{equation}
\label{eq:defin_integral_on_chains}
    \int_{{\sfc}} g := \sum_{i=0}^{N-1}\frac{g(q_i) + g(q_{i+1})}{2}\sfd(q_i,q_{i+1}).
\end{equation}
For every function $g$ and every two chains $\sfc,\sfc' \in \Chain{}$ it holds
\begin{equation}
\label{eq:integral_is_symmetric}
    \int_{\sfc \star \sfc'} g = \int_\sfc g + \int_{\sfc'} g,\qquad \int_{-\sfc} g = \int_\sfc g.
\end{equation}
Moreover the integral over a fixed chain $\sfc$ is linear, i.e. $\int_\sfc a g + b h = a \int_\sfc g + b \int_\sfc h$ for every $g,h \colon \X \to [0,+\infty]$ and every $a,b \ge 0$. 

The length of a chain $\sfc = \{q_i\}_{i=0}^N \in \Chain{}$ is $\ell(\sfc):=\int_{\sfc} 1 =\sum_{i=0}^{N-1}\sfd(q_i,q_{i+1})$. 
\begin{remark}
\label{rmk:lambda_integral_over_chains}
    There is not a canonical way to define the integral over a chain. Let $\lambda \in [0,1]$. If $a,b \in \R \cup \{+\infty\}$ we set $[a,b]_\lambda := \lambda a + (1-\lambda) b$. Given $\sfc=\{ q_i \}_{i=0}^N \in \Chain{}$, a function $g \colon \X \to [0,+\infty]$ and $\lambda \in [0,1]$, we define the $\lambda$-integral of $g$ over $\sfc$ as
    \begin{equation}
    \label{eq:defin_integral_on_chains_lambda}
        \leftindex^\lambda{\int}_{{\sfc}} g := \sum_{i=0}^{N-1}[g(q_i),g(q_{i+1})]_\lambda\sfd(q_i,q_{i+1}).
    \end{equation}
    When $\lambda = \frac12$ we recover the definition in \eqref{eq:defin_integral_on_chains}, while for $\lambda = 1$ we find the expression used in \cite{Eriksson-Bique-Density-in-energy}. 
    The $\lambda$-integral is linear in the sense above and it satisfies the first equality of \eqref{eq:integral_is_symmetric}. The second equality of \eqref{eq:integral_is_symmetric} becomes
    $\leftindex^\lambda{\int}_{-\sfc} g = \leftindex^{1-\lambda}{\int}_\sfc g$ for every $g\colon \X \to [0,+\infty]$ and every $\sfc \in \Chain{}$. In the paper we will present the results for the $\frac12$-integral, for simplicity, and we will briefly comment on how it works for different values of $\lambda \in [0,1]$.
\end{remark}

\subsection{Convergence of chains to curves}
We recall the notion of convergence of a sequence of chains to a curve defined in \cite{Eriksson-Bique-Density-in-energy}. Given a chain $\sfc = \{q_i\}_{i=0}^N$ we define the set of interpolating times as $(t_0, \ldots, t_N)$ by $t_0 = 0$ and $t_i = \frac{\ell(\{q_0,\ldots,q_i\})}{\ell(\sfc)}$. Then we define the function $\gamma_{\sfc}\colon [0,1] \to \X$ piecewisely defined by $[t_i,t_{i+1}) \ni t \mapsto q_i$, for $i=0,\ldots,N$. We say that a sequence of chains $\{\sfc_j\}_j$ converges to a curve $\gamma \colon [0,1] \to \X$ if $\{\gamma_{\sfc_j}\}_j$ converges uniformly to $\gamma$ as $j$ goes to $+\infty$. We have the following compactness result for complete metric spaces.
    \begin{proposition}[{\cite[Lemma 2.18]{Eriksson-Bique-Density-in-energy}}]
    \label{prop:Sylvester_compactness_chains}
    Let $(\X,\sfd)$ be a complete metric space. Let $K_j \subseteq \X$ be an increasing sequence of compact subsets of $\X$ and let $h_j(x):= \sum_{i=1}^j \min\{ j\sfd(x,K_i),1\}$. Let $M,L,\Delta > 0$ be constants. Let $\sfc_j = \{q_0^j,\ldots,q_{N_j}^j\} \in \Chain{\frac{1}{j}}(\X)$ be chains such that:
    \begin{itemize}
        \item[(i)] $\ell(\sfc_j) \le L$ for every $j$;
        \item[(ii)] ${\rm Diam}(\sfc_j) := \max\{ \sfd(q,q') \, : \, q,q'\in \sfc_j\} \ge \Delta$ for every $j$;
        \item[(iii)] $\sum_{m=0}^{N_j-1} h_j(q_m^j)\sfd(q_m^j,q_{m+1}^j) \le M$ for every $j$.
    \end{itemize}
    Then there exists a subsequence of $\{\sfc_j\}_j$ that converges to a curve $\gamma \colon [0,1] \to \X$. 
\end{proposition}

\begin{remark}
    \label{rmk:compactness_chains_with_lambda}
    Condition (iii) in Proposition \ref{prop:Sylvester_compactness_chains} is exactly $\leftindex^1{\int}_{\sfc_j} h_j \le M$, for every $j$. The proof of \cite[Lemma 2.18]{Eriksson-Bique-Density-in-energy} can be straightforwardly modified replacing (iii) with
    \begin{itemize}
        \item[(iii)$_\lambda$] $\leftindex^\lambda{\int}_{\sfc_j} h_j \le M$ for every $j$,
    \end{itemize}
    for every $\lambda \in [0,1]$.
\end{remark}

The next goal is to compare the integral of a function along a sequence of chains $\{\sfc_j\}_j$ converging to a curve $\gamma$ with the integral of the same function on $\gamma$. To this aim, we need the following approximation of Lebesgue integral by Riemann sums. A classical reference is \cite[Pag.\ 63]{Doob90}, but we adopt a strategy very close to the proof in \cite[Prop.\ 3.18]{CaputoGigliPasqualetto2021}.
\begin{proposition}
\label{prop:doob_approximation}
    Let $f \colon [0,\ell] \to \mathbb{R} \cup \{+\infty\}$ be an integrable function such that $f(0),f(\ell) < \infty$. For $n \in \N$ and $t\in [0,1]$ we set
    \begin{equation}
        \begin{aligned}
            R_t(f,n) := \frac{f(0) + f\left(\ell\left(\frac{t}{n}\right)\right)}{2} \cdot \ell\left(\frac{t}{n}\right) &+ \sum_{i=0}^{n-2}\frac{f\left(\ell\left(\frac{t+i}{n}\right)\right)+f\left(\ell\left(\frac{t+(i+1)}{n}\right)\right)}{2}\cdot\frac{\ell}{n} \\
            &+ \frac{f\left(\ell \left(\frac{t+n-1}{n}\right)\right) + f(\ell)}{2} \cdot \ell\left(\frac{1-t}{n}\right).
        \end{aligned}
    \end{equation}
    Then
    $$\lim_{n\to +\infty} \int_0^1 \left\vert R_t(f,n) - \int_0^\ell f(s) \,\d s \right\vert\,\d t = 0.$$
\end{proposition}
\begin{remark}
\label{rmk:approximation_integral_a.e.t}
    The quantity $R_t(f,n)$ should be thought as a Riemann sum associated to the partition $0 \le \ell\left(\frac{t}{n}\right) < \ldots < \ell \left( \frac{t+n-1}{n} \right) \le \ell$ of $[0,\ell]$. The difference is in the average of $f$ evaluated at two successive points of the partition instead of the value of $f$ at every point of the partition. This is due to our definition of integral along chains. 
    The statement above implies that, up to subsequence, $R_t(f,n) \to \int_0^\ell f(s)\,\d s$ for a.e.\ $t \in [0,1]$.
\end{remark}
\begin{proof}
    For an integrable function $f\colon [0,\ell] \to \R \cup \{+\infty\}$ we define the auxiliary quantity
    \begin{equation}
        R_t'(f,n) := \sum_{i=0}^{n-2}\frac{f\left(\ell\left(\frac{t+i}{n}\right)\right)+f\left(\ell\left(\frac{t+(i+1)}{n}\right)\right)}{2}\cdot\frac{\ell}{n},
    \end{equation}
    which is the middle term in the definition of $R_t(f,n)$. First of all we estimate
    \begin{equation}
        \label{eq:R_t-R_t'}
        \begin{aligned}
            \int_0^1\left\vert R_t(f,n) - R_t'(f,n)  \right\vert\,\d t &\le \frac{\ell}{n}\int_0^1\left\vert \frac{f(0) + f\left(\ell\left(\frac{t}{n}\right)\right)}{2} \right\vert\ + \left\vert \frac{f\left(\ell \left(\frac{t+n-1}{n}\right)\right) + f(\ell)}{2} \right\vert\,\d t \\
            &\le \frac{\ell}{2n}\left( f(0) + f(\ell)\right) + \frac{1}{2}\left(\int_0^{\frac{\ell}{n}} f(u)\,\d u + \int_{\ell\left(\frac{n-1}{n}\right)}^\ell f(u)\,\d u \right),
        \end{aligned}
    \end{equation}
    and the last two terms go to $0$ as $n$ goes to $\infty$, respectively because $f$ assume finite values at $0$ and $\ell$ and dominated convergence. We now set
    \begin{equation*}
    \begin{aligned}
        D(f,n):=\int_0^1 \left| R_t'(f,n) -\int_0^\ell f(s)\,\d s \right|\,\d t.
    \end{aligned}
    \end{equation*}
    If $f \in C^0([0,\ell])$ we have
    \begin{equation}
    \label{eq:hahn_computation1}
        D(f,n) \le \int_0^1 \sum_{i=0}^{n-2} \int_{\ell\left(\frac{i}{n}\right)}^{\ell\left(\frac{i+1}{n}\right)} \left| \frac{f\left(\ell\left(\frac{t+i}{n}\right)\right)-f(s)}{2}\right|+\left| \frac{f\left(\ell\left(\frac{t+i+1}{n}\right)\right)-f(s)}{2}\right|\,\d s\,\d t 
        \le \ell\varepsilon.
    \end{equation}
    for $n > n(f,\varepsilon)$, by uniform continuity of $f$ on $[0,\ell]$.
    Finally, given two Borel functions $f$ and $f'$, we estimate
    \begin{equation}
    \label{eq:hahn_computation2}
    \begin{aligned}
        |D(f,n)-D(f',n)| & \le
        \int_0^1 |R_t'(f,n)-R_t'(f',n)|\,\d t + \int_0^1 \left| \int_0^\ell (f-f')(s)\,\d s \right|\,\d t
        \\
        &\le 
        \int_0^1 \left| \sum_{i=0}^{n-2}\frac{(f-f')\left(\ell\left(\frac{t+i}{n}\right)\right)+(f-f')\left(\ell\left(\frac{t+(i+1)}{n}\right)\right)}{2}\cdot\frac{\ell}{n}\right|\,\d t+\|f-f'\|_{L^1(0,\ell)}\\
        & \le \frac{1}{2}\sum_{i=0}^{n-2}\int_0^1 |f-f'|\left( \ell\left(\frac{t+i}{n} \right)\right)\,\d t + \frac{1}{2}\sum_{i=0}^{n-2}\int_0^1 |f-f'|\left( \ell\left(\frac{t+i}{n} \right)\right)\,\d t   \\
        &\,\,\,\,\,\,+ \|f-f'\|_{L^1(0,\ell)}\\
        &\le \frac{1}{2}\sum_{i=0}^{n-2}\int_{\ell(\frac{i}{n})}^{\ell(\frac{i+1}{n})} |f-f'|\,\d t + \frac{1}{2}\sum_{i=0}^{n-2}\int_{\ell(\frac{i+1}{n})}^{\ell(\frac{i+2}{n})} |f-f'|\,\d t 
        + \|f-f'\|_{{\rm L}^1(0,\ell)}\\
        &\le 2 \|f-f'\|_{L^1(0,\ell)}.\\
    \end{aligned}
    \end{equation}
    By approximating $f \in L^1([0,\ell])$ in $L^1$-norm with a sequence $\{f_j\} \subseteq C^0([0,\ell])$, applying the triangular inequality, \eqref{eq:hahn_computation1} and \eqref{eq:hahn_computation2} we conclude that $\lims_{n \to \infty} D(f,n)=0$. This, together with \eqref{eq:R_t-R_t'}, proves the claim.
\end{proof}

\begin{remark}
\label{rmk:section_4_for_lambda}
    Given $\lambda \in [0,1]$, one can prove the equivalent of Proposition \ref{prop:doob_approximation} for the approximation
    \begin{equation}
       \begin{aligned}
           R_t^\lambda(f,n) := \left[f(0), f\left(\ell\left(\frac{t}{n}\right)\right)\right]_\lambda \cdot \ell\left(\frac{t}{n}\right) &+ \sum_{i=0}^{n-2}\left[f\left(\ell\left(\frac{t+i}{n}\right)\right),f\left(\ell\left(\frac{t+(i+1)}{n}\right)\right)\right]_\lambda\cdot\frac{\ell}{n} \\
           &+ \left[f\left(\ell \left(\frac{t+n-1}{n}\right)\right), f(\ell)\right]_\lambda \cdot \ell\left(\frac{1-t}{n}\right),
       \end{aligned}
   \end{equation}
   where $f\colon [0,\ell] \to \R \cup \{+\infty\}$ is integrable and such that $f(0),f(\ell) < \infty$. When $\lambda = 1$, this is exactly the statement of \cite[Pag.63]{Doob90}.
\end{remark}

On one side we can adapt the proof of \cite[Lemma 2.19]{Eriksson-Bique-Density-in-energy} to prove the following lemma.
\begin{lemma}
    \label{lemma:lower-semicontinuity-integral-chains-convergence}
    Let $(\X,\sfd)$ be a metric space.
    Let $g \colon \X \to [0,+\infty]$ be a lower semicontinuous function. Let $g_j \colon \X \to [0,+\infty)$ be a sequence of continuous functions such that $g_j(x) \nearrow g(x)$ for every $x\in \X$. Let $\{\sfc_j\}_j$ be a sequence of chains with $\sup_j \ell(\sfc_j) <\infty$ and converging to a curve $\gamma\colon [0,1] \to \X$. Then
    $$ \int_\gamma g \,\d s \le \limi_{j\to +\infty} \int_{\sfc_j} g_j.$$
\end{lemma}
\begin{proof}
    The proof is identical to the one of \cite[Lemma 2.19]{Eriksson-Bique-Density-in-energy}, where he uses the notion of $\lambda$-integral along chains with parameter $\lambda = 1$.
    We just need to modify it for the $\lambda = \frac12$-integral as we did for the proof of Proposition \ref{prop:doob_approximation} with respect to the original proof in \cite{Doob90}.
\end{proof}

On the other hand we have the next result.
\begin{proposition}
\label{prop:approx_integral_curve_with_subchains}
    Let $(\X,\sfd)$ be a metric space, let $\gamma\colon [0,L] \to \X$ be a curve parametrized by arc-length and let $g\colon \X \to [0,+\infty]$ be such that $g(\alpha(\gamma)), g(\omega(\gamma)) < \infty$ and $\int_\gamma g < +\infty$. For $t\in [0,1]$ and $n\in \N$ define
    \begin{equation}
    \sfc_{t,n} := \left\{\gamma(0), \gamma\left(L\left(\frac{t}{n}\right)\right), \gamma\left(L\left(\frac{t + 1}{n}\right)\right), \ldots, \gamma\left(L\left(\frac{t + n - 1}{n}\right)\right), \gamma(L) \right\} \in \Chain{\frac{L}{n}}.
\end{equation}
    Then, there exists $t\in [0,1]$ and a subsequence $n_j$ such that
    $$\int_\gamma g \ge \lims_{j\to +\infty} \int_{\sfc_{t,n_j}}g.$$
\end{proposition}

\begin{proof}
    We apply Proposition \ref{prop:doob_approximation} to the function $h=g \circ \gamma$, which is integrable by assumption and satisfies $h(0),h(L) < \infty$. In particular, there exists $t\in [0,1]$ and a subsequence $n_j$ such that
\begin{equation}
    \lim_{j\to +\infty} R_t(h,n_j) = \int_0^L h(s)\,\d s
\end{equation}
as noted in Remark \ref{rmk:approximation_integral_a.e.t}.
For every $j$ we compute
{\allowdisplaybreaks[4]
\begin{align}
\label{eq:from_time_to_points}
            R_t(h,n_j) &= \frac{h(0) + h\left(L\left(\frac{t}{n_j}\right)\right)}{2} \cdot L\left(\frac{t}{n_j}\right) + \sum_{i=0}^{n_j-2}\frac{h\left(L\left(\frac{t+i}{n_j}\right)\right)+h\left(L\left(\frac{t+(i+1)}{n_j}\right)\right)}{2}\cdot\frac{L}{n_j} \\
            &+ \frac{h\left(L \left(\frac{t+n_j-1}{n_j}\right)\right) + h(L)}{2} \cdot L\left(\frac{1-t}{n_j}\right)\\
            &= \frac{g(\gamma(0)) + g\left(\gamma\left(L\left(\frac{t}{n_j}\right)\right)\right)}{2} \cdot \ell\left(\gamma\restr{\left[0,L\left(\frac{t}{n_j}\right)\right]}\right) \\
            &+ \sum_{i=0}^{n_j-2}\frac{g\left(\gamma\left(L\left(\frac{t+i}{n_j}\right)\right)\right)+g\left(\gamma\left(L\left(\frac{t+(i+1)}{n_j}\right)\right)\right)}{2}\cdot\ell\left(\gamma\restr{\left[ L\left(\frac{t+i}{n_j}\right),L\left(\frac{t+(i+1)}{n_j}\right)\right]}\right) \\
            &+ \frac{g\left(\gamma\left(L \left(\frac{t+n_j-1}{n_j}\right)\right)\right) + g(\gamma(L))}{2} \cdot \ell\left(\gamma\restr{\left[L\left(\frac{1-t}{n_j}\right),1\right]}\right)\\
            &\ge \int_{\sfc_{t,n_j}} g\\
        \end{align}
}
where we used in the last inequality that $\sfd(\gamma(a),\gamma(b))\le \ell(\gamma\restr{[a,b]})$ for every $a,b \in [0,L]$.
\end{proof}

\begin{remark}
    Combining Lemma \ref{lemma:lower-semicontinuity-integral-chains-convergence} and Proposition \ref{prop:approx_integral_curve_with_subchains} we get that for every lower semicontinuous $g\colon \X \to [0,+\infty]$ there exists $t\in [0,1]$ and a subsequence $n_j$ such that
    $$\int_\gamma g = \lim_{j\to +\infty} \int_{\sfc_{t,n_j}}g.$$
\end{remark}

\begin{remark}
\label{rmk:section_2_for_lambda}
    The proofs of Lemma \ref{lemma:lower-semicontinuity-integral-chains-convergence} and Proposition \ref{prop:approx_integral_curve_with_subchains} can be adapted to the case of $\lambda$-integrals along chains, for $\lambda \in [0,1]$, using Remark \ref{rmk:section_4_for_lambda}.
\end{remark}

\section{Sobolev and $BV$ spaces à la Cheeger and Ambrosio-Gigli-Savaré}
\label{sec:Cheeger-AGS}
In this section we recall the definitions of two functionals that have been used by Cheeger (\cite{Cheeger99}) and Ambrosio-Gigli-Savaré (\cite{AmbrosioGigliSavare11-3, AmbrosioGigliSavare14}) to define Sobolev spaces, for $p>1$, and $BV$ spaces, for $p=1$, via relaxation.

Let $u\colon \X \to \R$ be a Borel function. A function $g\colon \X \to [0,+\infty]$ is an upper gradient of $u$, and we write $g\in \UG{}(u)$, if 
\begin{equation}
    \label{eq:defin_UG_curves}
    \vert u(\omega(\gamma)) - u(\alpha(\gamma))\vert \le \int_\gamma g
\end{equation}
for every rectifiable curve $\gamma$.

Cheeger considered the functional
\begin{equation}
    \label{eq:defin_F_Cheeger}
    \F{curve}\colon \calL{p}(\X) \to [0,+\infty], \quad u\mapsto \inf \left\{ \Vert g \Vert_{L^p(\X)}\,:\, g\in \UG{}(u) \right\},
\end{equation}
with the usual convention that the infimum over an empty set is $+\infty$. The relaxation of $\F{curve}$ is then
\begin{equation}
    \label{eq:defin_relaxation_curve}
    \relF{curve}(u) = \inf\left\{ \limi_{j\to +\infty}\inf_{g\in \UG{}(u_j)} \Vert g \Vert_{L^p(\X)}\, : \, u_j \underset{L^p(\X)}{\longrightarrow} u\right\}.
\end{equation}
We denote the associated Banach space defined as in \eqref{eq:defin_H_full_generality} by $(\HH{p}{curve}(\X), \Vert \cdot \Vert_{\HH{p}{curve}(\X)})$.

\begin{remark}
    If $p>1$ the space $(\HH{p}{curve}(\X), \Vert \cdot \Vert_{\HH{p}{curve}(\X)})$ is isometric to the $p$-Newtonian-Sobolev space, see \cite[Theorem 4.10]{Shanmu00}. Instead the space $(\HH{1}{curve}(\X), \Vert \cdot \Vert_{\HH{1}{curve}(\X)})$ can be used as a possible definition of the space of $BV$ functions (equivalent to other ones in literature when $(\X,\sfd)$ is complete by \cite{AmbrosioDiMarino14}), which generally strictly contains the $1$-Newtonian-Sobolev space.
\end{remark}

Ambrosio, Gigli and Savaré defined the functional
\begin{equation}
    \label{eq:defin_F_AGS}
    \F{AGS}\colon \calL{p}(\X) \to [0,+\infty], \quad u\mapsto 
    \begin{cases}
        \Vert \lip\,u \Vert_{L^p(\X)} &\text{ if }u\in \Lip(\X);\\
        +\infty &\text{ otherwise}
    \end{cases}
\end{equation}
The relaxation of $\F{AGS}$ is then
\begin{equation}
    \label{eq:defin_relaxation_AGS}
    \relF{AGS}(u) = \inf\left\{ \limi_{j\to +\infty} \Vert \lip\,u_j \Vert_{L^p(\X)}\, : \, u_j \in \Lip(\X) \text{ and } u_j \underset{L^p(\X)}{\longrightarrow} u\right\}.
\end{equation}
We denote the associated Banach space defined as in \eqref{eq:defin_H_full_generality} by $(\HH{p}{AGS}(\X), \Vert \cdot \Vert_{\HH{p}{AGS}(\X)})$.

It is known that if $u\in \Lip(\X)$ then $\lip\,u \in \UG{}(u)$ (see for instance \cite[Lemma 6.2.6]{HKST15}). This gives immediately that $\relF{AGS}(u) \ge \relF{curve}(u)$ for every $u\in L^p(\X)$. When the metric space $(\X,\sfd)$ is complete, the several proofs of density in energy of Lipschitz functions (see \cite{AmbrosioGigliSavare11-3,AmbrosioDiMarino14,Eriksson-Bique-Density-in-energy,LucicPasqualetto2024}) say that $\relF{AGS}(u) = \relF{curve}(u)$ for every $u\in L^p(\X)$. We summarize these well known results in the following proposition.
\begin{proposition}
\label{prop:AGS=curve_complete}
    Let $(\X,\sfd,\mm)$ be a metric measure space. Then $\HH{p}{AGS}(\X) \subseteq \HH{p}{curve}(\X)$ with $\Vert u \Vert_{\HH{p}{curve}(\X)} \le \Vert u \Vert_{\HH{p}{AGS}(\X)}$ for every $u \in L^p(\X)$. Moreover, if $(\X,\sfd)$ is complete then $\HH{p}{AGS}(\X) = \HH{p}{curve}(\X)$ with $\Vert u \Vert_{\HH{p}{curve}(\X)} = \Vert u \Vert_{\HH{p}{AGS}(\X)}$ for every $u \in L^p(\X)$.
\end{proposition}

The last part of the statement cannot hold without the completeness assumption. The motivation is the following: the space $\HH{p}{AGS}(\X)$ does not change if we take the completion of $\X$, while $\HH{p}{curve}(\X)$ is not preserved.
\begin{proposition}
    \label{prop:AGS_completion}
    Let $(\X,\sfd,\mm)$ be a metric measure space and let $(\bar{\X},\bar{\sfd},\bar{\mm})$ be its completion. Then the identity map $\iota\colon L^p(\X) \to L^p(\bar{\X})$ induces an isometry between $\HH{p}{AGS}(\X)$ and $\HH{p}{AGS}(\bar{\X})$.
\end{proposition}
\begin{proof}
    Given $u \in {\rm Lip}(\X)$, there exists a unique extension $\bar{u} \in {\rm Lip}(\bar{\X})$. For every $\varepsilon >0$ and $x \in \X$, we have
    \begin{equation*}
        \sup_{y \in (B_\varepsilon(x) \cap \X) \setminus \{x\}}\frac{|u(y)-u(x)|}{\sfd(y,x)} =\sup_{y \in B_\varepsilon(x)\setminus \{x\}}\frac{|\bar{u}(y)-\bar{u}(x)|}{\bar{\sfd}(y,x)},
    \end{equation*}
    where the balls are in $(\bar{\X},\bar{\sfd})$. By denoting with a superscript the space in which the local Lipschitz constant is computed, we have that
    \begin{equation*}
        \lip^\X u(x) =\lip^{\bar{\X}}\bar{u}(x)
    \end{equation*}
    for every $x\in \X$. This in particular implies that
    \begin{equation*}
        \|\lip^\X u\|_{L^p(\X)}= \|\lip^{\bar{\X}} \bar{u}\|_{L^p(\bar{\X})}.
    \end{equation*}
    Moreover, if a sequence of functions $u_j \in \Lip(\X)$ converges to $u$ in $L^p(\X)$ then the extensions $\bar{u}_j \in \Lip(\bar{\X})$ converge to $\iota(u)$ in $L^p(\bar{\X})$ as well, since $\bar{\mm}$ is concentrated on $\X$. 
    Thus $\iota(\HH{p}{AGS}(\X)) \subseteq \HH{p}{AGS}(\bar{\X})$ and $\|\iota(u)\|_{\HH{p}{AGS}(\bar{\X})} \le \| u \|_{\HH{p}{AGS}(\X)}$. On the other hand the operator $r \colon \HH{p}{AGS}(\bar{\X}) \to \HH{p}{AGS}(\X)$ induced by the restriction from $\bar{\X}$ to $\X$ is linear, $1$-Lipschitz and satisfies $r \circ \iota = \iota \circ r = \text{id}$, thus concluding the proof.
\end{proof}
\begin{example}
\label{ex:AGS_different_curve}
    In the simple example of $\X := \R \setminus \Q$ endowed with the Euclidean distance and the Lebesgue measure, we have $\HH{p}{AGS}(\X) \neq \HH{p}{curve}(\X)$. Indeed from one side we have, by Proposition \ref{prop:AGS_completion}, that $\HH{p}{AGS}(\X) \cong \HH{p}{AGS}(\R)$, and the latter is the classical Sobolev space on $\R$ for $p>1$ and the classical space of functions with bounded variations on $\R$ for $p=1$, while $\HH{p}{curve}(\X) \cong L^p(\X) \cong L^p(\R)$ since there are no nonconstant curves in $\X$ and so the constant function $0$ is an upper gradient of every $L^p(\X)$ function.
\end{example}
 
\begin{remark}
    Example \ref{ex:AGS_different_curve} shows that $\HH{p}{curve}(\X) \neq \HH{p}{curve}(\bar{\X})$ in general. The two spaces are the same if for instance the $p$-capacity of $\bar{\X}\setminus \X$, namely ${\rm Cap}_p(\bar{\X}\setminus \X)$, is zero. For the definition of $p$-capacity we refer to \cite[Chapter 7]{HKST15}. The following is a (non-exhaustive) list of papers studying sufficient conditions that ensure that $\HH{p}{curve}(\X) = \HH{p}{curve}(\bar{\X})$: \cite{HKST15,Koskela99,KoskelaShanmuTuominen2000,Lahti2023}. They are expressed in terms of capacity or porosity-type conditions.
\end{remark}

\section{Chain upper gradients}
\label{sec:chain_upper_gradients}
The goal of the next sections is to recover a description of the space $\HH{p}{AGS}(\X)$ in the sense of Proposition \ref{prop:AGS=curve_complete}, even when $\X$ is not complete. This is possible if we replace upper gradients with chain upper gradients.

Let $u\colon \X \to \R$ and $\varepsilon > 0$. A Borel function $g \colon \X \to [0,+\infty]$ is a $\varepsilon$-upper gradient of $u$, and we write $g\in \UG{\varepsilon}(u)$, if for all $\sfc \in \Chain{\varepsilon}$ it holds
\begin{equation}
    \label{eq:defin_epsilon_upper_gradient}
    |u(\omega(\sfc)) - u(\alpha(\sfc))| \le \int_\mathsf{c} g.
\end{equation}
The definition of $\varepsilon$-upper gradient is very sensitive to the value of the function at every point. Sometimes it is preferable to impose some regularity on the function. With this in mind we consider the class of Lipschitz $\varepsilon$-upper gradients of $u$, namely $\LUG{\varepsilon}(u) := \UG{\varepsilon}(u) \cap \Lip(\X)$.
\begin{remark}
    \label{rmk:lambda_upper_gradient}
    Let $u\colon \X \to \R$, $\varepsilon > 0$ and $\lambda \in [0,1]$. A Borel function $g \colon \X \to [0,+\infty]$ is a $(\varepsilon, \lambda)$-upper gradient of $u$, and we write $g\in \UG{\varepsilon,\lambda}(u)$, if for all $\sfc \in \Chain{\varepsilon}$ it holds
\begin{equation}
\label{eq:defin_epsilon_upper_gradient_lambda}
   u(\omega(\sfc)) - u(\alpha(\sfc)) \le \leftindex^\lambda{\int}_\mathsf{c} g.
\end{equation}
In the symmetric case, i.e. when $\lambda = \frac12$, this is equivalent to \eqref{eq:defin_epsilon_upper_gradient}. We also set $\LUG{\varepsilon,\lambda}(u) := \UG{\varepsilon,\lambda}(u) \cap \Lip(\X)$.
\end{remark}
\begin{remark}
\label{rem:integral_along_chains_of_length_2}
    The $\varepsilon$-upper gradient condition can be tested on nonconstant chains with two elements. Namely, a function $g\colon \X \to [0,+\infty]$ is a $\varepsilon$-upper gradient of $u\colon \X \to \R$ if and only if for every chain $\{x,y\}$, $x,y\in \X$, $x\neq y$ with $\sfd(x,y) \le \varepsilon$, it holds that
    $$\vert u(x) - u(y) \vert \le \frac{g(x)+g(y)}{2}\sfd(x,y) = \int_{\{x,y\}}g.$$
    One implication is obvious. For the other one we fix a chain $\sfc = \{q_i\}_{i=0}^N$ and we compute
    $$\vert u(q_N) - u(q_0) \vert \le \sum_{i=0}^{N-1} \vert u(q_i) - u(q_{i+1}) \vert \le \sum_{i=0}^{N-1} \int_{\{q_i,q_{i+1}\}} g = \int_\sfc g.$$

    A similar conclusion holds for $(\varepsilon,\lambda)$-upper gradients, for every $\lambda \in [0,1]$.
\end{remark}
The next lemma shows that the slope at level $\varepsilon$ is always a $\varepsilon$-upper gradient. On the other hand, the local Lipschitz constant is smaller than every upper semicontinuous (in particular every Lipschitz) $\varepsilon$-upper gradient.
\begin{lemma}
\label{lemma:slopeeps_epsug}
    Let $(\X,\sfd)$ be a metric space and let $u\colon \X \to \R$. Then $\sl{\varepsilon} u \in \UG{\varepsilon}(u)$ for every $\varepsilon > 0$. Moreover, for every $g\in \UG{\varepsilon}(u)$ it holds $\sl{\varepsilon'} u(x) \leq \sup_{B_{\varepsilon'}(x)} g$ for every $\varepsilon' \le \varepsilon$. Finally, if $g\in \UG{\varepsilon}(u)$ is upper semicontinuous then $\lip\, u \le g$.
\end{lemma}
\begin{proof}
    By Remark \ref{rem:integral_along_chains_of_length_2} it suffices to consider $\sfc = \{q_0,q_1\}$. Then
    \begin{equation}
    \begin{aligned}
         |u(q_{1}) - u(q_{0})| &= \frac{1}{2}\left( \frac{\vert u(q_{1}) - u(q_{0}) \vert}{\sfd(q_0,q_{1})} + \frac{\vert u(q_{1}) - u(q_{0}) \vert}{\sfd(q_0,q_{1})} \right)\sfd(q_0,q_{1})\\
        &\leq \frac{ \sl{\varepsilon} u(q_0) + \sl{\varepsilon} u(q_{1})}{2}\sfd(q_0,q_{1}) = \int_{\mathsf{c}} \sl{\varepsilon} u.
    \end{aligned}
    \end{equation}
    This proves the first part of the statement. 
    
    We move to the second part. Fix $x \in \X$ and consider $y\in \X$ such that $\sfd(x,y) \le \varepsilon'$ with $\varepsilon' \le \varepsilon$. Since $g\in \UG{\varepsilon}(u)$ and $\{x,y\} \in \Chain{\varepsilon}$, we have
    \begin{equation*}
        |u(x)-u(y)| \le \sfd(x,y)\,\frac{g(x)+g(y)}{2} \le \sfd(x,y)\,\sup_{B_\varepsilon(x)}g.
    \end{equation*}
    By taking the supremum over $y \in B_{\varepsilon'}(x)$ the second conclusion follows. Taking the limit as $\varepsilon' \to 0$ and using the uppersemicontinuity of $g$, we conclude also the third part.
\end{proof}
On the other hand, every $\varepsilon$-upper gradient $g$ satisfies the upper gradient inequality along every curve with endpoints in the set $\{g < \infty\}$. 
%
%
\begin{proposition}
\label{prop:UGeta_subset_UG}
    Let $(\X,\sfd)$ be a metric space and let $u\colon \X \to \R$. If $g\in \UG{\varepsilon}(u)$ for some $\varepsilon >0$, then 
    \begin{equation}
    \label{eq:UG_inequality}
        \vert u(\omega(\gamma)) - u(\alpha(\gamma)) \vert \le \int_\gamma g
    \end{equation}
    for every rectifiable curve $\gamma$ such that $\omega(\gamma), \alpha(\gamma) \in \{g < +\infty \}$. In particular if $g$ takes only finite values then $g \in \UG{}(u)$.
\end{proposition}
\begin{remark}
The inequality \eqref{eq:UG_inequality} does not hold in general for curves whose endpoints belong to the set $\{g = +\infty\}$.
Indeed let us consider $u=\chi_\Q$, i.e.\ the characteristic function of the set of rational numbers in $\R$. The function $g\colon \R \to [0,+\infty]$, $g(x) = +\infty$ if $x\in \Q$ and $g(x) = 0$ otherwise, belongs to $\UG{\varepsilon}(u)$ for every $\varepsilon > 0$. However it does not belong to $\UG{}(u)$ since $\int_\gamma g = 0$ for every rectifiable curve $\gamma$ of $\R$.
\end{remark}
\begin{proof}[Proof of Proposition \ref{prop:UGeta_subset_UG}] 

Assume by contradiction that there exists a rectifiable curve, that we can assume parametrized by arc length $\gamma\colon [0,L] \to \X$, with $L:=\ell(\gamma)$, such that 
\begin{equation}
\label{eq:g-is-integrable}
    \int_0^L (g\circ \gamma)(s)\,\d s = \int_\gamma g < |u(\omega(\gamma))-u(\alpha(\gamma))| < \infty
\end{equation}
and $g(\alpha(\gamma)), g(\omega(\gamma)) < \infty$.
By Proposition \ref{prop:approx_integral_curve_with_subchains} we can find a subsequence $n_j$ such that 
$$\int_\gamma g \ge \limi_{j\to +\infty} \int_{\sfc_{t,n_j}} g.$$
This, together with \eqref{eq:g-is-integrable}, implies the existence of a $\frac{L}{n_j}$-chain $\sfc_{t,n_j}$ with same endpoints of $\gamma$ such that 
$$\int_{\sfc_{t,n_j}} g < |u(\omega(\gamma))-u(\alpha(\gamma))| = \vert u(\alpha(\sfc_{t,n_j})) - u(\omega(\sfc_{t,n_j}))\vert.$$
This proves that $g \notin \UG{\frac{L}{n_j}}(u)$, which is a contradiction.
\end{proof}

\begin{remark}
\label{rmk:section_4_for_lambda_true}
    The results of this section remain true if we consider the $\lambda$-integral and $(\varepsilon,\lambda)$-upper gradients, for every $\lambda \in [0,1]$. The proof of Proposition \ref{prop:UGeta_subset_UG} follows by Remark \ref{rmk:section_2_for_lambda}.
\end{remark}

\section{$p$-weak $\varepsilon$-upper gradients}

\label{sec:weak_eps_upper_gradient}
In the classical theory of Sobolev spaces, one weakens the definition of upper gradients along curves by requiring that \eqref{eq:defin_UG_curves} holds for $\Mod_p$-almost every curve. The definition of the outer measure $\Mod_p$ will be recalled in Section \ref{sec:PI}. In this section we will give a similar definition for chain upper gradients and we will show similarities and differences with the classical setting of curves.

Let $(\X,\sfd,\mm)$ be a metric measure space.
Let $\varepsilon > 0$ and $p\ge 1$. The $(\varepsilon, p)$-modulus of a family of chains $\sfC \subseteq \Chain{}$ is
\begin{equation*}
    \cModepsilon{\varepsilon}{p}({\sf C}):=\inf \left\{ \int \rho^p\,\d\mm:\, \rho \in {\rm Adm}^{\varepsilon}({\sfC})\right\}
\end{equation*}
where ${\rm Adm}^{\varepsilon}({\sfC})=\left\{ \rho \ge 0\,:\, \rho \text{ Borel, }\int_{\sfc} \rho \ge 1 \text{ for every } \sfc \in {\sfC} \cap \Chain{\varepsilon} \right\}$.
\begin{proposition}
\label{prop:outer_measure}
    Let $\varepsilon > 0$ and $p\ge 1$. Then $\cModepsilon{\varepsilon}{p}$ is an outer measure.
\end{proposition}
\begin{proof}
    The only non trivial property to be proven is $\cModepsilon{\varepsilon}{p}(\bigcup_{i=1}^\infty \sfC_i)\le \sum_{i=1}^\infty \cModepsilon{\varepsilon}{p}(\sfC_i)$ for $\{\sfC_i\}_{i=1}^\infty \subseteq \Chain{}$. We assume that the right hand side is finite, otherwise there is nothing to prove. 
    We fix $\delta>0$ and we choose $\{\eta_i\}_i$ such that $\sum_{i}\eta_i \le \delta$.
    We take $\rho_i \in {\rm Adm}^{\varepsilon}(\sfC_i)$ such that 
    $\int \rho_i^p \,\d \mm \le \cModepsilon{\varepsilon}{p}(\sfC_i) + \eta_i$.
    The function $\rho:=(\sum_{i=1}^\infty (\rho_i)^p)^{\frac{1}{p}}$ satisfies $\rho \in {\rm Adm}^{\varepsilon}(\bigcup_{i=1}^\infty \sfC_i)$. Therefore
    \begin{equation*}
        \cModepsilon{\varepsilon}{p}\left(\bigcup_{i=1}^\infty \sfC_i\right)\le \int \rho^p\,\d \mm \le \sum_{i=1}^\infty \int \rho_i^p\,\d \mm \le \sum_{i=1}^\infty \cModepsilon{\varepsilon}{p}(\sfC_i) +\delta.
    \end{equation*}
    Taking the limit as $\delta$ converges to $0$, we get the thesis.
\end{proof}
\begin{remark}
    The $(\varepsilon, p)$-modulus is concentrated on $\varepsilon$-chains in the following sense.
    Let $\sfC \subseteq \Chain{}$ and let $\sfC^\varepsilon := \sfC \cap \Chain{\varepsilon}$. Then $\cModepsilon{\varepsilon}{p}(\sfC \setminus \sfC^\varepsilon) = 0$ since the function $\rho=0$ belongs to ${\rm Adm}^{\varepsilon}({\sfC}\setminus {\sfC}^\varepsilon)$.
\end{remark}
    Let $(\X,\sfd,\mm)$ be a metric measure space. Given a function $u \colon \X \to \mathbb{R}$, we say that a Borel function $g \colon \X \to [0,+\infty]$ is a $p$-weak $\varepsilon$-upper gradient of $u$ and we write $g \in \WUG{p}{\varepsilon}(u)$ if
    \begin{equation*}
        \vert u(\omega(\sfc))-u(\alpha(\sfc)) \vert \le \int_{\sfc} g \qquad \text{ for }\cModepsilon{\varepsilon}{p}\text{-a.e. chain.}
    \end{equation*}
In particular if $g \in \UG{\varepsilon}(u)$ then $g\in \WUG{p}{\varepsilon}(u)$ for every $p\ge 1$.
\begin{remark}
    For $\lambda \in [0,1]$ one can define the $(\varepsilon, \lambda, p)$-modulus of a family of chains $\sfC \subseteq \Chain{}$ by
\begin{equation*}
    \cModepsilon{\varepsilon,\lambda}{p}({\sf C}):=\inf \left\{ \int \rho^p\,\d\mm:\, \rho \in {\rm Adm}^{\varepsilon, \lambda}({\sfC})\right\}
\end{equation*}
where ${\rm Adm}^{\varepsilon, \lambda}({\sfC})=\left\{ \rho \ge 0\,:\, \rho \text{ Borel, }\leftindex^\lambda{\int}_{\sfc} \rho \ge 1 \text{ for every } \sfc \in {\sfC} \cap \Chain{\varepsilon} \right\}$. This is still an outer measure which is concentrated on $\varepsilon$-chains. A Borel function $g \colon \X \to [0,+\infty]$ is a $p$-weak $(\varepsilon, \lambda)$-upper gradient of $u\colon \X \to \R$ if
    \begin{equation*}
        u(\omega(\sfc))-u(\alpha(\sfc)) \le \leftindex^\lambda{\int}_{\sfc} g \qquad \text{ for }\cModepsilon{\varepsilon,\lambda}{p}\text{-a.e. chain.}
    \end{equation*}
The set of $p$-weak $(\varepsilon, \lambda)$-upper gradient of $u$ is denoted by $\WUG{p}{\varepsilon, \lambda}(u)$. It holds $\UG{\varepsilon, \lambda}(u) \subseteq \WUG{p}{\varepsilon, \lambda}(u)$.
\end{remark}
The set of $p$-integrable $p$-weak $\varepsilon$-upper gradients is closed under $L^p(\X)$-convergence. This can be seen as a consequence of an appropriate version of  Fuglede's Lemma in this context.
\begin{proposition}[Fuglede's lemma for chains]
\label{prop:Fuglede's-lemma}
    Let $g_j$ be a sequence of Borel functions that converges in $L^p(\X)$. Then there is a subsequence $g_{j_k}$ with the following property: if $g$ is any Borel representative of the $L^p(\X)$-limit of $g_j$ then
    $$\lim_{k\to +\infty} \int_{\sfc} \vert g_{j_k} - g\vert = 0$$
    for $\cModepsilon{\varepsilon}{p}$-a.e. chain.
\end{proposition}

In the chain case this result is easier to prove and it is a consequence of the following easy but important fact. 
\begin{lemma}
\label{lemma:measure_zero_implies_modulezero}
    Let $E\subseteq \X$, $\varepsilon > 0$, $p\ge 1$. If $\mm(E) = 0$ then $\cModepsilon{\varepsilon}{p}(\Chain{}(E)) = 0$.
\end{lemma}
\begin{proof}
    We write
    $\Chain{}(E) = \bigcup_{k\in \N} \sfC_k$, where
    $$\sfC_k := \left\{ \sfc = \{q_i\}_{i=0}^N \in \Chain{}(E) \,:\, \min_{0\le i < N} \sfd(q_i,q_{i+1}) \geq \frac{1}{k} \right\}.$$
    By Proposition \ref{prop:outer_measure}, $\cModepsilon{\varepsilon}{p}$ is an outer measure. So it is enough to prove that $\cModepsilon{\varepsilon}{p}(\sfC_k) = 0$ for every $k$. The function $\rho = 2k\cdot \chi_E$ belongs to Adm$^{\varepsilon}(\sfC_k)$. Therefore $\cModepsilon{\varepsilon}{p}(\sfC_k) \le \int \rho^p\,\d\mm = (2k)^p\mm(E)=0$.
\end{proof}


\begin{remark}
    The previous lemma differs to the classical case in which the modulus is defined in terms of rectifiable curves. Given a Borel set $E$ with $\mm(E)=0$, ${\rm Mod}_p(\Gamma(E))=0$ if and only if ${\rm Cap}_p(E)=0$, see \cite[Prop.\ 7.2.8]{HKST15}, where $\Gamma(E)$ denotes the family of curves intersecting $E$. In other words, given a $\mm$-null set with positive capacity, the $p$-modulus of the curves hitting this set is positive. This does not happen in the case of chains, as Lemma \ref{lemma:measure_zero_implies_modulezero} shows. 
    On the other hand, the same proof of Lemma \ref{lemma:measure_zero_implies_modulezero} shows in the case of curves that the $p$-modulus of the set of curves spending a positive time in $E$ is zero, if $\mm(E) = 0$. In this case ${\rm Mod}_p(\Gamma(E))$ is concentrated on the family of curves spending time zero in $E$ (see also \cite[Lemma 5.2.15]{HKST15}). 
\end{remark}

\begin{remark}
\label{rmk:WUG_up_to_measure_zero}
    Lemma \ref{lemma:measure_zero_implies_modulezero} implies that if $g \in \WUG{p}{\varepsilon}(u)$ and $h$ is a function such that $\mm(\{g \neq h\})=0$, then $h \in \WUG{p}{\varepsilon}(u)$ as well.
\end{remark}

\begin{proof}[Proof of Proposition \ref{prop:Fuglede's-lemma}]
Since $g_j$ converges to $g$ in $L^p(\X)$ then we can find a subsequence $g_{j_k}$ which converges to $g$ pointwise almost everywhere. This means that we can find a set $E \subseteq \X$ with $\mm(E) = 0$ such that $\lim_{k\to +\infty} g_{j_k}(x) = g(x)$ for every $x\in \X \setminus E$. Let us consider the set of chains $\Chain{}(E)$ which has $(\varepsilon, p)$-modulus $0$ by Lemma \ref{lemma:measure_zero_implies_modulezero}. We claim that for every chain which is not in $\Chain{}(E)$ we have $\lim_{k\to +\infty} \int_\sfc \vert g_{j_k} - g \vert = 0$. 
Let us fix a chain $\sfc = \{q_i\}_{i=0}^N$ which is not in $\Chain{}(E)$. 
Then $q_i \notin E$ for every $0\le i \le N$. In particular $\lim_{k\to +\infty}g_{j_k}(q_i) = g(q_i)$ for $0\le i \le N$. Therefore
$$\lim_{k\to +\infty} \int_\sfc \vert g_{j_k} - g \vert = \lim_{k\to +\infty}  \sum_{i=0}^{N-1} \frac{\vert g_{j_k}(q_i) - g(q_i) \vert + \vert g_{j_k}(q_{i+1}) - g(q_{i+1})\vert}{2}\sfd(q_i,q_{i+1}) = 0.$$
\end{proof}
\begin{remark}
    \label{rmk:chain-modulus-of-measure-zero-sets}
    Lemma \ref{lemma:measure_zero_implies_modulezero} is true for $\cModepsilon{\varepsilon,\lambda}{p}$, for every $\lambda\in(0,1)$, with the same proof. An alternative argument is to observe that $\cModepsilon{\varepsilon,\lambda}{p}$ and $\cModepsilon{\varepsilon,\lambda'}{p}$ are mutually absolutely continuous if $\lambda,\lambda'\in (0,1)$. 
    However, if $\lambda \in \{0,1\}$, Lemma \ref{lemma:measure_zero_implies_modulezero} is no more true. Indeed, if $\Q$ is the set of rational numbers in $\R$, then $\cModepsilon{\varepsilon,\lambda}{p}(\Chain{}(\Q)) = +\infty$ for $\lambda \in \{0,1\}$. This follows from the fact that every admissible function has to be equal to $+\infty$ on $\R\setminus \Q$. On the other hand, the same proof as above shows that given $E\subseteq \X$ such that $\mm(E)=0$, then
    $$\cModepsilon{\varepsilon,1}{p}(\{ \sfc = \{q_i\}_{i=0}^N \in \Chain{}\, : \, q_i \in E \text{ for some } i=0, \ldots, N-1\}) = 0$$
    and
    $$\cModepsilon{\varepsilon,0}{p}(\{ \sfc = \{q_i\}_{i=0}^N \in \Chain{}\, : \, q_i \in E \text{ for some } i=1, \ldots, N\}) = 0.$$
    This is enough for adapting the proof of Proposition \ref{prop:Fuglede's-lemma} to $\lambda \in \{0,1\}$.
\end{remark}
As a consequence we prove that the set $\WUG{p}{\varepsilon}(u) \cap L^p(\X)$ is closed in $L^p(\X)$. Actually a stronger statement, that is the chain counterpart of \cite[Proposition 6.3.30]{HKST15}, holds. Notice that in \cite[Proposition 6.3.30]{HKST15} is required the convergence of $u_j$ to $u$ ${\rm Cap}_p$-a.e.\, while in our result it is enough to consider convergence $\mm$-a.e.
\begin{proposition}
    \label{prop:limit_weak_upper_gradients}
    Let $u_j \to u$ pointwise $\mm$-almost everywhere, let $g_j \in \WUG{p}{\varepsilon}(u_j)$ and suppose $g_j \to g$ in $L^p(\X)$. Then $g \in \WUG{p}{\varepsilon}(u)$.
\end{proposition}
\begin{proof}
    Let $E$ be the set of points of $\X$ where the convergence of $u_j$ to $u$ does not hold. By Proposition \ref{prop:Fuglede's-lemma} we can extract a further subsequence, not relabeled, and a set of chains $\sfC$ such that $\lim_{j\to +\infty} \int_\sfc g_j = \int_\sfc g$ for every $\sfc \in \sfC$ and with $\cModepsilon{\varepsilon}{p}(\Chain{}\setminus \sfC) = 0$. Since $g_j \in \WUG{p}{\varepsilon}(u_j)$ we can find set of chains ${\sfC}_j$ such that $|u_j(\omega(\sfc)) - u_j(\alpha(\sfc))| \leq \int_\sfc g_j$ for every $\sfc \in {\sfC}_j$ and such that $\cModepsilon{\varepsilon}{p}(\Chain{} \setminus {\sfC}_j) = 0$. The set of chains ${\sfC}' = \left({\sfC} \cap \bigcap_{j \in \N} {\sfC}_j\right) \setminus \Chain{}(E)$ still satisfies $\cModepsilon{\varepsilon}{p}(\Chain{} \setminus {\sf C}') = 0$, because of Proposition \ref{prop:outer_measure}, since $\Chain{}\setminus {\sfC}' = \Chain{}(E) \cup (\Chain{} \setminus {\sfC}) \cup \bigcup_j (\Chain{} \setminus {\sfC}_j)$.\\
    For every $\sfc \in {\sfC}'$ we have 
    $$|u(\omega(\sfc)) - u(\alpha(\sfc))| = \lim_{j\to +\infty} |u_j(\omega(\sfc)) - u_j(\alpha(\sfc))| \le \lim_{j\to +\infty} \int_\sfc g_j = \int_\sfc g.$$
    This shows that $g\in \WUG{p}{\varepsilon}(u)$.
\end{proof}
\begin{remark}
    Proposition \ref{prop:limit_weak_upper_gradients} remains true for every $\lambda \in (0,1)$, while it is not clear if it holds for $\lambda \in \{0,1\}$. However it is still true, and the proof is the same, that $\WUG{p}{\varepsilon,1}(u)$ and $\WUG{p}{\varepsilon,0}(u)$ are closed with respect to the $L^p(\X)$-topology.
\end{remark}

As a consequence we can find a $p$-weak $\varepsilon$-upper gradient of minimal norm.
\begin{proposition}
\label{prop:minimal-weak-upper-gradient}
    The set $\WUG{p}{\varepsilon}(u) \cap L^p(\X)$ is a closed, convex subset of $L^p(\X)$. If not empty, it contains an element of minimal $L^p(\X)$-norm. If $p>1$ such element is unique.
\end{proposition}

\begin{proof}
    We already showed in Proposition \ref{prop:limit_weak_upper_gradients} that $\WUG{p}{\varepsilon}(u)$ is closed, while its convexity is trivial. The existence of an element of minimal norm, i.e. the existence of a projection of $0 \in L^p(\X)$ on $\WUG{p}{\varepsilon}(u)$, follows directly. The uniqueness statement for $p>1$ is a consequence of the strict convexity of the norm of $L^p(\X)$ for $p>1$.
\end{proof}
On the other hand the minimal norm can be computed also using true $\varepsilon$-upper gradients because of the next result.
\begin{proposition}
    \label{prop:WUG=closure-of-UG}
    The set $\WUG{p}{\varepsilon}(u) \cap L^p(\X)$ is the $L^p(\X)$-closure of $\UG{\varepsilon}(u)$.
\end{proposition}
\begin{proof}
    By Proposition \ref{prop:limit_weak_upper_gradients} we know that $\WUG{p}{\varepsilon}(u) \cap L^p(\X)$ is closed in $L^p(\X)$ and therefore it contains the $L^p(\X)$-closure of $\UG{\varepsilon}(u)$. Let $g\in \WUG{p}{\varepsilon}(u)$. Let $\sfC$ be a family of chains such that $\cModepsilon{\varepsilon}{p}(\Chain{}\setminus \sfC) = 0$ and such that $|u(\omega(\sfc)) - u(\alpha(\sfc))| \le \int_\sfc g$ for every $\sfc \in \sfC$. By definition, for every $j\ge 1$ there exists an admissible map $\rho_j \in \text{Adm}^\varepsilon(\Chain{}\setminus \sfC)$ such that $\int \rho_j^p\,\d\mm \le 2^{-jp}$. We set $\rho = \left(\sum_{j\ge 1} \rho_j^p\right)^{\frac{1}{p}} \in \text{Adm}^\varepsilon(\Chain{}\setminus \sf C)$, because $\rho \ge \rho_j$ for every $j$. Moreover $\rho \in L^p(\X)$ and $\int_\sfc\rho = \infty$ for every $\sfc \in \Chain{}\setminus \sfC$. Now, for every $k\in \N$, we define the function $g_k := g + 2^{-k}\rho$. It is easy to check that $g_k$ converges to $g$ in $L^p(\X)$ and that $g_k \in \UG{\varepsilon}(u)$.
\end{proof}
\begin{example}[{Non uniqueness of minimal weak chain upper gradient}]
\label{rmk:mWUG_not_unique}
    If $p=1$ it can happen that there is more than one element of minimal norm in $\WUG{1}{\varepsilon}(u)$. We now produce an example of a metric measure space $(\X,\sfd,\mm)$ and a function $u\colon \X \to \R$ such that for every $0<\varepsilon \le \frac{1}{3}$ there are infinitely many elements of minimal norm in $\WUG{1}{\varepsilon}(u)$. 

    We define the following two sequences of real numbers: $x_n = n$, $y_n = n + \frac{1}{n}$, for $n\ge 3$.
    Let $\X$ be the countable set $\X := \bigcup_{n \ge 3} \{ x_n, y_n\}\subset \R$. We endow $\X$ with the Euclidean distance and with the reference measure $\mm := \sum_{n \ge 3} \frac{1}{n^3} \left( \delta_{x_n}+ \delta_{y_n} \right)$. Notice that $\X$ is complete. We define the function $u \colon \X \to \R$ as $u(x_n)=1$ and $u(y_n)=0$ for every $n \ge 3$. We fix $0<\varepsilon\le \frac{1}{3}$ and we notice that all the possible nonconstant $\varepsilon$-chains with two elements are of the form $\{x_n,y_n\}$ and $\{y_n,x_n\}$, for $n\ge \varepsilon^{-1}$. Therefore, by Remark \ref{rem:integral_along_chains_of_length_2}, a function $g\colon \X \to [0,+\infty]$ is a $\varepsilon$-upper gradient if and only if
    \begin{equation}
        \label{eq:example_UG_condition}
        g(x_n) + g(y_n) \ge 2n \quad \forall n \ge \varepsilon^{-1}.
    \end{equation}
    Hence, its $L^1(\mm)$ norm satisfies the following lower bound
    \begin{equation*}
        \|g\|_{L^1(\mm)} = \sum_{n \ge 3} \frac{g(x_n)+g(y_n)}{n^3} \ge 2\sum_{n \ge \varepsilon^{-1}} \frac{1}{n^2} =:L_\varepsilon.
    \end{equation*}
    For every $\mu \in [0,1]$ the function $g_{\mu,\varepsilon}$ defined as
    \begin{equation*}
        g_{\mu,\varepsilon}(x) := \begin{cases}
            2n\mu & x=x_n \text{ for }n \ge \varepsilon^{-1}\\
            2n(1-\mu) & x=y_n \text{ for }n \ge \varepsilon^{-1}\\
            0 & \text{otherwise}
        \end{cases}
    \end{equation*}
    we have $\|g_{\mu,\varepsilon}\|_{L^1(\mm)}=L_\varepsilon$, therefore they are all $\varepsilon$-chain upper gradients of minimal $L^1$-norm, by Proposition \ref{prop:WUG=closure-of-UG}.
\end{example}

\begin{remark}
    Every result of this section extends verbatim to the case $\lambda \in (0,1)$, while some differences appear in case $\lambda \in \{0,1\}$. For simplicity we state the results for $\lambda = 1$, the case $\lambda = 0$ being analogous. We have already noticed that we do not know if Proposition \ref{prop:limit_weak_upper_gradients} holds for $\lambda = 1$, but in any case $\WUG{p}{\varepsilon,1}(u)$ is closed. Moreover it is possible to show, adapting verbatim the proofs of \cite[Lemma 6.3.8]{HKST15}, that $\WUG{p}{\varepsilon,1}(u)$ is a lattice. As a consequence there exists a minimal $p$-weak $(\varepsilon,1)$-upper gradient $g_u$ of $u$ in the following sense: if $g\in \WUG{p}{\varepsilon, 1}(u)$ then $g_u \le g$ $\mm$-a.e. For the function $u$ in Example \ref{rmk:mWUG_not_unique} and $p=1$, then $g_u$ is equal to
    \begin{equation*}
        g_{u}(x) := \begin{cases}
            n & x=x_n \text{ for }n \ge \varepsilon^{-1}\\
            0 & \text{otherwise.}
        \end{cases}
    \end{equation*}
    The unique element of minimal norm in $\WUG{1}{\varepsilon, 0}(u)$ is instead
    \begin{equation*}
        g_u(x) := \begin{cases}
            n & x=y_n \text{ for }n \ge \varepsilon^{-1}\\
            0 & \text{otherwise.}
        \end{cases}
    \end{equation*}
\end{remark}

\section{The chain Sobolev spaces}
\label{sec:chain_sob_space}
In this section we introduce two new functionals and we study their relaxations. The two functionals are
\begin{equation}
    \begin{aligned}
        \F{\Chain{}} \colon \calL{p}(\X) \to  [0,+\infty], \quad &u \mapsto \lim_{\varepsilon \to 0} \inf\left\{ \| g \|_{L^p(\X)}\,:\, g\in \UG{\varepsilon}(u)\right\},\\
        \F{\Chain{},\, \Lip} \colon \calL{p}(\X) \to  [0,+\infty], \quad &u \mapsto \begin{cases}
            \lim_{\varepsilon \to 0} \inf\left\{ \| g \|_{L^p(\X)}\,:\, g\in \LUG{\varepsilon}(u)\right\} &\text{if } u\in \Lip(\X),\\
            +\infty &\text{otherwise},
        \end{cases}
    \end{aligned}
\end{equation}
where the infimum of an empty set is $+\infty$.
By Proposition \ref{prop:WUG=closure-of-UG}, $\F{\Chain{}}$ can be equivalently defined by $\F{\Chain{}}(u) = \lim_{\varepsilon \to 0} \inf\left\{ \| g \|_{L^p(\X)}\,:\, g\in \WUG{p}{\varepsilon}(u)\right\}$. The limits in the definitions exist because the arguments are decreasing functions. Indeed if $\varepsilon' \le \varepsilon$ then every $\varepsilon$-upper gradient is also a $\varepsilon'$-upper gradient. The two functionals satisfy properties (a),(b) and (c) of Section \ref{subsec:relaxation}. The less trivial property, which is (c), is consequence of the symmetric property in \eqref{eq:integral_is_symmetric} that implies that $\UG{\varepsilon}(u) = \UG{\varepsilon}(-u)$ for every Borel function $u\colon \X \to \R$.

The relaxations of the functionals above are denoted respectively by $\relF{\Chain{}}$ and $\relF{\Chain{}, \,\Lip}$. The associated Banach spaces are respectively $(\HH{p}{\Chain{}}(\X), \Vert \cdot \Vert_{\HH{p}{\Chain{}}(\X)})$ and $(\HH{p}{\Chain{}, \, \Lip}(\X), \Vert \cdot \Vert_{\HH{p}{\Chain{}, \, \Lip}(\X)})$. Since $\F{\Chain{}, \, \Lip}(u) \ge \F{\Chain{}}(u)$ for every $u\in \calL{p}(\X)$, we have that 
$$\HH{p}{\Chain{},\,\Lip}(\X) \subseteq \HH{p}{\Chain{}}(\X),$$
and the inclusion is $1$-Lipschitz.
\begin{remark}
    The domain of $\F{\Chain{}}$, namely the set of functions in $\calL{p}(\X)$ that admits an $L^p(\X)$-integrable $\varepsilon$-upper gradient for some $\varepsilon > 0$, is clearly larger than the domain of $\F{\Chain{},\,\Lip}$, but it is also bigger than the set of functions (not necessarily Lipschitz) that admit $p$-integrable, Lipschitz, $\varepsilon$-upper gradients for some $\varepsilon > 0$. Indeed, it contains functions that are highly non-regular, as the next example shows. Let $u = \chi_\Q$ be the characteristic function of the rational numbers on $\R$. It can be readily checked that $\LUG{\varepsilon}(u) = \emptyset$, thus it does not admit a $p$-integrable, Lipschitz, $\varepsilon$-upper gradient for every $\varepsilon > 0$, while $g\colon \X \to [0,+\infty]$, $g = \infty \cdot \chi_\Q$ belongs to $\UG{\varepsilon}(u)$. Since $\Vert g \Vert_{L^p(\R)} = 0$ we have that $\F{\Chain{}}(u) = 0$.
\end{remark}

In order to familiarize with the definition of $\F{\Chain{}}$ and $\HH{p}{\Chain{}}$ we compute explicitly this space in the case of a snowflake of a metric measure space. In view of Theorem \ref{theo-intro:equivalences-of-all-norms-complete-spaces} and the well known fact that there are no rectifiable curves in such spaces, hence $\HH{p}{curve}(\X) = L^p(\X)$, we know that it must hold that $\HH{p}{\Chain{}}(\X) = L^p(\X)$ too.
\begin{example}[{Snowflaking of a metric space $(\X,\sfd)$}]
\label{rem:snowflake}
    Let $(\X,\sfd)$ be a metric space and $0<\alpha<1$. We consider $(\X,\sfd^\alpha)$, where $\sfd^\alpha(x,y):=(\sfd(x,y))^\alpha$. Let $\mm$ be a Borel measure on $(\X,\sfd)$ (so a Borel measure on $(\X,\sfd^\alpha)$ too)
    and consider $\varepsilon >0$. Notice that $\sfc \in \Chain{\varepsilon}(\X,\sfd)$ if and only if $\sfc \in \Chain{\varepsilon^\alpha}(\X,\sfd^\alpha)$. 
    We claim that, if $g$ is a $\varepsilon$-upper gradient of $u$ on $(\X,\sfd)$, then $\varepsilon^{1-\alpha} g$ is a $\varepsilon^\alpha$-upper gradient of $u$ on $(\X,\sfd^\alpha)$.
    Indeed, for a chain $\sfc=\{ q_i\}_{i=0}^N$ such that $\sfd(q_i,q_{i+1})\le \varepsilon$ we have 
    \begin{align}
        |u(\omega(\sfc))-u(\alpha(\sfc))| & \le \sum_i \frac{g(q_i) + g(q_{i+1})}{2}\,\sfd(q_i,q_{i+1}) \\
        & = \sum_i \frac{g(q_i) + g(q_{i+1})}{2}\,\sfd^\alpha(q_i,q_{i+1})\,\sfd^{1-\alpha}(q_i,q_{i+1}) \\
        &\le \sum_i \frac{\varepsilon^{1-\alpha}g(q_i) + \varepsilon^{1-\alpha}g(q_{i+1})}{2}\,\sfd^\alpha(q_i,q_{i+1}).
    \end{align}
    Therefore, for every $u\in L^p(\X)$, we have
    \begin{equation}
        \begin{aligned}
            \F{\Chain{}}^{\X,\sfd^\alpha,\mm}(u) &= \lim_{\varepsilon \to 0} \inf\left\{ \Vert g \Vert_{L^p(\X)}\,:\, g \text{ is a } \varepsilon^\alpha\text{-upper gradient of } u  \text{ in } (\X,\sfd^\alpha)  \right\} \\
            &\le \lim_{\varepsilon \to 0} \varepsilon^{1-\alpha}\inf\left\{ \Vert g \Vert_{L^p(\X)}\,:\, g \text{ is a } \varepsilon\text{-upper gradient of } u  \text{ in } (\X,\sfd)  \right\}.
        \end{aligned}
    \end{equation}
    In particular, since every function $u\in \Lip(\X,\sfd)$ with bounded support has a chain upper gradient, namely $\sl{\varepsilon}(u)$, which is in $L^p(\X)$, then $\F{\Chain{}}^{\X,\sfd^\alpha,\mm}(u)= 0$ for all such functions. Therefore $\relF{\Chain{}}^{\X,\sfd^\alpha,\mm}(u) = 0$ for every $u\in L^p(\X)$ since the class of Lipschitz functions (w.r.t. $\sfd$) with bounded support is dense in $L^p(\X)$. Thus $\HH{p}{\Chain{}}(\X,\sfd^\alpha,\mm) = L^p(\X)$.
\end{example}

\subsection{Proof of the main results}
\label{sec:proof_main_result}
The goal of this section is to compare the space $\HH{p}{\Chain{}}(\X)$ with $\HH{p}{curve}(\X)$ and $\HH{p}{AGS}(\X)$, with the help of $\HH{p}{\Chain{},\,\Lip}(\X)$.
\begin{proposition}
\label{prop:equivalence_H_noncomplete}
    Let $(\X,\sfd,\mm)$ be a metric measure space. Then 
    \begin{equation}
        \HH{p}{\Chain{},\, \Lip}(\X) \subseteq \HH{p}{curve}(\X)
    \end{equation}
    and
    \begin{equation}
        \|u \|_{\HH{p}{curve}(\X)} \le \|u \|_{\HH{p}{\Chain{}, \Lip}(\X)}
    \end{equation}
    for every $u\in L^p(\X)$.
\end{proposition}
\begin{proof}
   Proposition \ref{prop:UGeta_subset_UG} implies that $\F{curve}(u) \le \F{\Chain{},\,\Lip}(u)$ for every $u\in \calL{p}(\X)$ and this concludes the proof.
\end{proof}
\begin{theorem}
    \label{theo:density_energy_complete}
    Let $(\X,\sfd,\mm)$ be a metric measure space such that $(\X,\sfd)$ is complete. Then \begin{equation}
        \HH{p}{\Chain{},\, \Lip}(\X) = \HH{p}{curve}(\X)
    \end{equation}
    and
    \begin{equation}
        \|u \|_{\HH{p}{\Chain{},\, \Lip}(\X)}  = \|u \|_{\HH{p}{curve}(\X)}
    \end{equation}
    for every $u\in L^p(\X)$.
\end{theorem}
\begin{remark}
\label{rem:comments_between_CC_EB}
The proof of Theorem \ref{theo:density_energy_complete} can be done following word by word the proof of \cite[Theorem 1.1]{Eriksson-Bique-Density-in-energy}, with very few modifications, like the obvious one due to our definition of integral along chains, that requires Lemma \ref{lemma:lower-semicontinuity-integral-chains-convergence} in place of \cite[Lemma 2.19]{Eriksson-Bique-Density-in-energy}. However, we are not able to use this scheme of demonstration in order to prove the next Theorem \ref{theo:density-energy-chains-complete}. For this reason we propose a proof of Theorem \ref{theo:density_energy_complete} which is still inspired to the one of \cite[Theorem 1.1]{Eriksson-Bique-Density-in-energy}, but that can be easily modified to prove Theorem \ref{theo:density-energy-chains-complete}. The main difference relies on the simplifications procedures: while we are able to reduce the proof to bounded functions with bounded support, in Theorem \ref{theo:density-energy-chains-complete} we are not able to restrict the study to nonnegative functions. This is due to the fact that the $p$-minimal $\varepsilon$-weak upper gradients are not local in any suitable sense. In particular is not clear how to show that $\relF{\Chain{}}(u) = \relF{\Chain{}}(u_+) + \relF{\Chain{}}(u_-)$, where $u_+$ and $u_-$ are the positive and negative part of $u$. Notice that, a posteriori, this has to be true because of Theorems \ref{theo:density-energy-chains-complete} and \ref{theo:density_energy_complete}, since $\relF{curve}(u) = \relF{curve}(u_+) + \relF{curve}(u_-)$.
\end{remark}
\begin{proof}[Proof of Theorem \ref{theo:density_energy_complete}]
    For simplicity we divide the proof in steps. 
    \vspace{2mm}
    
    {\bf Step 1.} It is enough to prove that for every bounded function $u$ with bounded support it holds $\relF{\Chain{},\,\Lip}(u) \le \F{curve}(u)$. 
    Indeed, by \cite[Proposition 7.1.35]{HKST15}, for every $u\in L^p(\X)$ such that $\F{curve}(u) < \infty$ we can find a sequence of bounded functions $u_j$ with bounded support such that $u_j \to u$ in $L^p(\X)$ and $\limi_{j\to +\infty}\F{curve}(u_j) \le \F{curve}(u)$. Therefore 
    $$\relF{\Chain{},\,\Lip}(u) \le \limi_{j\to +\infty}\relF{\Chain{},\,\Lip}(u_j) \le \limi_{j\to +\infty}\F{curve}(u_j) \le \F{curve}(u),$$
    because of the lower semicontinuity of $\relF{\Chain{},\,\Lip}$. Since $\relF{curve}$ is the biggest lower semicontinuous functional which is smaller than or equal to $\F{curve}$, we infer that $\relF{\Chain{},\,\Lip}(u) \le \relF{curve}(u)$ for every $u\in L^p(\X)$ such that $\F{curve}(u) < \infty$. The thesis for an arbitrary $u\in L^p(\X)$ follows directly by the definitions of $\relF{\Chain{},\,\Lip}(u)$ and $\relF{curve}(u)$.
    \vspace{2mm}

    In the following we fix $x_0\in \X$ and we assume that $u\colon \X \to [-M,M]$ and that there exists $R\ge 3$ such that $u\restr{\X\setminus B_R(x_0)} = 0$. 
    \vspace{2mm}
    
    {\bf Step 2.} We claim that it is enough to prove the following statement. For every upper gradient $g\in \UG{}(u)$ and for every $\eta > 0$ there exists another upper gradient $g_\eta \in \UG{}(u)$ such that 
    \begin{equation}
        \label{eq:approx_with_g_eta}
        \Vert g - g_\eta \Vert_{L^p(\X)} < \eta
    \end{equation} 
    and with the following property. For every $j\in \N$ there exist functions $u_{\eta,j} \colon \X \to \R$ and $g_{\eta,j}\in \LUG{\frac{1}{j}}(u_{\eta,j})$ such that 
    \begin{equation}
        \label{eq:approx_with_u_eta,j}
        \lims_{j\to +\infty} \Vert u_{\eta,j} - u \Vert_{L^p(\X)} \le \eta \quad \text{and} \quad \lim_{j\to +\infty} \Vert g_{\eta,j} - g_\eta \Vert_{L^p(\X)} = 0
    \end{equation}
    Indeed, if the claim is true, then we have
        \begin{equation}
            \limi_{j\to +\infty}\F{\Chain{},\, \Lip}(u_{\eta,j}) \le \limi_{j\to +\infty} \Vert g_{\eta,j}\Vert_{L^p(\X)}= \Vert g_{\eta}\Vert_{L^p(\X)} \le \Vert g\Vert_{L^p(\X)} + \eta
        \end{equation}
    for every $\eta > 0$, where we used \eqref{eq:approx_with_g_eta} and \eqref{eq:approx_with_u_eta,j}. By a diagonal argument we deduce that $\relF{\Chain{},\,\Lip}(u)\le \Vert g\Vert_{L^p(\X)}$. By the arbitrariness of $g\in \UG{}(u)$ we infer that $\relF{\Chain{},\,\Lip}(u) \le \F{curve}(u)$, which is the statement we had to prove from Step 1.
    \vspace{2mm}

    In the remaining steps we will prove the claim of Step 2. In order to simplify the proof we notice that it is enough to prove the statement for upper gradients $g\in \UG{}(u)$ that are lower semicontinuous and such that $g \equiv 0$ on $\X\setminus B_{2R}(x_0)$. The first assertion follows by Vitali-Carathéodory Theorem (cp. \cite[page 108]{HKST15}), while the second one follows by truncation: for every $g\in \UG{}(u)$, the truncated function $g\cdot \chi_{B_{2R}(x_0)}$ is still an upper gradient of $u$ since $u\restr{\X\setminus B_R(x_0)} = 0$ and it is smaller than the original one. In the sequel we assume that $g$ has these properties.
    \vspace{2mm}

    {\bf Step 3:} For every $g$ as above, we define $g_\eta$. By Lusin's Theorem and the fact that $\mm$ is Radon, we can find compact sets $K_j \subseteq B_{2R}(x_0)$ such that $\mm(B_{2R}(x_0) \setminus K_j)^{\frac{1}{p}} \le 2^{-j}\eta$ and so that $u\restr{K_j}$ is continuous. We can suppose $K_j \subseteq K_{j+1}$ for every $j$. Let $\sigma \in (0, 1)$ be so that $\mm(B_{2R}(x_0))^{\frac{1}{p}}\sigma \le \frac{\eta}{2}$.
    Define
        \begin{equation}
            g_\eta(x) := g(x) + \sigma\chi_{B_{2R}(x_0)} + \sum_{i=1}^\infty \chi_{B_{2R}(x_0)\setminus K_i}(x).
        \end{equation}
        The function $g_\eta$ is still lower semicontinuous and belongs to $\UG{}(u)$, since it is bigger than $g$. Moreover, $g_\eta \equiv 0$ on $\X \setminus B_{2R}(x_0)$. 
        We show that \eqref{eq:approx_with_g_eta} holds. Indeed
        \begin{equation}
            \Vert g - g_\eta \Vert_{L^p(\X)} \le \sigma \mm(B_{2R}(x_0))^{\frac{1}{p}} + \sum_{i=1}^\infty \mm(B_{2R}(x_0)\setminus K_i)^{\frac{1}{p}} \le \eta.
        \end{equation}

    In the next step we define auxiliary functions $\hat{g}_{\eta,j}$. Later we will slightly modify these functions in order to define the $g_{\eta,j}$'s of Step 2.
    \vspace{2mm}
    
    {\bf Step 4.} We proceed to the definition of $\hat{g}_{\eta,j}$. Let $g_{j}$ be an increasing sequence of bounded Lipschitz functions such that $g_j \nearrow g$ pointwise, whose existence is guaranteed for instance by \cite[Corollary 4.2.3]{HKST15}. Define $\psi_{j}(x) := \max\{0,\min\{1, j(2R - \sfd(x_0, x))\}\}$. Observe that $\psi_{j}$ is $j$-Lipschitz, that $\psi_{j} \le \chi_{B_{2R}(x_0)}$, that $\psi_{j} \equiv 1$ on $B_{2R - \frac{1}{j}}(x_0)$ and that $\psi_{j} \to \chi_{B_{2R}(x_0)}$ pointwise as $j\to +\infty$.    
    We define
    \begin{equation}
            \hat{g}_{\eta,j}(x) = g_j(x) + \sigma\psi_j(x) + \sum_{i=1}^j \min\{j\sfd(x,K_i), \psi_j(x)\}.
        \end{equation}
    By definition, $\hat{g}_{\eta,j}$ is Lipschitz and bounded. Moreover, $\hat{g}_{\eta,j} \le g_\eta$ for every $j$ and $\hat{g}_{\eta,j} \nearrow g_\eta$ pointwise as $j \to +\infty$. 
    \vspace{2mm}

    We define auxiliary functions $\hat{u}_{\eta,j}$. In Step 6 will modify them in order to define the functions $u_{\eta,j}$ required by Step 2.

    {\bf Step 5.} We define $\hat{u}_{\eta,j}$. We choose $N \in \N$ so that $\mm(B_{2R}(x_0) \setminus K_N)^{\frac{1}{p}} \le (2M)^{-1}\eta$. We define the closed set $A := K_N \cup (\X \setminus B_R(x_0))$. Since $u\restr{K_N}$ and $u\restr{X\setminus B_R(x_0)}$ are continuous and since both sets are closed, then $u\restr{A}$ is continuous as well. 
    We set
    $$\hat{u}_{\eta,j}(x) := \min\left\{ M, \inf\left\{ u(\alpha(\sfc)) + \int_\sfc \hat{g}_{\eta,j} \,:\, \sfc \in \Chain{\frac{1}{j}}, \omega(\sfc) = x, \alpha(\sfc) \in A\right\} \right\}.$$
    These functions satisfy the following properties:
        \begin{itemize}
            \item[(a)] $\hat{u}_{\eta,j}\colon \X \to [-M,M]$ and $\hat{u}_{\eta,j} \le u$ on $A$: this follows directly from the definition.
            \item[(b)] $\hat{u}_{\eta,j}$ is $\max\{2Mj, \sup_\X \hat{g}_{\eta,j}\}$-Lipschitz. Indeed if $x,y\in \X$ are such that $\sfd(x,y) > \frac{1}{j}$ then $\vert \hat{u}_{\eta,j}(x) - \hat{u}_{\eta,j}(y)\vert \le 2M \le 2Mj\sfd(x,y)$. On the other hand, if $\sfd(x,y)<\frac{1}{j}$ then $\{x,y\}\in\Chain{\frac{1}{j}}_{x,y}$, implying that $\vert \hat{u}_{\eta,j}(y) - \hat{u}_{\eta,j}(x)\vert \le \frac{\hat{g}_{\eta,j}(x) + \hat{g}_{\eta,j}(y)}{2}\sfd(x,y) \le \sup_\X \hat{g}_{\eta,j} \sfd(x,y)$.
            \item[(c)] $\hat{g}_{\eta,j} \in \UG{\frac{1}{j}}(\hat{u}_{\eta,j})$. We prove that $\hat{u}_{\eta,j}(y) - \hat{u}_{\eta,j}(x) \le \int_{\sfc}\hat{g}_{\eta,j}$ for every $x,y\in \X$ and every $\sfc \in \Chain{\frac{1}{j}}_{x,y}$. This is enough to show the thesis, since the integral is symmetric. If $\hat{u}_{\eta,j}(x) = M$ there is nothing to prove. Otherwise for every $\delta > 0$ we can find a chain $\sfc_\delta \in \Chain{\frac{1}{j}}$ with $\omega(\sfc_\delta) = x$ and $\alpha(\sfc_\delta) \in A$ such that $\hat{u}_{\eta,j}(x) \ge u(\alpha(\sfc_\delta)) + \int_{\sfc_\delta} \hat{g}_{\eta,j} - \delta$. The chain $\sfc_{\delta} \star \sfc$ is admissible for the computation of the infimum in the definition of $\hat{u}_{\eta,j}(y)$, giving $\hat{u}_{\eta,j}(y) \le u(\alpha(\sfc_\delta)) + \int_{\sfc_\delta} \hat{g}_{\eta,j} + \int_\sfc \hat{g}_{\eta,j} \le \hat{u}_{\eta,j}(x) + \int_\sfc g_{\eta,j} + \delta$. The thesis follows by the arbitrariness of $\delta$.
            \item[(d)] $\hat{u}_{\eta,j}(x) \le \hat{u}_{\eta,k}(x)$ if $j\le k$. This follows since $\hat{g}_{\eta,j} \le \hat{g}_{\eta,k}$ and since each $\frac{1}{k}$-chain is also a $\frac{1}{j}$-chain.
            \item[(e)] $\hat{u}_{\eta,j}$ is constant on each $\frac{1}{j}$-chain connected component of $\X \setminus B_{2R}(x_0)$. Indeed, let $x,y \in \X$ be such that there exists a $\frac{1}{j}$-chain $\sfc_{x,y} \subseteq \X \setminus B_{2R}(x_0)$. Let $\sfc \in \Chain{\frac{1}{j}}$ be such that $\alpha(\sfc) \in A$ and $\omega(\sfc) = x$. Then $\sfc' = \sfc \star \sfc_{x,y} \in \Chain{\frac{1}{j}}$ and satisfies $\alpha(\sfc') \in A$, $\omega(\sfc') = y$. Moreover, $$\int_{\sfc'} \hat{g}_{\eta,j} = \int_{\sfc} \hat{g}_{\eta,j} + \int_{\sfc_{x,y}} \hat{g}_{\eta,j} = \int_{\sfc} \hat{g}_{\eta,j}$$
            since $\hat{g}_{\eta,j} \equiv 0$ on $\X \setminus B_{2R}(x_0)$. This is enough to show that $\hat{u}_{\eta,j}(y) \le \hat{u}_{\eta,j}(x)$. Reversing the roles of $x$ and $y$ we get the opposite inequality and so that $\hat{u}_{\eta,j}(y) = \hat{u}_{\eta,j}(x)$.
        \end{itemize}
    \vspace{2mm}

    {\bf Step 6.} Definition of $u_{\eta,j}$ and $g_{\eta,j}$. We define $u_{\eta,j}$ with a cutoff procedure to impose that $u_{\eta,j} \equiv 0$ outside $B_{3R}(x_0)$. We piecewisely define $u_{\eta,j}$. On $B_{2R}(x_0)$ we set $u_{\eta,j} = \hat{u}_{\eta,j}$. Then we define $u_{\eta,j}$ on each $\frac{1}{j}$-chain connected component $\Y$ of $\X\setminus B_{2R}(x_0)$ in the following way.
    By items (e) and (a) of Step 5, we have that $\hat{u}_{\eta,j}$ is constantly equal to some $\delta_{\Y} \in [-M,0]$ on $\Y$. We define $u_{\eta,j}$ on $\Y$ by
            \begin{equation}
                u_{\eta,j}(x) :=
                \begin{cases}
                    -\frac{\delta_\Y}{R}\sfd(x,x_0) +3\delta_\Y &\text{ if } \sfd(x,x_0)\in [2R,3R],\\
                    0 &\text{ if } \sfd(x,x_0) \ge 3R.
                \end{cases}
            \end{equation}
        The same proof of item (b) of Step 5 implies that $u_{\eta,j}$ is $\max\{ 2Mj, \sup_\X \hat{g}_{\eta,j} + \frac{M}{R}\}$-Lipschitz. Indeed, the only non trivial case is when $\sfd(x,y) < \frac{1}{j}$. In that case, if $x,y \in B_{2R}(x_0)$ then the proof does not change. If $x,y \in \X \setminus B_{2R}(x_0)$ then they must be in the same $\frac{1}{j}$-chain connected component $\Y$ of $\X \setminus B_{2R}(x_0)$, so $\vert u_{\eta,j}(x) - u_{\eta,j}(y) \vert \le \frac{\vert \delta_Y \vert}{R}\vert \sfd(x,x_0) - \sfd(y,x_0) \vert \le \frac{M}{R} \sfd(x,y)$. In the last case we have $x\in B_{2R}(x_0)$ and $y\in \X\setminus B_{2R}(x_0)$. Here we have
        $$\vert u_{\eta,j}(x) - u_{\eta,j}(y) \vert \le \vert \hat{u}_{\eta,j}(x) - \hat{u}_{\eta,j}(y) \vert + \frac{\vert \delta_\Y \vert}{R} (\sfd(y,x_0) - 2R) \le \sup_{\X} \hat{g}_{\eta,j} \sfd(x,y) + \frac{M}{R} \sfd(x,y),$$
        where $\Y$ is the $\frac{1}{j}$-chain connected component of $\X \setminus B_{2R}(x_0)$ containing $y$.

        It remains to define the new gradients $g_{\eta,j}$. We set 
        $$\delta_j := \sup \left\{ \vert \delta_\Y \vert \,:\, \Y \in \Chain{\frac{1}{j}}\textup{-cc}(\X \setminus B_{2R}(x_0))\right\}$$
        and  
        $$h_{\eta,j} := \frac{\delta_j}{R} \cdot \max\left\{0,\min\left\{1, 5 - \frac{\sfd(x_0, x)}{R}\right\}\right\}.$$
        Observe that $h_{\eta,j} = \frac{\delta_j}{R}$ on $B_{4R}(x_0)$, that $h_{\eta,j}$ is Lipschitz and that $h_{\eta,j} \equiv 0$ on $\X\setminus B_{5R}(x_0)$. Finally define $g_{\eta,j} := \hat{g}_{\eta,j} + h_{\eta,j}$.
        We claim that $g_{\eta,j} \in \LUG{\frac{1}{j}}(u_{\eta,j})$. By definition, $g_{\eta,j}$ is Lipschitz, so it remains to show it is a $\frac{1}{j}$-upper gradient of $u_{\eta,j}$. Let $\sfc \in \Chain{\frac{1}{j}}$. We divide $\sfc$ in subchains $\sfc_i$ such that $\omega(\sfc_i) = \alpha(\sfc_{i+1})$ for every $i$ and such that each $\sfc_i$ is of one of the following forms: 
        \begin{itemize}
            \item $\sfc_i \subseteq B_{2R}(x_0)$;
            \item $\sfc_i \subseteq A_{2R,3R}(x_0) := \overline{B}_{3R}(x_0) \setminus B_{2R}(x_0)$;
            \item $\sfc_i \subseteq \X \setminus B_{3R}(x_0)$
            \item $\sfc_i = \{x_i,y_i\}$ with $x_i \in B_{2R}(x_0)$ and $y_i \notin B_{2R}(x_0)$ or $x_i \notin B_{2R}(x_0)$ and $y_i \in B_{2R}(x_0)$;
            \item $\sfc_i = \{x_i,y_i\}$ with $x_i \in B_{3R}(x_0)$ and $y_i \notin B_{3R}(x_0)$ or $x_i \notin B_{3R}(x_0)$ and $y_i \in B_{3R}(x_0)$.
        \end{itemize}
        In all these cases we prove that $\vert u_{\eta,j}(\omega(\sfc_i)) - u_{\eta,j}(\alpha(\sfc_i)) \vert \le \int_{\sfc_i}g_{\eta,j}$.
         If $\sfc_i \subseteq B_{2R}(x_0)$, we use item (c) of Step 5 to get $\vert u_{\eta,j}(\omega(\sfc_i)) - u_{\eta,j}(\alpha(\sfc_i)) \vert = \vert \hat{u}_{\eta,j}(\omega(\sfc_i)) - \hat{u}_{\eta,j}(\alpha(\sfc_i)) \vert \le \int_{\sfc_i} \hat{g}_{\eta,j} \le \int_{\sfc_i} g_{\eta,j}$. If $\sfc_i \subseteq A_{2R,3R}(x_0)$, then it must be contained in the same $\frac{1}{j}$-chain connected component $\Y$ of $\X \setminus B_{2R}(x_0)$. Therefore we have $\vert u_{\eta,j}(\omega(\sfc_i)) - u_{\eta,j}(\alpha(\sfc_i)) \vert \le \frac{\vert \delta_Y \vert}{R}\vert \sfd(\omega(\sfc_i),x_0) - \sfd(\alpha(\sfc_i),x_0) \vert \le \frac{\delta_j}{R} \sfd(\omega(\sfc_i),\alpha(\sfc_i)) \le \int_{\sfc_i} h_{\eta,j} \le \int_{\sfc_i} g_{\eta,j}$. If $\sfc_i \subseteq \X \setminus B_{3R}(x_0)$ then $ 0 = \vert u_{\eta,j}(\omega(\sfc_i)) - u_{\eta,j}(\alpha(\sfc_i)) \vert \le \int_{\sfc_i} g_{\eta,j}$.
        If $\sfc_i = \{x_i,y_i\}$ is as in the last two cases, we have
        \begin{equation}
            \begin{aligned}
                \vert u_{\eta,j}(y_i) - u_{\eta,j}(x_i) \vert &\le \vert \hat{u}_{\eta,j}(y_i) - \hat{u}_{\eta,j}(x_i)\vert + \frac{\vert \delta_j \vert}{R}\max\{\sfd(y_i,\partial A_{2R,3R}(x_0)), \sfd(x_i,\partial A_{2R,3R}(x_0))\}\\
                &\le \int_{\{x_i,y_i\}} \hat{g}_{\eta,j} + \int_{\{x_i,y_i\}} h_{\eta,j} = \int_{\{x_i,y_i\}} g_{\eta,j},
            \end{aligned}
        \end{equation}
        because $\max\{\sfd(x_i,\partial A_{2R,3R}(x_0)), \sfd(y_i,\partial A_{2R,3R}(x_0))\} \le \sfd(x_i,y_i)$ and item (c) of Step 5. Therefore
        $$\vert u_{\eta,j}(\omega(\sfc)) - u_{\eta,j}(\alpha(\sfc)) \vert \le \sum_i \vert u_{\eta,j}(\omega(\sfc_i)) - u_{\eta,j}(\alpha(\sfc_i)) \vert \le \sum_i\int_{\sfc_i} g_{\eta,j} = \int_{\sfc} g_{\eta,j}.$$
        \vspace{2mm}

        \textbf{Step 7.} In this step we show that $u_{\eta,j}$ and $g_{\eta,j}$ satisfy \eqref{eq:approx_with_u_eta,j} if we prove that $u_{\eta,j}$ converges pointwise to $u$ on $K_N$ and $\hat{u}_{\eta,j}$ converges uniformly to $u \equiv 0$ on $\X \setminus B_{2R}(x_0)$ as $j\to +\infty$.
        Indeed, if this is true, we get that $\delta_j$ satisfies
        \begin{equation}
        \label{eq:delta_j_tends_to_0}
            \lims_{j\to +\infty} \delta_j = \lims_{j \to +\infty} \|\hat{u}_{\eta,j}\|_{L^\infty(\X\setminus B_{2 R}(x_0))}= 0.
        \end{equation}
        Therefore we obtain
        \begin{equation}
            \begin{aligned}
                &\lim_{j\to +\infty} \Vert u - u_{\eta,j} \Vert_{L^p(\X)}  \\
                &=\lim_{j\to +\infty}\left(\int_{K_N} \vert u - u_{\eta,j} \vert^p \,\d\mm + \int_{B_{2R}(x_0)\setminus K_N} \vert u - u_{\eta,j} \vert^p \,\d\mm + \int_{\X \setminus B_{2R}(x_0)} \vert u - u_{\eta,j} \vert^p \,\d\mm \right)^{\frac{1}{p}} \\
                &\le 0 + (2M)\mm(B_{2R}(x_0) \setminus K_N)^{\frac{1}{p}} + 0 \le \eta.
            \end{aligned}
        \end{equation}
        We used dominated convergence for the estimate of the first summand and the estimate $\vert u - u_{\eta,j} \vert \le 2M$ since both functions take values on $[-M,M]$ and $\mm(B_{2R}(x_0) \setminus K_N)^{\frac{1}{p}} \le \eta (2M)^{-1}$ for the second term. For the third term we divided the integral in two parts: on the annulus $A_{2R,3R}(x_0)$ we used again dominated convergence since $\vert u-u_{\eta,j} \vert = \vert u_{\eta,j}\vert \le \delta_j$ on it, and we can use \eqref{eq:delta_j_tends_to_0}, while on $\X \setminus B_{3R}(x_0)$ we have $\vert u - u_{\eta,j} \vert = 0$.
        This concludes the first estimate in \eqref{eq:approx_with_u_eta,j}. For the second one we observe that
        \begin{equation}
            \lims_{j\to +\infty} \Vert g_{\eta,j} - g_\eta \Vert_{L^p(\X)} \le \lims_{j\to +\infty} \left( \Vert \hat{g}_{\eta,j} - g_\eta \Vert_{L^p(\X)} + \Vert h_{\eta,j} \Vert_{L^p(\X)}\right) = 0,
        \end{equation}
        where we used dominated convergence on the first term, since $\hat{g}_{\eta,j} \to g_\eta$ pointwise almost everywhere and they are supported on $B_{2R}(x_0)$, and the estimate
        $$\Vert h_{\eta,j} \Vert_{L^p(\X)} \le \frac{\delta_j}{R} \mm(\overline{B}_{5R}(x_0))^{\frac{1}{p}},$$
        where the limit superior of the right hand side is $0$ because of \eqref{eq:delta_j_tends_to_0}.
    \vspace{2mm}

    In the last two steps we show that $u_{\eta,j}$ converges to $u$ pointwise on $K_N$ and that $\hat{u}_{\eta,j}$ converges uniformly to $0$ outside $B_{2R}(x_0)$.
    \vspace{2mm}
    
    {\bf Step 8.} We prove that $u_{\eta,j}$ converges pointwise to $u$ on $K_N$ as $j\to +\infty$. 
    We suppose by contradiction that there exists some $x \in K_N$ such that $u_{\eta,j}(x)$ does not converge to $u(x)$ as $j$ goes to $+\infty$. On $K_N$ we have $u_{\eta,j} = \hat{u}_{\eta,j}$, by definition. By item (d) of Step 5, the sequence $\hat{u}_{\eta,j}(x)$ is increasing and so it admits a limit. Moreover, by item (a) of Step 5, $\hat{u}_{\eta,j}(x) \le u(x)$ for every $j$. So, our assumption means that $\lim_{j\to +\infty} \hat{u}_{\eta,j}(x) < u(x)$. Let us fix $\delta > 0$ such that $\lim_{j\to +\infty} \hat{u}_{\eta,j}(x) \le u(x) - \delta$. Since $u(x) \le M$, we get $\hat{u}_{\eta,j}(x) \le M-\delta$ for every $j$. By definition, we can find chains $\sfc_j \in \Chain{\frac{1}{j}}$ such that $\omega(\sfc_j) = x$, $\alpha(\sfc_j) \in A$ and
    \begin{equation}
        \label{eq:UG_inequality_contradicted}
        u(\alpha(\sfc_j)) + \int_{\sfc_j} \hat{g}_{\eta,j} < u(x) - \frac{\delta}{2} \le M.
    \end{equation}

    \begin{figure}
        \centering
        \includegraphics[scale=0.92]{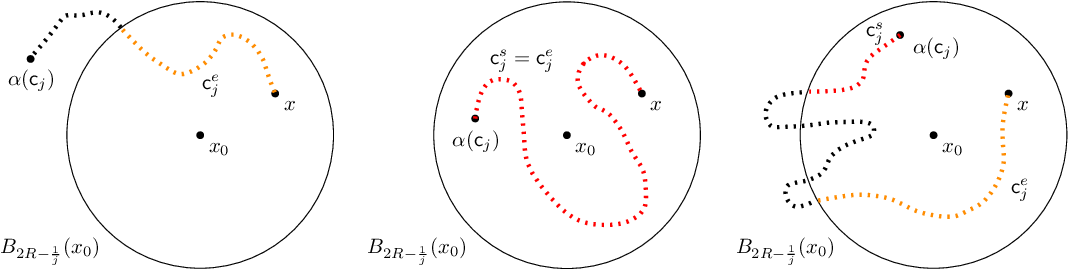}
        \caption{The picture shows the definition of $\sfc_j^s$ and $\sfc_j^e$ in three different situations that cover all possible cases. On the left, $\alpha(\sfc_j) \notin B_{2R-\frac{1}{j}}(x_0)$, so $\sfc_j^s = \emptyset$ and $\sfc_j^e \neq \sfc_j$. In the middle, $\alpha(\sfc_j) \in B_{2R-\frac{1}{j}}(x_0)$ and $\sfc_j$ is contained in $B_{2R-\frac{1}{j}}(x_0)$, so $\sfc_j = \sfc_j^s = \sfc_j^e$. On the right, $\alpha(\sfc_j) \in B_{2R-\frac{1}{j}}(x_0)$, but $\sfc_j \cap (\X \setminus B_{2R - \frac{1}{j}}(x_0)) \neq \emptyset$, so $\sfc_j^s \neq \emptyset$, $\sfc_j^s \neq \sfc_j$ and $\sfc_j^e \neq \sfc_j$.}
        \label{fig:definition_of_subchains}
    \end{figure}
    
    We consider two subchains. If $\sfc_j = \{q_0^j, \ldots, q_{N_j}^j = x\}$ we define $\sfc_j^s := \{q_0^j,\ldots,q_{i_j}^j\}$ and $\sfc_j^e := \{q_{k_j}^j,\ldots,q_{N_j}^j\}$, where $i_j$ is the biggest integer $i$ such that the chain $\{ q_0^j,\ldots,q_{i}^j \}$ is contained in $B_{2R - \frac{1}{j}}(x_0)$, while
    $k_j$ is the smallest integer $k$ such that the chain $\{q_{k}^j,\ldots,q_{N_j}^j = x\}$ is contained in $B_{2R - \frac{1}{j}}(x_0)$. Here, the superscript stay for `start' and `end', respectively. Figure \ref{fig:definition_of_subchains} represents these subchains in three different exhaustive situations. As $x\in K_N \subseteq B_{2R}(x_0)$, we have that $x \in B_{2R - \frac{1}{j}}(x_0)$ for $j$ big enough. For these indices, $\sfc_j^e$ is not empty, as it contains at least $x$. Moreover, $\omega(\sfc_j^e) = x$ and $\alpha(\sfc_j^e) \in A$. The last assertion can be proved as follows: either $\sfc_j^e = \sfc_j$, so $\alpha(\sfc_j^e) = \alpha(\sfc_j) \in A$, or $\sfd(\alpha(\sfc_j^e), x_0) \ge 2R - \frac{2}{j} \ge R$, because of the maximality property of $\sfc_j^e$, and so $\alpha(\sfc_j^e) \in \X \setminus B_{R}(x_0) \subseteq A$. On the other hand, $\sfc_j^s$ can be empty, and it is empty if and only if $\alpha(\sfc_j) \notin B_{2R - \frac{1}{j}}(x_0)$. If $\sfc_j^s$ is not empty then either $\sfc_j^s = \sfc_j$, so $\omega(\sfc_j^s) = x \in A$, or $\omega(\sfc_j^s) \notin B_{2R - \frac{2}{j}}(x_0)$ by the maximality property of $\sfc_j^s$, so $\omega(\sfc_j^s) \in A$. Moreover, $\alpha(\sfc_j^s) = \alpha(\sfc_j) \in A$. In any case the four points $\alpha(\sfc_j^s)$, $\omega(\sfc_j^s)$, $\alpha(\sfc_j^e)$, $\omega(\sfc_j^e)$ belong to $A$.
    There are three possible cases:
    \begin{itemize}
        \item[(1)] $\alpha(\sfc_j) \notin B_{R}(x_0)$, so $u(\alpha(\sfc_j)) = 0 = u(\alpha(\sfc_j^e))$. In this case, \eqref{eq:UG_inequality_contradicted} implies that
        \begin{equation}
            \label{eq:UG_inequality_contradicted_case_1}
            u(\alpha(\sfc_j^e)) + \int_{\sfc_j^e} \hat{g}_{\eta,j} \le u(\alpha(\sfc_j)) + \int_{\sfc_j} \hat{g}_{\eta,j} < u(x) - \frac{\delta}{2} = u(\omega(\sfc_j^e)) - \frac{\delta}{2}.
        \end{equation}
        \item[(2)] $\alpha(\sfc_j) \in B_{R}(x_0)$, which means $\sfc_j^s \neq \emptyset$ and $\alpha(\sfc_j^s) = \alpha(\sfc_j)$, and we have
        \begin{equation}
            \label{eq:UG_inequality_contradicted_case_2}
            u(\alpha(\sfc_j^s)) + \int_{\sfc_j^s} \hat{g}_{\eta,j} < u(\omega(\sfc_j^s)) - \frac{\delta}{4}.
        \end{equation}
        \item[(3)] $\alpha(\sfc_j) \in B_{R}(x_0)$ and \eqref{eq:UG_inequality_contradicted_case_2} does not hold. In this case, in view of \eqref{eq:UG_inequality_contradicted}, we necessarily have $\sfc_j^s \neq \sfc_j \neq \sfc_j^e$, so $\omega(\sfc_j^s), \alpha(\sfc_j^e) \notin B_{2R - \frac{2}{j}}(x_0)$. Moreover
        \begin{equation}
            \label{eq:UG_inequality_contradicted_case_3}
            \begin{aligned}
                u(\alpha(\sfc_j^e)) + \int_{\sfc_j^e} \hat{g}_{\eta,j} = u(\omega(\sfc_j^s)) + \int_{\sfc_j^e} \hat{g}_{\eta,j} &\le u(\alpha(\sfc_j^s)) + \int_{\sfc_j^s} \hat{g}_{\eta,j} + \int_{\sfc_j^e} \hat{g}_{\eta,j} + \frac{\delta}{4}\\
                &\le u(\alpha(\sfc_j)) + \int_{\sfc_j} \hat{g}_{\eta,j} + \frac{\delta}{4} \\
                &< u(x) - \frac{\delta}{4} = u(\omega(\sfc_j^e)) - \frac{\delta}{4}
            \end{aligned}
        \end{equation}
    \end{itemize}

    One of the three cases (1), (2) or (3) holds true for infinitely many $j$'s. We now show how to conclude the proof supposing that case (1) occurs infinitely many times. Later we will show how to conclude in the other two cases. We restrict to the indices where (1) holds true and we do not relabel the subsequence. We claim that the assumptions of Proposition \ref{prop:Sylvester_compactness_chains} (see also the discussion in Remark \ref{rmk:compactness_chains_with_lambda}) are satisfied by $\{ \sfc_j^e \}_j$.
    \begin{itemize}
        \item {\em Length}. $\ell(\sfc_j^e) \le \sigma^{-1}\int_{\sfc_j^e}\hat{g}_{\eta,j} \le \sigma^{-1}M$, where we used \eqref{eq:UG_inequality_contradicted_case_1}, \eqref{eq:UG_inequality_contradicted} and the fact that $\sfc_j^e \subseteq B_{2R - \frac{1}{j}}(x_0)$, so $\hat{g}_{\eta,j} \ge \sigma$  on $\sfc_j^e$. 
        \item {\em Diameter}. Since $u\restr{A}$ is continuous, we can take $\Delta > 0$ so that for every $y\in A$ such that $\sfd(x, y) \le \Delta$ we have $\vert u(x) - u(y) \vert \le \frac{\delta}{4}$. Since by \eqref{eq:UG_inequality_contradicted_case_1} $u(\alpha(\sfc_j^e)) < u(x) - \frac{\delta}{4}$, and since $\alpha(\sfc_j^e) \in A$, we conclude that ${\rm Diam}(\sfc_j^e) \ge \sfd(\alpha(\sfc_j^e), x) \ge \Delta$.
        \item {\em $h$-sum}. By definition, $h_j\restr{B_{2R - \frac{1}{j}}(x_0)} \le \hat{g}_{\eta,j}\restr{B_{2R - \frac{1}{j}}(x_0)}$, so
        \begin{equation}
            \label{eq:weight_upper_bound}
            \int_{\sfc_j^e}h_j \le \int_{\sfc_j^e} \hat{g}_{\eta,j}\le M,
        \end{equation}
        again by \eqref{eq:UG_inequality_contradicted_case_1} and \eqref{eq:UG_inequality_contradicted}.
    \end{itemize}
    Therefore the chains $\sfc_j^e$ subconverge to a curve $\gamma$ of $\X$, by Proposition \ref{prop:Sylvester_compactness_chains} and Remark \ref{rmk:compactness_chains_with_lambda}. We relabel the sequence accordingly, and we denote the chains again by $\sfc_j^e$. Since $A$ is closed, so $\alpha(\gamma) = \lim_{j\to +\infty} \alpha(\sfc_j^e)$ belongs to $A$, and $u\restr{A}$ is continuous, we have that 
    $u(\alpha(\gamma)) = \lim_{j\to +\infty} u(\alpha(\sfc_j^e)).$
    This, together with Step 4, says that we are in position to apply Lemma \ref{lemma:lower-semicontinuity-integral-chains-convergence} to the functions $\hat{g}_{\eta,j} \nearrow g_\eta$ and to the sequence of chains $\sfc_j^e$ converging to $\gamma$, concluding that
    \begin{equation}
        \begin{aligned}
            u(\alpha(\gamma)) + \int_\gamma g_\eta \le \limi_{j\to +\infty} u(\alpha(\sfc_j^e)) + \int_{\sfc_j^e} \hat{g}_{\eta,j} \le
            u(x) - \frac{\delta}{2} = u(\omega(\gamma)) - \frac{\delta}{2}.
        \end{aligned}
    \end{equation}
    Here we used \eqref{eq:UG_inequality_contradicted_case_1} in the last inequality.
    This contradicts the fact that $g_\eta \in \UG{}(u)$.

    Suppose now that case (3) occurs for infinitely many indices $j$. Then \eqref{eq:UG_inequality_contradicted_case_3} and the same proof given above says that $\{\sfc_j^e\}_j$ is a sequence of chains satisfying again the assumptions of Proposition \ref{prop:Sylvester_compactness_chains}. The remaining part of the argument is exactly the same, using \eqref{eq:UG_inequality_contradicted_case_3} to violate the fact that $g_\eta$ is an upper gradient of $u$.

    Finally suppose that (2) occurs for infinitely many indices. Then the sequence $\{\sfc_j^s\}_j$ satisfies the assumptions of Proposition \ref{prop:Sylvester_compactness_chains}. Indeed, the estimate of the length and the $h$-sum is identical, using \eqref{eq:UG_inequality_contradicted_case_2} instead of \eqref{eq:UG_inequality_contradicted_case_1}. In the proof of the lower bound of the diameter we need to distinguish two cases. If $\sfc_j^s = \sfc_j$ then the same proof as above says that ${\rm Diam}(\sfc_j^s) \ge \Delta$. Otherwise we have that $\omega(\sfc_j^s) \notin B_{2R-\frac{2}{j}}(x_0)$, so ${\rm Diam}(\sfc_j^s) \ge \sfd(\alpha(\sfc_j^s), \omega(\sfc_j^s)) \ge R - \frac{2}{j} \ge 1$. In any case, ${\rm Diam}(\sfc_j^s) \ge \min\{\Delta, 1\} > 0.$    
    The remaining part of the argument is again the same, using \eqref{eq:UG_inequality_contradicted_case_2} to violate the fact that $g_\eta$ is an upper gradient of $u$. We remark that the argument works since both endpoints of $\sfc_j^s$ belong to $A$, which is closed and on which $u$ is continuous.
    \vspace{2mm}

    {\bf Step 9.} We prove that $\hat{u}_{\eta,j}$ converges uniformly to $0$ on $\X\setminus B_{2R}(x_0)$ as $j\to +\infty$. Suppose it is not the case and recall that $\hat{u}_{\eta,j} \le u$ because of item (a) of Step 5. Then we can find $0<\delta <M$ and a sequence of points $x_j\notin B_{2R}(x_0)$ such that $\hat{u}_{\eta,j}(x_j) < -\delta$. By definition of $\hat{u}_{\eta,j}$ there must be a chain $\sfc_j = \{q_0^j,\ldots,q_{N_j}^j\} \in \Chain{\frac{1}{j}}$ with $\alpha(\sfc_j) \in A$ and $\omega(\sfc_j) = x_j$ such that $u(\alpha(\sfc_j)) + \int_{\sfc_j}\hat{g}_{\eta,j} \le -\frac{\delta}{2}$. Since $u\restr{\X \setminus B_{R}(x_0)} = 0$, then $\alpha(\sfc_j)$ must belong to $B_R(x_0)$. Let $\sfc_j^s = \{q_0^j,\ldots,q_{i_j}^j\}$ be the subchain of $\sfc_j$ with the property that $i_j$ is the biggest integer $i$ such that $\{q_0^j,\ldots,q_{i}^j\} \subseteq B_{2R - \frac{1}{j}}(x_0)$. By maximality we have that $q_{i_j}^j \notin B_{2R - \frac{2}{j}}(x_0)$. Therefore $u(q_{i_j}^j) = u(\omega(\sfc_j^s)) = 0$. Moreover, we have
    \begin{equation}
        \label{eq:UG_inequality_violated_uniform_convergence}
        u(\alpha(\sfc_j^s)) + \int_{\sfc_j^s}\hat{g}_{\eta,j} \le u(\alpha(\sfc_j)) + \int_{\sfc_j}\hat{g}_{\eta,j} \le -\frac{\delta}{2} = u(\omega(\sfc_j^s)) - \frac{\delta}{2}.
    \end{equation}
    We now claim that $\{\sfc_j^s\}_j$ satisfies the assumptions of Proposition \ref{prop:Sylvester_compactness_chains}. The proof of the upper bound on the length is the same given in Step 8, since $\hat{g}_{\eta,j} \ge \sigma$ on $\sfc_j^s$, so
    $$\ell(\sfc_j^s) \le \sigma^{-1}\int_{\sfc_j^s} \hat{g}_{\eta,j} \le \sigma^{-1}\left(-u(\alpha(\sfc_j^s)) - \frac{\delta}{2}\right) \le \sigma^{-1}M.$$
    For the diameter we have: ${\rm Diam}(\sfc_j^s) \ge \sfd(\alpha(\sfc_j^s), \omega(\sfc_j^s)) \ge R - \frac{2}{j}$, since $\alpha(\sfc_j^s) \in B_R(x_0)$ and $\omega(\sfc_j^s) \notin B_{2R - \frac{2}{j}}(x_0)$. Finally, 
    $$\int_{\sfc_j^s} h_j \le \int_{\sfc_j^s} \hat{g}_{\eta,j} \le M.$$
    Using again Proposition \ref{prop:Sylvester_compactness_chains} and Remark \ref{rmk:compactness_chains_with_lambda}, we conclude that the sequence of chains $\{\sfc_j^s\}_j$ subconverges to a curve $\gamma$ of $\X$. Arguing as in Step 8 we deduce that $g_\eta$ violates the upper gradient inequality on $\gamma$ since the endpoints of $\sfc_j^s$ belong to $A$ and $u$ is continuous on $A$. This is a contradiction.
\end{proof}

\begin{remark}
\label{rmk:UG-UGLip}
    By Proposition \ref{prop:UGeta_subset_UG} we know that every $\varepsilon$-upper gradient with finite values is an upper gradient in the classical sense. Therefore the proof above implies that if $(\X,\sfd)$ is complete then $\relF{curve}(u)$ can be realized as an infimum of the $L^p(\X)$-norms of Lipschitz upper gradients of Lipschitz functions that converge to $u$ in $L^p(\X)$. 
\end{remark}

As announced, the proof of Theorem \ref{theo:density_energy_complete} can be adapted to show that $\HH{p}{\Chain{}, \, \Lip}(\X)$ and $\HH{p}{\Chain{}}(\X)$ are isometric if $(\X,\sfd)$ is complete.
\begin{theorem}
    \label{theo:density-energy-chains-complete}
    Let $(\X,\sfd,\mm)$ be a metric measure space such that $(\X,\sfd)$ is complete. Then 
    \begin{equation}
        \HH{p}{\Chain{},\, \Lip}(\X) = \HH{p}{\Chain{}}(\X)
    \end{equation}
    and
    \begin{equation}
        \|u \|_{\HH{p}{\Chain{},\, \Lip}(\X)}  = \|u \|_{\HH{p}{\Chain{}}(\X)}
    \end{equation}
    for every $u\in L^p(\X)$.
\end{theorem}

Before doing that we need a version of the Leibniz rule for chain upper gradients.
\begin{proposition}[Leibniz rule]
\label{prop:Leibniz-rule-chains}
    Let $u\colon \X \to \R$ be Borel and $\varphi \in {\rm Lip}(\X)$.
    For every $g \in {\rm UG}^\varepsilon(u)$, we have
    \begin{equation}
    \label{eq:varepsilon_upper_gradient_Leibniz}
        \vert u \vert\, {\rm sl}_\varepsilon \varphi + Q_\varepsilon \varphi\, g \in {\rm UG}^\varepsilon(u \varphi),
    \end{equation}
    where $Q_\varepsilon\varphi(x):= \sup_{y\in \overline{B}_\varepsilon(x)} \vert \varphi \vert(y)$. In particular, for every $u \in \mathcal{L}^p(\X)$ it holds that
    \begin{equation}
    \label{eq:Leibniz_rule}
        \F{\Chain{}}(u\varphi) \le \left( \int |u|^p (\lip\,\varphi)^p \,\d \mm \right)^{\frac{1}{p}} + \|\varphi\|_{L^\infty(\X)}\, \F{\Chain{}}(u).
    \end{equation}

\end{proposition}
\begin{proof}
  Let $g \in {\rm UG}^\varepsilon(u)$. We verify \eqref{eq:varepsilon_upper_gradient_Leibniz}. Given $\sfc =\{ q_i\}_{i=0}^N \in \Chain{\varepsilon}$, we compute
  {\allowdisplaybreaks
  \begin{align}
      |&(u \varphi)(\omega(\sfc))- (u \varphi)(\alpha(\sfc))|  \le \sum_{i=0}^{N-1} |(u \varphi)(q_{i+1})- (u \varphi)(q_{i})|\\
      & \le \sum_{i=0}^{N-1} \bigg| u(q_{i+1})\varphi(q_{i+1}) -\frac{1}{2} u(q_{i+1})\varphi(q_{i}) + \frac{1}{2} u(q_{i+1})\varphi(q_{i}) \\
      &\qquad\quad-\frac{1}{2}u(q_{i}) \varphi(q_{i+1})+ \frac{1}{2}u(q_{i}) \varphi(q_{i+1}) - u(q_{i}) \varphi(q_{i}) \bigg| \\
      &\le \sum_{i=0}^{N-1} \left( \frac{1}{2}\vert u(q_{i+1}) \vert{\rm sl}_\varepsilon \varphi(q_{i+1}) + \frac{1}{2}\vert u(q_{i}) \vert{\rm sl}_\varepsilon \varphi(q_{i}) \right)\,\sfd(q_i,q_{i+1})+ \frac{1}{2}\vert \varphi(q_{i+1})\vert \vert u(q_{i+1})-u(q_{i}) \vert\\
      &\qquad\quad+ \frac{1}{2}\vert \varphi(q_{i}) \vert \vert u(q_{i+1})-u(q_{i}) \vert \\
      & \le \int_{\sfc} \vert u \vert\,{\rm sl}_{\varepsilon} \varphi + \sum_{i=0}^{N-1} \frac{1}{4} (\vert \varphi(q_{i})\vert + \vert\varphi(q_{i+1})\vert)(g(q_i)+g(q_{i+1}))\sfd(q_i,q_{i+1})\\
      & \le \int_{\sfc} \vert u \vert\,{\rm sl}_{\varepsilon} \varphi + \sum_{i=0}^{N-1} \frac{(Q_\varepsilon \varphi\, g)(q_i) + (Q_\varepsilon \varphi\, g)(q_{i+1})}{2}  \sfd(q_i,q_{i+1}) = \int_{\sfc} (\vert u \vert \,{\rm sl}_{\varepsilon} \varphi + Q_\varepsilon \varphi\, g).
  \end{align}
  }
  Now, to estimate $\F{\Chain{}}(\varphi\, u)$, we compute
  \begin{equation*}
  \begin{aligned}
      \inf \{ \|h\|_{L^p(\X)}:\, h \in {\rm UG}^\varepsilon(u \varphi)\} &\le \inf\{ \| |u|\,{\rm sl}_{\varepsilon} \varphi + Q_\varepsilon \varphi\, g \|_{L^p(\X)}\,:\, g\in \UG{\varepsilon}(u) \} \\
      &\le \inf\{ \| |u| \,{\rm sl}_{\varepsilon} \varphi \|_{L^p(\X)} + \| Q_\varepsilon \varphi\, g \|_{L^p(\X)}\,:\, g\in \UG{\varepsilon}(u) \}\\
      & \le \left(\int |u|^p\,\sl{\varepsilon}\varphi^p\,\d \mm\right)^{\frac{1}{p}} + \| \varphi \|_{L^\infty(\X)} \inf\{ \|g\|_{L^p(\X)}:\, g \in {\rm UG}^\varepsilon(u) \}\\
   \end{aligned}
  \end{equation*}
  where we used that $\|Q_\varepsilon \varphi\|_{L^\infty(\X)} =\|\varphi\|_{L^\infty(\X)}$.
  By taking the limit as $\varepsilon \to 0$ and using the fact that $\varphi$ is Lipschitz, dominated convergence and the definition of $\F{\Chain{}}(\cdot)$, we get the conclusion.
\end{proof}

\begin{proof}[Proof of Theorem \ref{theo:density-energy-chains-complete}]
    The proof is a modification of the one of Theorem \ref{theo:density_energy_complete}. We highlight what are the differences in each step.
    
    {\bf Step 1.} The proof does not change if we show that for every $u\in L^p(\X)$ we can find a sequence of bounded functions $u_j$ with bounded support such that $u_j \to u$ in $L^p(\X)$ and $\limi_{j\to +\infty}\F{\Chain{}}(u_j) \le \F{\Chain{}}(u)$. We fix a basepoint $x_0 \in \X$ and we consider the $1$-Lipschitz function $\varphi_j(x) = \max\{0,\min\{1,j+1-\sfd(x,x_0)\}\}$. We define $u_j := \min \{j, \max\{-j, \varphi_j u\}\}$. By definition, $u_j$ is bounded and has bounded support. Moreover, every $\varepsilon$-upper gradient of $\varphi_j u$ is also a $\varepsilon$-upper gradient of $u_j$. This implies that $\F{\Chain{}}(u_j) \le \F{\Chain{}}(\varphi_j u)$ for every $j$. Proposition \ref{prop:Leibniz-rule-chains} implies that
    \begin{equation}
        \begin{aligned}
            \limi_{j\to +\infty} \F{\Chain{}}(u_j) \le \limi_{j\to +\infty} \F{\Chain{}}(\varphi_j u) &\le \limi_{j\to +\infty}\left( \int |u|^p (\lip\,\varphi_j)^p \,\d \mm \right)^{\frac{1}{p}} + \|\varphi_j\|_{L^\infty}\, \F{\Chain{}}(u)\\
            &= \limi_{j\to +\infty}\left( \int_{\overline{B}_{j+1}(x_0)\setminus B_j(x_0)} |u|^p \,\d \mm \right)^{\frac{1}{p}} +\F{\Chain{}}(u) = \F{\Chain{}}(u),
        \end{aligned}
    \end{equation}
    where in the last equality we used that $u\in L^p(\X)$.
    
    {\bf Step 2.} It does not change. In particular the claim we have to prove is the following. For every $\varepsilon > 0$, for every $\varepsilon$-upper gradient $g\in \UG{\varepsilon}(u)$ and for every $\eta > 0$ there exists another $\varepsilon$-upper gradient $g_\eta \in \UG{\varepsilon}(u)$ such that 
    \begin{equation}
        \label{eq:approx_with_g_eta_chains}
        \Vert g - g_\eta \Vert_{L^p(\X)} < \eta
    \end{equation} 
    and with the following property. For every $j\in \N$ there exist functions $u_{\eta,j} \colon \X \to \R$ and $g_{\eta,j}\in \LUG{\frac{1}{j}}(u_{\eta,j})$ such that 
    \begin{equation}
        \label{eq:approx_with_u_eta,j_chain}
        \lims_{j\to +\infty} \Vert u_{\eta,j} - u \Vert_{L^p(\X)} \le 2\eta \quad \text{and} \quad \lim_{j\to +\infty} \Vert g_{\eta,j} - g_\eta \Vert_{L^p(\X)} = 0.
    \end{equation}
    Let $R\ge 1$ be such that $u\equiv 0$ on $\X\setminus B_{R}(x_0)$. We recall that it is enough to consider chain upper gradients $g\in \UG{\varepsilon}(u)$ that are lower semicontinuous and such that $g \equiv 0$ on $\X \setminus B_{2R}(x_0)$, by a truncation argument.

    {\bf Step 3.} The definition of $g_\eta$ does not change. Observe that $g_\eta \equiv 0$ on $\X \setminus B_{2R}(x_0)$ and that $g_\eta \in \UG{\varepsilon}(u)$ because $g_\eta \ge g$.
    
    {\bf Step 4.} The definition of $\hat{g}_{\eta,j}$ does not change and satisfies the same properties.
    
    {\bf Step 5.} Here there is a difference in the definition of the set $A$. Since $g_\eta \in L^p(\X)$ then $\mm(\{g_\eta = +\infty \})=0$. By outer regularity of the measure we can find an open set $U_\eta$ containing $\{g_\eta = +\infty \}$ and such that $\mm(U_\eta)^{\frac{1}{p}} < \eta (2M)^{-1}$. Moreover, since $g_\eta \equiv 0$ on $\X \setminus B_{2R}(x_0)$, we can choose $U_\eta$ such that $U_\eta \subseteq B_{2R}(x_0)$.
    Now we change the definition of the set $A$ by setting $A := (K_N \cup (\X \setminus (B_R(x_0))) \setminus U_\eta$. It is still closed and $u\restr{A}$ is still continuous.
    Now, the definition of $\hat{u}_{\eta,j}$ does not change, except for the fact that we use this set $A$. Properties (a)-(e) continue to hold.

    {\bf Step 6.} The definitions of $u_{\eta,j}$ and $g_{\eta,j}$ do not change.
    
    {\bf Step 7.} Here we claim that it is enough to show that $u_{\eta,j}(x)$ converges to $u(x)$ for every $x\in K_N\setminus U_\eta$ as $j \to +\infty$ and that $\hat{u}_{\eta,j}$ converges uniformly to $0$ on $\X\setminus B_{2R}(x_0)$. Indeed if this is true we have
    \begin{equation}
        \begin{aligned}
            &\lim_{j\to +\infty} \int_\X \Vert u - u_{\eta,j} \Vert_{L^p(\X)} \,\d\mm \\
            &\le\lim_{j\to +\infty}\left(\int_{K_N \setminus U_\eta} \vert u - u_{\eta,j} \vert^p \,\d\mm + \int_{(B_{2R}(x_0)\setminus K_N) \cup U_\eta} \vert u - u_{\eta,j} \vert^p \,\d\mm + \int_{(\X \setminus B_{2R}(x_0))} \vert u - u_{\eta,j} \vert^p \,\d\mm \right)^{\frac{1}{p}}\\
            &\le 0 + (2M) \mm(B_{2R}(x_0) \setminus K_N)^{\frac{1}{p}} + (2M)\mm(U_\eta)^{\frac{1}{p}} + 0 \le 2\eta.
        \end{aligned}
    \end{equation}

    {\bf Step 8.} Here we need an additional argument that justifies the different choice of $A$. Indeed, we claim that in any of three cases, the limit curve has its extreme points in $A$.
    
    In cases (1) and (3) this is true because either $\alpha(\sfc_j^e) = \alpha(\sfc_j) \in A$ by definition, or $\sfd(\alpha(\sfc_j^e), \X \setminus B_{2R}(x_0)) \le \frac{2}{j}$, by maximality of $\sfc_j^e$. In the first case the result is trivial because $A$ is closed and $u\restr{A}$ is continuous. In the second case $\alpha(\gamma) = \lim_{j\to +\infty} \alpha(\sfc_j^e) \in \X \setminus B_{2R}(x_0)$, so $\alpha(\gamma)\in A$ since $U_\eta \subseteq B_{2R}(x_0)$. Moreover, $u(\alpha(\gamma)) = 0 = u(\alpha(\sfc_j^e))$ because all these points belong to $\X \setminus B_R(x_0)$. On the other hand $\omega(\gamma) = \lim_{j\to +\infty} \omega(\sfc_j^e) = x \in A$ because $x$ is chosen in $K_N\setminus U_\eta$. Here, we have $u(\omega(\gamma)) = u(x) = u(\omega(\sfc_j^e))$.   
    
    In case (2) we have that $\alpha(\gamma) = \lim_{j\to +\infty} \alpha(\sfc_j^s) \in A$, because $\alpha(\sfc_j^s) \in A$ for every $j$ and $A$ is closed. Moreover, $u(\alpha(\gamma)) = \lim_{j\to +\infty} u(\alpha(\sfc_j^s))$ since $u\restr{A}$ is continuous. On the other hand, either $\omega(\sfc_j^s) = x$ for every $j$, so $\omega(\gamma) = x \in A$ and $u(\omega(\gamma)) = u(x) = u(\omega(\sfc_j^s))$, or $\sfd(\omega(\sfc_j^s), \X \setminus B_{2R}(x_0)) \le \frac{2}{j}$, by maximality of $\sfc_j^s$. Arguing as before, we get that $\omega(\gamma) \in \X \setminus B_{2R}(x_0)$, so it belongs to $A$ and $u(\omega(\gamma)) = 0 = u(\omega(\sfc_j^s))$. 
    
    In every case, the extreme points of $\gamma$ belong to the set $\{g_\eta < +\infty\}$. Hence $g_\eta$ satisfies the upper gradient inequality along $\gamma$ because of Proposition \ref{prop:UGeta_subset_UG}, while the proof shows that this is not the case, giving a contradiction.

    {\bf Step 9.} The proof does not change, using the same modifications we did in Step 8.
\end{proof}
    The combination of Theorems \ref{theo:density_energy_complete}, \ref{theo:density-energy-chains-complete} and Proposition \ref{prop:AGS=curve_complete} gives the proof of Theorem \ref{theo-intro:equivalences-of-all-norms-complete-spaces}.

The next theorem states that the two spaces defined via chains do not change if we take the completion.
\begin{theorem}
\label{theo:natural_isometries_chains}
    Let $(\X,\sfd,\mm)$ be a metric measure space and let $(\bar{\X}, \bar{\sfd}, \bar{\mm})$ be its completion. Then the identity map $\iota\colon L^p(\X) \to L^p(\bar{\X})$ induces isometries between $\HH{p}{\Chain{}, \, \Lip}(\X)$ and $\HH{p}{\Chain{}, \, \Lip}(\bar{\X})$ and between $\HH{p}{\Chain{}}(\X)$ and $\HH{p}{\Chain{}}(\bar{\X})$.
\end{theorem}
\begin{proof}
    Let $u_0\in \mathcal{L}^p(\bar{\X})$ be any representative of $\iota(u)$.
    Since $\varepsilon$-chains in $\X$ are $\varepsilon$-chains in $\bar{\X}$, restrictions of elements in $\UG{\varepsilon}(u_0)$ belong to $\UG{\varepsilon}(u)$ and the property of being Lipschitz is preserved. Thus, $\HH{p}{\Chain{}}(\bar{\X})\subseteq \iota(\HH{p}{\Chain{}}(\X))$ and $\|u\|_{\HH{p}{\Chain{}}(\X)} \le \| \iota(u) \|_{\HH{p}{\Chain{}}(\bar{\X})}$ for every $u \in L^p(\X)$, and similarly $\HH{p}{\Chain{},\,\Lip}(\bar{\X})\subseteq \iota(\HH{p}{\Chain{},\,\Lip}(\X))$ and $\|u\|_{\HH{p}{\Chain{},\,\Lip}(\X)} \le \| \iota(u) \|_{\HH{p}{\Chain{},\,\Lip}(\bar{\X})}$ for every $u \in L^p(\X)$. 
    
    For the other inequality we proceed in two different ways. If $u \in L^p(\X) \cap \Lip(\X)$ and $g\in \LUG{\varepsilon}(u) \cap L^p(\X)$, then we consider the Lipschitz extensions $\bar{u}$, $\bar{g}$ of $u$ and $g$ on $\bar{X}$. We claim that $\bar{g}\in \LUG{\frac{\varepsilon}{2}}(\bar{u}) \cap L^p(\X)$. Indeed, given a chain $\sfc = \{q_i\}_{i=0}^N \in \Chain{\frac{\varepsilon}{2}}(\bar{\X})$ we can find sequence of chains $\sfc_j = \{q_i^j\}_{i=0}^N \in \Chain{\varepsilon}(\X)$ such that $q_i^j$ converges to $q_i$ for every $i=0,\ldots,N$ as $j\to +\infty$. By continuity of $\bar{g}$ and $\bar{u}$ we then have
    $$\bar{u}(\omega(\sfc)) - \bar{u}(\alpha(\sfc)) = \lim_{j\to +\infty} u(\omega(\sfc_j)) - u(\alpha(\sfc_j)) \le \lim_{j\to +\infty}\int_{\sfc_j} g = \int_{\sfc}\bar{g}.$$
    Therefore $\relF{\Chain{},\,\Lip}(\iota(u)) = \relF{\Chain{},\,\Lip}(\bar{u}) \le \F{\Chain{},\,\Lip}(u)$ for every $u\in L^p(\X)$. This is enough to conclude that $\iota(\HH{p}{\Chain{},\,\Lip}(\X)) \subseteq \HH{p}{\Chain{},\,\Lip}(\bar{\X})$ and $\|\iota(u) \|_{\HH{p}{\Chain{},\,\Lip}(\bar{\X})} \le \|u\|_{\HH{p}{\Chain{},\,\Lip}(\X)}$ for every $u \in L^p(\X)$.
    
    For the remaining inequality we recall that in the definition of $\F{\Chain{}}(u)$ the infimum of the $L^p(\X)$-norms can be taken among the $p$-weak $\varepsilon$-upper gradients of $u$. Moreover, since $\mm(\bar{\X}\setminus \X) = 0$, then $\cModepsilon{\varepsilon}{p}(\Chain{}(\bar{\X}\setminus \X)) = 0$. This means that every $g\in \WUG{p}{\varepsilon}(u) \cap L^p(\X)$ belongs also to $\WUG{p}{\varepsilon}(u_0) \cap L^p(\bar{\X})$, where $u_0\in \mathcal{L}^p(\bar{\X})$ is any representative of $\iota(u)$. Hence $\F{\Chain{}}(u_0) \le \F{\Chain{}}(u)$ for every $u\in L^p(\X)$. This implies that $\iota(\HH{p}{\Chain{}}(\X)) \subseteq \HH{p}{\Chain{}}(\bar{\X})$ and $\|\iota(u) \|_{\HH{p}{\Chain{}}(\bar{\X})} \le \|u\|_{\HH{p}{\Chain{}}(\X)}$ for every $u \in L^p(\X)$.
\end{proof}

\begin{reptheorem}{theo-intro:equivalences-of-AGS-Chain}
    Let $(\X,\sfd,\mm)$ be a metric measure space, possibly non-complete. Then
    \begin{equation}
        \HH{p}{\Chain{},\, \Lip}(\X) = \HH{p}{\Chain{}}(\X) = \HH{p}{AGS}(\X)
    \end{equation}
    and
    \begin{equation}
        \|u \|_{\HH{p}{\Chain{},\, \Lip}(\X)}  = \|u \|_{\HH{p}{\Chain{}}(\X)} = \|u \|_{\HH{p}{AGS}(\X)}
    \end{equation}
    for every $u\in L^p(\X)$.
\end{reptheorem}
\begin{proof}
    Direct consequence of Theorems \ref{theo-intro:equivalences-of-all-norms-complete-spaces}, \ref{theo:natural_isometries_chains} and Proposition \ref{prop:AGS_completion}.
\end{proof}
\subsection{Comments on the main results with the $\lambda$-integral}
\label{sec:comments_for_every_lambda}
        If one considers $(\varepsilon, \lambda)$-upper gradients instead of $\varepsilon$-upper gradients, for $\lambda \in [0,1]$, one defines natural variants of the functionals $\F{\Chain{}}$ and $\F{\Chain{},\,\Lip}$, denoted by $\F{\Chain{}}^\lambda$ and $\F{\Chain{},\,\Lip}^\lambda$. Their relaxations are $\relF{\Chain{}}^\lambda$ and $\relF{\Chain{},\,\Lip}^\lambda$. Let us outline some differences.
        
        For $\lambda \neq \frac{1}{2}$, the symmetric property in \eqref{eq:integral_is_symmetric} does not hold, see Remark \ref{rmk:lambda_integral_over_chains}. Therefore, it is not obvious that $\F{\Chain{}}^\lambda$ and $\F{\Chain{},\,\Lip}^\lambda$ satisfy property (c) of Section \ref{subsec:relaxation}. This is due to the fact that it is not true in general that if $g\in \UG{\varepsilon,\lambda}(u)$ then $g \in \UG{\varepsilon,\lambda}(-u)$ when $\lambda \neq \frac12$. However the same proofs of Theorems \ref{theo:density_energy_complete} and \ref{theo:density-energy-chains-complete} show that $\relF{\Chain{},\,\Lip}^\lambda(u) = \relF{curve}(u) = \relF{\Chain{}}^\lambda(u)$ for every $u\in L^p(\X)$, if $(\X,\sfd)$ is complete. In particular, a posteriori, $\relF{\Chain{}}^\lambda$ and $\relF{\Chain{},\Lip}^\lambda$ are seminorms, when $(\X,\sfd)$ is complete and the related Sobolev spaces are denoted by $\HH{p}{\Chain{}, \, \Lip,\lambda}$ and $\HH{p}{\Chain{},\lambda}$. There are some subtleties to be taken into consideration.

        First, the proof of Proposition \ref{prop:Leibniz-rule-chains} holds for every $\lambda \in [0,1]$ under the additional assumption that $\varphi \ge 0$, which is enough to perform Step 1 in the proof of Theorem \ref{theo:density-energy-chains-complete}. One can follows verbatim the same proof removing the absolute values except for the first term in the third line and replacing the first two $\frac12$-factors on the second line with $(1-\lambda)$ and the other two $\frac12$-factors with $\lambda$.
        
        Second, one can use Remark \ref{rmk:section_2_for_lambda} to arrive to a contradiction in Steps 8 and 9.
        
    \vspace{0.7cm}
    Theorem \ref{theo:natural_isometries_chains} holds also for the spaces $\HH{p}{\Chain{}, \, \Lip,\lambda}$ and $\HH{p}{\Chain{},\lambda}$, for every $\lambda \in [0,1]$. For $\HH{p}{\Chain{}, \, \Lip,\lambda}$ the proof is identical. Also for $\HH{p}{\Chain{},\lambda}$, when $\lambda \in (0,1)$, the proof is the same, in view of Remark \ref{rmk:chain-modulus-of-measure-zero-sets}. When $\lambda \in \{0,1\}$ one needs a different argument because it is no more true that $\cModepsilon{\varepsilon,1}{p}(\Chain{}(\bar{\X}\setminus \X) = 0$ and similarly for $\lambda = 0$, see Remark \ref{rmk:chain-modulus-of-measure-zero-sets}. We do it for the case $\lambda = 1$, the other being similar. We extend $u\in \mathcal{L}^p(\X)$ as $\bar{u}(z) := \lim_{r\to 0}\sup_{w\in B_r(z) \cap \X} u(w)$, for $z \in \bar{\X}\setminus \X$. Moreover, we extend every $g\in \UG{\varepsilon,1}(u)$ on $\X$ as $+\infty$ on $\bar{\X}\setminus \X$. We claim that $\bar{g}\in \UG{\frac{\varepsilon}{2},1}(\bar{u})$. Let $\sfc = \{q_i\}_{i=0}^N \in \Chain{\frac{\varepsilon}{2}}(\bar{\X})$. If there exists $i \in \{0,\ldots, N-1\}$ such that $q_i \in \bar{\X}\setminus \X$ then $\leftindex^1{\int}_{\sfc} g = +\infty$ and there is nothing to prove. Otherwise $q_i \in \X$ for every $i\in \{0,\ldots,N-1\}$. For every $w \in B_r(\omega(\sfc)) \cap \X$ we have that $\sfc_w := \{q_0,\ldots,q_{N-1},w\}$ is a $\varepsilon$-chain contained in $\X$ if $r<\frac{\varepsilon}{2}$. For every $0<r<\frac{\varepsilon}{2}$ we have
    \begin{equation}
        \begin{aligned}
            \bar{u}(\omega(\sfc)) - \bar{u}(\alpha(\sfc)) &\le \sup_{w\in B_r(\omega(\sfc)) \cap \X} u(w) - u(\alpha(\sfc)) \le \sup_{w\in B_r(\omega(\sfc)) \cap \X}\leftindex^1{\int}_{\sfc_w} g \\
            &= \sup_{w\in B_r(\omega(\sfc)) \cap \X} \sum_{i=0}^{N-2} g(q_i)\sfd(q_i,q_{i+1}) + g(q_{N-1})\sfd(q_{N-1},w) \\
            &\le \sum_{i=0}^{N-2} g(q_i)\sfd(q_i,q_{i+1}) + g(q_{N-1})(\sfd(q_{N-1},q_N) + r).
        \end{aligned}
    \end{equation}
    By taking $r\to 0$ on the right hand side we get $\bar{u}(\omega(\sfc) - \bar{u}(\alpha(\sfc)) \le \leftindex^1{\int}_{\sfc} \bar{g}$. This is enough to conclude the proof.


    As a consequence, the spaces $\HH{p}{\Chain{}, \lambda}(\X)$ are all isometric, for every possible value of $\lambda \in [0,1]$, even when $\X$ is not complete. The same holds for the spaces $\HH{p}{\Chain{},\,\Lip, \lambda}(\X)$.

\section{Poincar\'{e} inequality}
\label{sec:PI}
We recall the notion of the Poincaré inequality that we now discuss. Let $u,g\colon \X \to \R$ and let $p \ge 1$. We say that the couple $(u,g)$ satisfies a $p$-Poincaré inequality if there exists $\lambda, C \geq 1$ such that
\begin{equation}
    \label{eq:PI-u-g-integral}
    \dashint_{B_{r}(x)} \left\vert u - \dashint_{B_r(x)}u\,\d\mm\right\vert \,\d \mm \leq Cr\left(\dashint_{B_{\lambda r}(x)} g^p\,\d\mm\right)^{\frac{1}{p}}
\end{equation}
for every ball $B_r(x)\subseteq \X$. The following result is a consequence of Theorem \ref{theo-intro:equivalences-of-AGS-Chain} and Proposition \ref{prop:AGS_completion}.
\begin{corollary}
\label{cor:PI_lip=chain}
    Let $(\X,\sfd,\mm)$ be a metric measure space. Then it satisfies a $p$-Poincaré inequality for all couples $(u,\lip\,u)$, where $u\in \Lip(\X)$, if and only if it satisfies a $p$-Poincaré inequality for all couples $(u,g)$, where $u$ is Borel and $g\in \UG{\varepsilon}(u)$ for some $\varepsilon > 0$, with same constants. Moreover, this happens if and only if the metric completion $(\bar{\X},\bar{\sfd},\bar{\mm})$ satisfies a $p$-Poincaré inequality for couples $(u,\lip\,u)$, where $u\in \Lip(\bar{\X})$.
\end{corollary}
\begin{proof}

    Let us consider the first equivalence. The if implication is trivial by Lemma \ref{lemma:slopeeps_epsug} and dominated convergence. The converse implication follows by applying Theorem \ref{theo-intro:equivalences-of-AGS-Chain} to the metric measure space $(B_{\lambda r}(x), \sfd,\mm)$. This gives a sequence $u_j \in {\rm Lip}(B_{\lambda r}(x))$ such that $u_j \to u$ in $L^p(B_{\lambda r}(x))$ and such that 
    $$\limi_{j\to +\infty} \| \lip\, u_j\|_{L^p(B_{\lambda r}(x))} =  \relF{\Chain{}}(u) \le \F{\Chain{}}(u) \le \| g\|_{L^p(B_{\lambda r}(x))}.$$
    Note that $\relF{\Chain{}}(u)$ and $\F{\Chain{}}(u)$ are defined on the metric measure space $(B_{\lambda r}(x), \sfd,\mm)$.
    To apply the hypothesis and conclude, we consider any Lipschitz extension $\tilde{u}_j \in \Lip(\X)$ of $u_j$. The last equivalence follows from Proposition \ref{prop:AGS_completion}.
\end{proof}
\begin{remark}
\label{rem:example_not_equivalence_poincare}
    The previous corollary is not true if we consider the $p$-Poincaré inequality for all couples $(u,g)$ with $u$ Borel and $g\in \UG{}(u)$.
    Indeed the metric measure space $(\X,\sfd,\mm) = ([0,1]\setminus \mathbb{Q},\sfd_e, \mathcal{L}^1 \restr{[0,1]\setminus \mathbb{Q}})$ satisfies the $1$-Poincaré inequality for all couples $(u,\lip\,u)$ with $u\in \Lip(\X)$, because of Corollary \ref{cor:PI_lip=chain}. However, it does not satisfy the $1$-Poincaré inequality for all couples $(u,g)$ with $u$ Borel and $g\in \UG{}(u)$. Indeed $g \equiv 0$ is an upper gradient of every function $u$.
\end{remark}
\begin{remark}
    As a consequence of Remark \ref{rmk:UG-UGLip} we have the following fact. A metric measure space $(\X,\sfd,\mm)$ such that $(\X,\sfd)$ is complete satisfies a $p$-Poincaré inequality with respect to couples $(u,g)$, $g\in \UG{}(u)$, if and only if it satisfies a $p$-Poincaré inequality with respect to couples $(u,g)$, $g\in \UG{}(u)$ with $u \in \Lip(\X)$ and $g\in \UG{}(u) \cap \Lip(\X)$, \emph{with same constants}.
    This result sharpens \cite[Theorem 2]{Kei03}, in which $\mm$ is required to be doubling and whose proof does not say that the constants of the Poincaré inequalities are the same, compare also with \cite[Theorem 8.4.1]{HKST15}.
\end{remark}
\subsection{Pointwise estimates with Riesz potential via chains}
\label{sec:pt_estimates_chains}
When the metric measure space is doubling, the Poincaré inequality is usually expressed in terms of pointwise estimates. We extend these classical results to our setting. In order to do that, we recall that, given a Borel function $u\colon \X \to \R$, a point $x\in \X$ is called a Lebesgue point of $u$ if $u(x) = \lim_{r\to 0} \dashint_{B_r(x)} u\,\d\mm$. The set of Lebesgue points of $u$ is denoted by ${\rm Leb}(u) \subseteq \X$. If $u\in L^p(\X)$ for some $1\le p < +\infty$ then $\mm(\X\setminus{\rm Leb}(u)) = 0$.
\begin{proposition}
\label{prop:PI_equiv_pointwise}
    Let $(\X,\sfd,\mm)$ be a doubling metric measure space. The following properties are quantitatively equivalent:
    \begin{itemize}
        \item[(i)] $\X$ satisfies a $p$-Poincaré inequality for all couples $(u,\lip\,u)$, with $u\in \Lip(\X)$;
        \item[(ii)] there exist $C >0$ and $L\ge 1$ such that for every Borel $u\colon \X \to \R$, for every $x,y\in {\rm Leb}(u)$, for every $\varepsilon > 0$ and for every $g\in \UG{\varepsilon}(u)$, it holds 
        \begin{equation}
            \label{eq:ptwise_PI_chain_UG}
            \vert u(x) - u(y) \vert ^p \le C\sfd(x,y)^{p-1}\int g^p \,\d\mm_{x,y}^L;
        \end{equation}
        \item[(iii)] there exist $C >0$ and $L\ge 1$ such that for every $u\in \Lip(\X)$ and for every $x,y\in \X$ it holds 
        \begin{equation}
            \label{eq:ptwise_PI_lip}
            \vert u(x) - u(y) \vert ^p \le C\sfd(x,y)^{p-1}\int (\lip\,u)^p \,\d\mm_{x,y}^L.
        \end{equation}
    \end{itemize}
\end{proposition}
The measure $\mm_{x,y}^L$ appearing in \eqref{eq:ptwise_PI_chain_UG} and \eqref{eq:ptwise_PI_lip} is defined as $R_{x,y}^L \mm$, where $R_{x,y}^L$ is the $L$-truncated Riesz potential with poles at $x,y$, namely
$$R_{x,y}^L(z) := \left(\frac{\sfd(x,z)}{\mm(B_{\sfd(x,z)}(x))} + \frac{\sfd(y,z)}{\mm(B_{\sfd(y,z)}(y))}\right) \chi_{B_{x,y}^L},$$
where $B_{x,y}^L = B_{L\sfd(x,y)}(x) \cup B_{L\sfd(x,y)}(y)$. At $x,y$ we impose by definition that $R_{x,y}^L(x) = R_{x,y}^L(y) = 0$.
If the measure $\mm$ is doubling
then $\mm_{x,y}^L(\X)$ is a finite measure, more precisely (see \cite[Proposition 2.3]{CaputoCavallucci2024}, whose proof does not use the completeness of $(\X,\sfd)$)
\begin{equation}
    \label{eq:bound_on_m_x,y^L}
    \mm_{x,y}^L(\X) \le 8C_DL\sfd(x,y).
\end{equation}
\begin{proof}[Proof of Proposition \ref{prop:PI_equiv_pointwise}]
    If (i) holds then $(\X,\sfd,\mm)$ satisfies a $p$-Poincaré inequality for all couples $(u,g)$ with $u$ Borel and $g\in \UG{\varepsilon}(u)$ for some $\varepsilon > 0$, by Corollary \ref{cor:PI_lip=chain}. Then (ii) can be proved as in \cite[Theorem 9.5]{Hei01}. Indeed two things are needed: that $x,y$ are Lebesgue points of $u$ and that the space $\X$ is geodesic. However, since $(\X,\sfd,\mm)$ satisfies (i) then also the completion $(\bar{\X},\bar{\sfd},\bar{\mm})$ satisfy (i). Therefore, after a biLipschitz change of the metric $\bar{\sfd}$, we can suppose that $\bar{\sfd}$ is geodesic. In general, $\sfd$ is not geodesic, but by density, there are points of $\X$ arbitrarily close to every point of a fixed geodesic of $\bar{\X}$. Therefore the proof of \cite[Theorem 9.5]{Hei01} can be easily adapted.

    Suppose (ii) holds and let $u\in \Lip(\X)$. We have ${\rm Leb}(u) = \X$, since $u$ is continuous. Moreover, by Lemma \ref{lemma:slopeeps_epsug}, $\sl{\varepsilon}u \in \UG{\varepsilon}(u)$ for every $\varepsilon > 0$. Therefore \eqref{eq:ptwise_PI_chain_UG} implies that $\vert u(x) - u(y) \vert^p \le C\sfd(x,y)^{p-1}\int (\sl{\varepsilon}u)^p \,\d\mm_{x,y}^L$ for every $x,y\in \X$ and every $\varepsilon > 0$. By dominated convergence, thanks to the fact that $\mm_{x,y}^L(\X) < +\infty$ by \eqref{eq:bound_on_m_x,y^L}, we get \eqref{eq:ptwise_PI_lip}, so (iii) holds.

    If (iii) holds then (i) holds by a combination of \cite[Theorem 9.5]{Hei01} and \cite[Theorem 8.1.7]{HKST15}.
\end{proof}
\begin{remark}
    The pointwise estimate of item (ii) cannot hold at every point. Indeed let $(\X,\sfd,\mm) = (\R,\sfd_e,\mathcal{L})$, $u=\chi_{0}$ and $g=+\infty\cdot\chi_0 \in \UG{\varepsilon}(u)$, for every $\varepsilon > 0$. If $x=0$ and $y = 1$ then $+\infty = \vert u(x)-u(y)\vert$, while $\int g^p\,\d\mm_{x,y}^L = 0$. This is in contrast with the case of upper gradients along curves when $(\X,\sfd)$ is complete. Indeed, even if a priori one gets the pointwise estimate with respect to every upper gradient only on the Lebesgue points of $u$, see \cite[Theorem 9.5]{Hei01} and \cite[Theorem 8.1.7]{HKST15}, in \cite[Theorem A.3]{CaputoCavallucci2024II} we showed that it actually holds everywhere.
\end{remark}
The pointwise estimate in item (ii) of Proposition \ref{prop:PI_equiv_pointwise} holds everywhere for chain upper gradients that assume finite values at $x$ and $y$. This is established in the next result.
\begin{proposition}
\label{prop:pt_PI_equivalent_conditions_A_p_connectedness}
    Let $(\X,\sfd,\mm)$ be a doubling metric measure space. Let $x,y\in \X$. The following properties are quantitatively equivalent:
    \begin{itemize}
        \item[(i)] there exist $C >0$ and $L\ge 1$ such that \eqref{eq:ptwise_PI_lip} holds for every $u\in \Lip(\X)$;
        \item[(ii)] there exist $C>0$ and $L\ge 1$ such that for every Borel function $g\colon \X\to [0,+\infty]$ with $g(x),g(y) < +\infty$ it holds
       \begin{equation}
        \label{eq:A_p-connectedness-limit}
            \lim_{\varepsilon \to 0} \inf_{\substack{\sfc \in \Chain{\varepsilon}_{x,y}\\ \ell(\sfc) \le C\sfd(x,y)}} \left(\int_{\sfc} g\right)^p \le C\sfd(x,y)^{p-1}\int g^p\,\d\mm_{x,y}^L.
        \end{equation}
        \item[(iii)] there exist $C >0$ and $L\ge 1$ such that \eqref{eq:ptwise_PI_chain_UG} holds for every Borel $u\colon \X \to \R$, for every $\varepsilon > 0$ and for every $g\in \UG{\varepsilon}(u)$ such that $g(x),g(y) < +\infty$;
    \end{itemize}
\end{proposition}
\begin{proof}[Proof of Proposition \ref{prop:pt_PI_equivalent_conditions_A_p_connectedness}]
    Let $\Y_\varepsilon$ be the $\varepsilon$-chain connected component of $\X$ containing $x$.
    
    \textbf{Reduction to the case $y \in \Y_\varepsilon$ for every $\varepsilon > 0$.} We assume that $y$ does not belong to $\Y_{\bar{\varepsilon}}$ for some $\bar{\varepsilon} > 0$. If this is the case, none of the conditions (i)-(iii) hold and thus the theorem holds trivially true. Indeed if $y\notin \Y_{\bar{\varepsilon}}$ then the function $u \equiv 0$ on $\Y_{\bar{\varepsilon}}$ and $u\equiv 1$ on $(X\setminus \Y_{\bar{\varepsilon}})$ is Lipschitz, since $\sfd(\Y_{\bar{\varepsilon}}, \X \setminus \Y_{\bar{\varepsilon}}) > \bar{\varepsilon}$ by \eqref{eq:chain_connected_components_positive_distance}, has $\lip\,u \equiv 0$ and contradicts (i). The same function contradicts (iii) since $g\equiv 0 \in \UG{\bar{\varepsilon}}(u)$.
    Condition (ii) does not hold, because it implies that $x,y$ belong to the same ${\bar{\varepsilon}}$-chain connected component. Therefore, we can assume that $x,y \in \Y$.

    \textbf{Reduction to the case $(\X,\sfd)$ is complete.} We claim that properties (i) to (iii) hold on $(\X,\sfd,\mm)$ if and only if they hold on the completion $(\bar{\X}, \bar{\sfd}, \bar{\mm})$. If they hold on $\X$ then they clearly hold on $\bar{\X}$. The vice versa in the case of (i) is given arguing as in Proposition \ref{prop:AGS_completion}.
    
    Regarding (ii) we argue as follows: given a Borel function $g\colon \X \to [0,+\infty]$ such that $g(x),g(y) < +\infty$ we extend it to $\bar{g} \colon \bar{\X} \to [0,+\infty]$ setting $\bar{g} \equiv +\infty$ on $\bar{\X} \setminus \X$. If $\int g^p\,\d\mm_{x,y}^L = +\infty$ there is nothing to prove. Otherwise, condition (ii) on $\bar{\X}$ gives that for every $\eta >0$ there exists $\varepsilon > 0$ and a chain $c_\eta \in \Chain{\varepsilon}_{x,y}$ such that $\ell(c_\eta) \le C\sfd(x,y)$ and $\left(\int_{\sfc_\eta} \bar{g} \right)^p \le C\sfd(x,y)^{p-1}\int \bar{g}^p\,\d\mm_{x,y}^L + \eta < +\infty$. Then $c_\eta \subseteq \X$ for every $\eta > 0$, since if not we would have $\int_{\sfc_\eta} \bar{g} = +\infty$. Therefore \eqref{eq:A_p-connectedness-limit} holds for $g$ on $\X$.
    
    Suppose now that (iii) holds on $\bar{\X}$. For every $u \colon \X \to \R$ Borel and every $g\in \UG{\varepsilon}(u)$ such that $g(x),g(y) < +\infty$ we consider the same extension $\bar{g}$ of $g$ to $\bar{\X}$ as above and we extend $u$ to $\bar{\X}$ by setting $u \equiv 0$ on $\bar{\X} \setminus \X$. 
    We observe that $\bar{g} \in \UG{\varepsilon}(\bar{u})$ on $\bar{\X}$. Indeed, for every $\sfc \in \Chain{\varepsilon}(\bar{\X})$ there are two possibilities: either $\sfc \in \Chain{\varepsilon}(\bar{\X}\setminus \X)$ and in that case $\int_\sfc \bar{g} = +\infty$, or $\sfc \subseteq \X$ and one gets $\vert \bar{u}(\omega(\sfc)) - \bar{u}(\alpha(\sfc))\vert = \vert u(\omega(\sfc)) - u(\alpha(\sfc))\vert \le \int_\sfc g = \int_{\sfc}\bar{g}$. Condition (iii) applied to the couple $(\bar{u}, \bar{g})$ on $\bar{\X}$ gives
    \begin{equation}
        \begin{aligned}
            \vert u(x) - u(y) \vert^p = \vert \bar{u}(x) - \bar{u}(y) \vert^p \le C\bar{\sfd}(x,y)^{p-1}\int_{\bar{\X}} \bar{g}^p \,\d\mm_{x,y}^L = C\sfd(x,y)^{p-1}\int_{\X} g^p \,\d\mm_{x,y}^L,
        \end{aligned}
    \end{equation}
    so condition (iii) holds also for the couple $(u,g)$ on $\X$. The discussion above allows us to prove the equivalences in the case $(\X,\sfd)$ is complete. 

    From now on, we assume that $(\X,\sfd)$ is complete and $x,y$ belongs to the same $\varepsilon$-chain connected component for every $\varepsilon>0$.
    
    \textbf{Main argument.} 
    We introduce the additional conditions
    \begin{itemize}
        \item[(i)$_{\Lip_\loc}$] there exist $C>0$, $L\ge 1$ such that \eqref{eq:ptwise_PI_chain_UG} holds for every $u\in \Lip_\loc(\X)$, every $\varepsilon > 0$ and every $g\in \LUG{\varepsilon}(u)$ bounded;
        \item[(i)$_{\Lip_\loc}'$] there exist $C>0$, $L\ge 1$ such that for every $\varepsilon > 0$ the inequality \eqref{eq:ptwise_PI_chain_UG} holds for every $u\in \Lip_\loc(\Y_\varepsilon)$ and every $g \colon \Y_\varepsilon \to [0,+\infty]$ with $g \in \LUG{\varepsilon}(u)$ and bounded;
        \item[(ii)$_{\Lip}$] there exist $C>0$ and $L\ge 1$ such that for every $g\in \Lip(\X)$, $g\ge 0$ and bounded, and for every $\varepsilon > 0$ it holds that
       \begin{equation}
        \label{eq:A_p-connectedness-limit_Lipschitz}
            \inf_{\substack{\sfc \in \Chain{\varepsilon}_{x,y}\\ \ell(\sfc) \le C\sfd(x,y)}} \left(\int_{\sfc} g\right)^p \le C\sfd(x,y)^{p-1}\int g^p\,\d\mm_{x,y}^L.
        \end{equation}
        \item[(ii)$'_{\Lip}$] there exist $C>0$ and $L\ge 1$ such that for every $\varepsilon > 0$ and every $g\in \Lip(\Y_\varepsilon)$, $g\ge 0$ and bounded there exists $\sfc \in \Chain{\varepsilon}_{x,y}(\Y_\varepsilon)$ such that $\ell(\sfc) \le C\sfd(x,y)$ and
       \begin{equation}
        \label{eq:A_p-connectedness-Y-varepsilon}
            \left(\int_{\sfc} g\right)^p \le C\sfd(x,y)^{p-1}\int_{\Y_\varepsilon} g^p\,\d\mm_{x,y}^L.
        \end{equation}
        
    \end{itemize}

    We prove the theorem by showing the following chains of implications: (i) $\Rightarrow$ (i)$_{\Lip_\loc}$ $\Rightarrow$ (i)$_{\Lip_\loc}'$ $\Rightarrow$ (ii)$_{\Lip}'$ $\Rightarrow$ (ii)$_\Lip$ $\Rightarrow$ (ii) $\Rightarrow$ (iii) $\Rightarrow$ (i).
    
    Suppose (i) holds, let $u \in \Lip_\loc(\X)$ and $g\in \LUG{\varepsilon}(u)$. Lemma \ref{lemma:slopeeps_epsug} says that $\lip\,u \le g$. The function $u$ is Lipschitz on the compact set $\overline{B}_{x,y}^L$, because of \cite[Theorem 4.2]{BeerGarrido2015}. By McShane Extension Theorem we can find a Lipschitz map $\hat{u} \in \Lip(\X)$ which coincides with $u$ on $\overline{B}_{x,y}^L$. Applying (i), and using that $\lip\,\hat{u} = \lip\,u$ $\mm_{x,y}^L$-a.e., we get
    \begin{equation}
        \begin{aligned}
            \vert u(x) - u(y) \vert^p = \vert \hat{u}(x) - \hat{u}(y) \vert^p \le C\sfd(x,y)^{p-1}\int (\lip\,\hat{u})^p\,\d\mm_{x,y}^L &= C\sfd(x,y)^{p-1}\int (\lip\,{u})^p\,\d\mm_{x,y}^L \\&\le C\sfd(x,y)^{p-1}\int g^p\,\d\mm_{x,y}^L,
        \end{aligned}
    \end{equation}
    which proves (i)$_{\Lip_\loc}$.

    We prove that if (i)$_{\Lip_\loc}$ holds with constants $C,L$ then (i)$'_{\Lip_\loc}$ holds with same constants. We fix $\varepsilon>0$ and consider 
    $u\in \Lip_\loc(\Y_\varepsilon)$ and $g \colon \Y_\varepsilon \to [0,+\infty]$ with $g\in \LUG{\varepsilon}(u)$. We define $\hat{u}$, $\hat{g}$ as function on $\X$ by extending $u$ and $g$ to be constantly equal to zero on $\X \setminus \Y_\varepsilon$, respectively. Using \eqref{eq:chain_connected_components_positive_distance} one gets that $\hat{u} \in \Lip_\loc(\X)$ and $\hat{g} \in \Lip(\X)$ with Lipschitz constant $\max\{ \frac{\sup g}{\varepsilon}, {\rm Lip}(g)\}$. Moreover, $\hat{g} \in \LUG{\varepsilon}(\hat{u})$. By applying (i)$_{\Lip_\loc}$, and using the fact that $\hat{g} \equiv 0$ on $\X \setminus \Y_\varepsilon$, we conclude.

    To prove that (i)$'_{\Lip_\loc} \Rightarrow$ (ii)$'_{\Lip}$, we adapt the argument of \cite[Theorem 1.5]{ErikssonBique2019II}. We fix $\varepsilon >0$ and a bounded $g\in \Lip(\Y_\varepsilon)$, with $g\ge 0$. We claim that (ii)$'_\Lip$ holds with $C' = 2^{p+4}CC_DL$ and $L'= \max\{L,C'\}$.    
    For every $\delta > 0$ such that $\int_{\Y_\varepsilon} g^p\,\d\mm_{x,y}^{L} < \delta^p \mm_{x,y}^{L}(\Y_\varepsilon)$ we consider the function
    $$u_\delta\colon \Y_\varepsilon \to [0,+\infty),\quad u_\delta(z) = \inf\left\{ \int_\sfc (g + \delta)\,:\, \sfc \in \Chain{\varepsilon}(\Y_\varepsilon), \alpha(\sfc) = x,\,\omega(\sfc)=z\right\}.$$
    With usual techniques it is possible to show that $u_\delta$ is $(\sup_\X g + \delta)$-Lipschitz up to scale $\varepsilon$, i.e. if $\sfd(z,w) \le \varepsilon$ then $\vert u_\delta(z) - u_\delta(w)\vert \le (\sup_\X g + \delta) \sfd(z,w)$. 
    Moreover, $(g + \delta)\in \LUG{\varepsilon}(u_\delta)$. The condition (i)$'_{\Lip_\loc}$ applied to the couple $(u_\delta,g + \delta)$ implies that 
    $$u_\delta(y)^p \le C\sfd(x,y)^{p-1}\int_{\Y_\varepsilon} (g + \delta)^p\,\d\mm_{x,y}^L.$$ By definition of $u_\delta$ we have that $\int_{\Y_\varepsilon} (g + \delta)^p\,\d\mm_{x,y}^{L} > 0$ and that we can find chains $\sfc_\delta \in \Chain{\varepsilon}_{x,y}(\Y_\varepsilon)$ such that 
    $$\left(\int_{\sfc_\delta} (g + \delta)\right)^p \le 2C\sfd(x,y)^{p-1}\int_{\Y_\varepsilon} (g + \delta)^p\,\d\mm_{x,y}^L.$$
    Moreover, using that $(g(x)+\delta)^p \le 2^{p-1}(g(x)^p + \delta^p)$ for all $x \in \Y_\varepsilon$, we have
    \begin{equation}
        \label{eq:chains_of_bounded_length}
        \begin{aligned}
            \delta^p\ell(\sfc_\delta)^p \le \left(\int_{\sfc_\delta} (g + \delta)\right)^p &\le 2^{p}C\sfd(x,y)^{p-1}\left(\int_{\Y_{\varepsilon}} g^p \,\d\mm_{x,y}^L + \delta^p\mm_{x,y}^L(\Y_\varepsilon)\right) \\
            &\le 2^{p + 1}C\sfd(x,y)^{p-1}\left( \delta^p\mm_{x,y}^L(\Y_\varepsilon)\right)\\            &\stackrel{\eqref{eq:bound_on_m_x,y^L}}{\le} 2^{p+1}C\cdot 8C_DL\sfd(x,y)^p \delta^p.
        \end{aligned}
    \end{equation}
    This implies that $\ell(\sfc_\delta) \le C'\sfd(x,y)$ for every $\delta$. 
    
    If $\int_{\Y_\varepsilon} g^p\,\d\mm_{x,y}^{L'} > 0$ then, by choosing $\delta$ such that $\delta^p\mm_{x,y}^L(\Y_\varepsilon) < 2\int_{\Y_\varepsilon} g^p \,\d\mm_{x,y}^{L'}$, we have that the chain $\sfc_\delta$ satisfies
    $$\left(\int_{\sfc_\delta} g\right)^p \le 3\cdot2^{p+1}C\sfd(x,y)^{p-1}\int_{\Y_\varepsilon} g^p\,\d\mm_{x,y}^{L'} \le C'\sfd(x,y)^{p-1}\int_{\Y_\varepsilon} g^p\,\d\mm_{x,y}^{L'}$$
    and one can take $\sfc = \sfc_\delta$ to get the thesis.
    
    If $\int_{\Y_\varepsilon} g^p\,\d\mm_{x,y}^{L'} = 0$, so $g\equiv 0$ on $B_{x,y}^{L'} \cap \Y_\varepsilon$ since $g$ is Lipschitz, we argue as follows. By \eqref{eq:chains_of_bounded_length} we have the existence of a chain $\sfc \in \Chain{\varepsilon}_{x,y}(\Y_\varepsilon)$ with $\ell(\sfc) \le C'\sfd(x,y)$, so $\sfc \subseteq B_{x,y}^{C'} \cap \Y_\varepsilon \subseteq B_{x,y}^{L'} \cap \Y_\varepsilon$. Therefore $\int_{\sfc} g = 0$ and (ii)$'_\Lip$ holds also in this case.

    We assume (ii)$'_\Lip$ and we show that (ii)$_\Lip$ holds with same constants. Let $g\in \Lip(\X)$, $g\ge 0$ and bounded. For every $\varepsilon > 0$, the restriction $g\restr{\Y_\varepsilon}$ satisfies the assumptions of (ii)$'_\Lip$. Therefore, there exists a chain $\sfc_\varepsilon \in \Chain{\varepsilon}(\X)$ such that
    \begin{equation}
            \left(\int_{\sfc_\varepsilon} g\right)^p \le C\sfd(x,y)^{p-1}\int_{\Y_\varepsilon} g^p\,\d\mm_{x,y}^L \le C\sfd(x,y)^{p-1}\int_{\X} g^p\,\d\mm_{x,y}^L.
    \end{equation}
    By taking the infimum on the left over $\{ \sfc \in \Chain{\varepsilon}_{x,y},\, \ell(\sfc) \le C\sfd(x,y)\}$ and the limit as $\varepsilon$ goes to $0$, we conclude.

    Suppose (ii)$_\Lip$ holds. Since $\X$ is complete we can use Proposition \ref{prop:Sylvester_compactness_chains} and Lemma \ref{lemma:lower-semicontinuity-integral-chains-convergence} to show that for every $g\in \Lip(\X)$, $g\ge 0$ and bounded, there exists a curve $\gamma \in \Gamma_{x,y}$ with $\ell(\gamma) \le C\sfd(x,y)$ and such that $\left(\int_\gamma g\right)^p \le C \sfd(x,y)^{p-1}\int g^p \,\d\mm_{x,y}^L$. This is condition (iii) of \cite[Theorem A.3]{CaputoCavallucci2024II} which is equivalent to the following: for every Borel $g\colon \X \to [0,+\infty]$ there exists $\gamma \in \Gamma_{x,y}$ such that $\ell(\gamma) \le C\sfd(x,y)$ and $\left(\int_\gamma g\right)^p \le C \sfd(x,y)^{p-1}\int g^p \,\d\mm_{x,y}^L$. If moreover $g(x), g(y) < +\infty$ we can use Proposition \ref{prop:approx_integral_curve_with_subchains} to find chains $\sfc_j \in \Chain{\frac{1}{j}}_{x,y}$ such that $\int_\gamma g \ge \lims_{j\to +\infty} \int_{\sfc_j}g$.     
    This implies (ii).
    
    If (ii) holds, \eqref{eq:A_p-connectedness-limit} and the chain upper gradient inequality gives
    $$\vert u(x) - u(y) \vert^p \le \lim_{\varepsilon \to 0} \inf_{\substack{\sfc \in \Chain{\varepsilon}_{x,y}\\ \ell(\sfc) \le C\sfd(x,y)}} \left(\int_{\sfc} g\right)^p \le C\sfd(x,y)^{p-1}\int g^p\,\d\mm_{x,y}^L,$$
    which is (iii).

    Finally, if (iii) holds and $u \in {\rm Lip}(\X)$, \eqref{eq:ptwise_PI_chain_UG} applied to $\sl{\varepsilon}u \in \UG{\varepsilon}(u)$ by Lemma \ref{lemma:slopeeps_epsug}, gives
    $$\vert u(x) - u(y) \vert^p \le C\sfd(x,y)^{p-1} \int (\sl{\varepsilon}u)^p\,\d\mm_{x,y}^L.$$
    Therefore (i) follows by applying dominated convergence.
\end{proof}
\begin{remark}
\label{rem:drop_quasiconvexity}
The proof above shows that the conditions of Proposition \ref{prop:pt_PI_equivalent_conditions_A_p_connectedness} are equivalent to
\begin{itemize}
    \item[(iii)$_\UG{}$] there exist $C >0$ and $L\ge 1$ such that for every Borel $u\colon \X \to \R$ and for every $g\in \UG{}(u)$ it holds $\vert u(x) - u(y) \vert ^p \le C\sfd(x,y)^{p-1}\int g^p \,\d\mm_{x,y}^L$,
\end{itemize}
in case $(\X,\sfd)$ is complete.
    Indeed if (iii)$_{\UG{}}$ holds then (i) holds because $\lip\,u \in \UG{}(u)$ if $u\in \Lip(\X)$. Vice versa, in the proof we showed that the conditions of Proposition \ref{prop:pt_PI_equivalent_conditions_A_p_connectedness} are also equivalent to condition (iii) of \cite[Theorem A.3]{CaputoCavallucci2024II}, which is in turn equivalent to (iii)$_{\UG{}}$ by the same \cite[Theorem A.3]{CaputoCavallucci2024II}.
    
    This generalizes the result of \cite[Theorem A.3]{CaputoCavallucci2024II}, in which the implication from item (i) of Proposition \ref{prop:pt_PI_equivalent_conditions_A_p_connectedness} and (iii)$_\UG{}$ is proved under the additional assumption of local quasiconvexity of the space. By using chains as we did, we are able to remove this assumption and to show the equivalence of pointwise estimates in general.

    The reason behind this improvement is the following. A standard technique, that we used also in the proof of Proposition \ref{prop:pt_PI_equivalent_conditions_A_p_connectedness}, consists in taking a bounded function $g$ and in associating the functions
    $$u_{\text{curve}}(z) := \inf\left\{ \int_\gamma g \,:\, \gamma \text{ curve}, \alpha(\gamma)=x,\,\omega(\gamma)=z \right\},$$
    $$u_{\varepsilon\text{-chain}}(z) := \inf\left\{ \int_\sfc g \,:\, \sfc\in \Chain{\varepsilon}, \alpha(\sfc)=x,\,\omega(\sfc)=z\right\}.$$
    As showed in the proof of Proposition \ref{prop:pt_PI_equivalent_conditions_A_p_connectedness}, using the discrete nature of chains, it is possible to say that $u_{\varepsilon\text{-chain}}$ is locally Lipschitz, actually Lipschitz up to scale $\varepsilon$, on the $\varepsilon$-chain connected component containing $x$. On the other hand, in order to say that $u_{\text{curve}}$ is locally Lipschitz one needs some connectivity property of the metric space $\X$, as the local quasiconvexity.
\end{remark}

\subsection{Keith's characterization via chains}
\label{sec:keith_chain}

The Poincaré inequality with upper gradients can be characterized via modulus estimates, see \cite[Theorem 2]{Kei03} and \cite[Proposition A.1]{CaputoCavallucci2024II}. We will show a similar statement for chains. Let $(\X,\sfd,\mm)$ be a metric measure space.

Let $\mathcal{F}$ be a family of Borel functions on $\X$.
Given a family of chains $\sfC$, the $(\varepsilon,p)$-modulus of $\sfC$ with respect to $\mathcal{F}$ is defined as
$$\cModepsilon{p}{\varepsilon}(\sfC, \mathcal{F}, \mm) := \inf \left\{  \int \rho^p\,\d\mm\,:\, \rho \in {\rm Adm}^\varepsilon(\sfC) \cap \mathcal{F}\right\},$$
where we recall that 
$${\rm Adm}^{\varepsilon}({\sfC})=\left\{ \rho \ge 0\,:\, \rho \text{ Borel, }\int_{\sfc} \rho \ge 1 \text{ for every } \sfc \in {\sfC} \cap \Chain{\varepsilon} \right\}.$$

If $\mathcal{F}$ is closed under finite sums, the same proof of Proposition \ref{prop:outer_measure} shows that the assignment $\sfC \mapsto \cModepsilon{\varepsilon}{p}(\sfC, \mathcal{F}, \mm)$  satisfies 
\begin{equation}
    \label{eq:chain_modulus_subfamily_outer_measure}
    \cModepsilon{\varepsilon}{p}\left(\bigcup_{i\in I}\sfC_i, \mathcal{F} , \mm\right) \le \sum_{i\in I} \cModepsilon{\varepsilon}{p}\left(\sfC_i, \mathcal{F}, \mm\right)
\end{equation}
for a finite set of indices $I$. In general it does not define an outer measure. Notice that $ (0,+\infty) \ni \varepsilon \mapsto \cModepsilon{p}{\varepsilon}(\sfC, \mathcal{F}, \mm) \in [0,+\infty]$ is non-decreasing.
%
%


We also recall the definition of the $p$-modulus of a family of curves. Let $\Gamma$ be a family of curves and let $\mathcal{F}$ be a family of Borel functions. Then
$$\Mod_p(\Gamma, \mathcal{F}, \mm) := \inf\left\{ \int \rho^p \,\d\mm\,:\, \rho \in {\rm Adm}(\Gamma) \cap \mathcal{F}\right\},$$
where ${\rm Adm}(\Gamma) := \{ \rho \colon \X \to [0,+\infty] \, : \, \int_\gamma \rho \ge 1 \text{ for all } \gamma \in \Gamma\}$. If $\mathcal{F}$ is the class of all Borel functions, then we simply write $\Mod_p(\Gamma, \mm)$. 

For the next result we define $\mathcal{F}_{x,y} := \{g \colon \X \to \R \,:\, g \text{ Borel and } g(x),g(y) < +\infty\}$, for $x,y\in \X$. 
\begin{proposition}
\label{prop:modulus_chain_modulus_limit}
    Let $(\X,\sfd,\mm)$ be a doubling metric measure space such that $(\X,\sfd)$ is complete, $x,y\in \X$ and $L\ge 1$. Then
    $$\Mod_p(\Gamma_{x,y}, \mm_{x,y}^L) = \lim_{\varepsilon \to 0} \cModepsilon{\varepsilon}{p}(\Chain{}_{x,y}, \mathcal{F}_{x,y} ,\mm_{x,y}^L) = \lim_{\varepsilon \to 0} \cModepsilon{\varepsilon}{p}(\Chain{}_{x,y}, \Lip(\X), \mm_{x,y}^L).$$
\end{proposition}
\begin{proof}
    During this proof we use the notation $\Chain{\varepsilon,\Lambda}_{x,y}$ to denote the family of chains $\sfc \in \Chain{\varepsilon}_{x,y}$ such that $\ell(\sfc) \le \Lambda\sfd(x,y)$. 
    In the same way, $\Gamma_{x,y}^\Lambda$ denotes the family of rectifiable curves with $\alpha(\gamma) = x$, $\omega(\gamma) = y$ and $\ell(\gamma) \le \Lambda \sfd(x,y)$. 
    We want to show
    \begin{equation}
        \label{eq:modulus_comparison_1}
        \lim_{\varepsilon \to 0} \cModepsilon{\varepsilon}{p}(\Chain{}_{x,y}, \Lip(\X),\mm_{x,y}^L) \le \Mod_p(\Gamma_{x,y}, \mm_{x,y}^L).
    \end{equation}
    We fix $\delta > 0$. The same proof of \cite[Lemma A.2]{CaputoCavallucci2024}, together with \eqref{eq:chain_modulus_subfamily_outer_measure}, shows that we can find $\Lambda \ge 1$ such that $\cModepsilon{\varepsilon}{p}(\Chain{}_{x,y}, \Lip(\X) ,\mm_{x,y}^L) \le \cModepsilon{\varepsilon}{p}(\Chain{\varepsilon, \Lambda}_{x,y}, \Lip(\X) ,\mm_{x,y}^L) + \delta$. 
    We consider the compact family of curves $\Gamma_{x,y}^\Lambda$.
    By \cite[Proposition 6]{Kei03} we have $\Mod_p(\Gamma_{x,y}^\Lambda, \mm_{x,y}^L) = \Mod_p(\Gamma_{x,y}^\Lambda, \Lip(\X), \mm_{x,y}^L)$. 
    Let $\rho \in {\rm Adm}(\Gamma_{x,y}^\Lambda) \cap \Lip(\X)$.  
    We claim that 
    \begin{equation}
        \label{eq:claim_modulus_limit}
        \lim_{\varepsilon \to 0} \inf_{\sfc \in \Chain{\varepsilon, \Lambda}_{x,y}} \int_\sfc \rho \ge 1.
    \end{equation}
    Assuming the claim holds true, this implies that for every $\eta > 0$ there exists $\varepsilon_\eta >0$ such that if $\varepsilon \le \varepsilon_\eta$ then $(1+\eta)\rho \in {\rm Adm}^\varepsilon(\Chain{\varepsilon,\Lambda}_{x,y})$. Hence
    $$\lim_{\varepsilon \to 0} \cModepsilon{\varepsilon}{p}(\Chain{\varepsilon, \Lambda}_{x,y}, \Lip(\X) ,\mm_{x,y}^L) \le \lim_{\eta \to 0} \int (1+\eta)\rho^p\,\d\mm_{x,y}^L = \int \rho^p\,\d\mm_{x,y}^L.$$
    By the arbitrariness of $\rho$, we would get 
    $$\lim_{\varepsilon \to 0} \cModepsilon{\varepsilon}{p}(\Chain{}_{x,y}, \Lip(\X) ,\mm_{x,y}^L) \le \Mod_p(\Gamma_{x,y}^\Lambda, \mm_{x,y}^L) + \delta \le \Mod_p(\Gamma_{x,y}, \mm_{x,y}^L) + \delta$$
    By taking $\delta \to 0$ we would conclude that 
    $$\lim_{\varepsilon \to 0} \cModepsilon{\varepsilon}{p}(\Chain{}_{x,y}, \Lip(\X) ,\mm_{x,y}^L) \le \Mod_p(\Gamma_{x,y}, \mm_{x,y}^L).$$
    We prove the claim. Suppose \eqref{eq:claim_modulus_limit} is not true. Then there exists $\eta > 0$ and chains $\sfc_\varepsilon \in \Chain{\varepsilon,\Lambda}_{x,y}$ such that $\int_{\sfc_{\varepsilon}} \rho < 1-\eta$, for every $\varepsilon$ sufficiently small. We are in position to apply Proposition \ref{prop:Sylvester_compactness_chains} in order to find a curve $\gamma \in \Gamma_{x,y}^\Lambda$ such that $\sfc_{\varepsilon}$ subconverges to $\gamma$ as $\varepsilon \to 0$. By Lemma \ref{lemma:lower-semicontinuity-integral-chains-convergence} we have
    $$\int_\gamma \rho \le \limi_{\varepsilon \to 0} \int_{\sfc_\varepsilon} \rho < 1-\eta,$$
    which is a contradiction to the fact that $\rho \in {\rm Adm}(\Gamma_{x,y}^\Lambda)$. This concludes the proof of the claim.

    The inequality
    $$\cModepsilon{\varepsilon}{p}(\Chain{}_{x,y}, \mathcal{F}_{x,y} ,\mm_{x,y}^L) \le \cModepsilon{\varepsilon}{p}(\Chain{}_{x,y}, \Lip(\X), \mm_{x,y}^L)$$
    holds trivially for every $\varepsilon > 0$. It remains to show that
    \begin{equation}
        \label{eq:modulus_estimate_2}
        \Mod_p(\Gamma_{x,y}, \mm_{x,y}^L) \le \lim_{\varepsilon \to 0} \cModepsilon{\varepsilon}{p}(\Chain{}_{x,y}, \mathcal{F}_{x,y} ,\mm_{x,y}^L).
    \end{equation}
    Let $\rho \in {\rm Adm}^\varepsilon(\Chain{}_{x,y}) \cap \mathcal{F}_{x,y}$. We claim that $\rho \in {\rm Adm}(\Gamma_{x,y})$.    
    Let $\gamma \in \Gamma_{x,y}$. If $\int_\gamma \rho = +\infty$ there is nothing to prove. Otherwise, applying Proposition \ref{prop:approx_integral_curve_with_subchains}, we have that
    $$\int_\gamma \rho \ge \lims_{j\to +\infty} \int_{\sfc_{t,n_j}} \rho \ge 1,$$
    for $\sfc_{t,n_j} \in \Chain{\frac{1}{n_j}}_{x,y}$ defined therein.    
    By the arbitrariness of $\rho$ we have
    $$\Mod_p(\Gamma_{x,y}, \mm_{x,y}^L) \le  \cModepsilon{\varepsilon}{p}(\Chain{}_{x,y}, \mathcal{F}_{x,y} ,\mm_{x,y}^L),$$
    for every $\varepsilon > 0$. By taking the limit as $\varepsilon \to 0$, we obtain \eqref{eq:modulus_estimate_2} and we conclude the proof.
    
\end{proof}
\begin{remark}
\label{rmk:chain_modulus_PI}
    Notice that the statement of Proposition \ref{prop:modulus_chain_modulus_limit} cannot be formulated with the class $\mathcal{F} = \{ g\colon \X \to [0,+\infty]\,:\, g \text{ Borel}\}$.
    Indeed, observe that $\cModepsilon{\varepsilon}{p}(\Chain{}_{x,y},\mm_{x,y}^L) = 0$ as soon as $\mm(\{x,y\}) = 0$, because of Lemma \ref{lemma:measure_zero_implies_modulezero}, while $\cModepsilon{\varepsilon}{p}(\Chain{}_{x,y}, \mathcal{F}_{x,y} ,\mm_{x,y}^L)$ can be different from $0$, because of Proposition \ref{prop:modulus_chain_modulus_limit}. The difference is due to the fact that $\cModepsilon{\varepsilon}{p}(\cdot, \mathcal{F}_{x,y} ,\mm_{x,y}^L)$ is not an outer measure. 
    Moreover, 
    \begin{equation}
        \begin{aligned}
            \cModepsilon{\varepsilon}{p}(\Chain{}_{x,y}, \mathcal{F}_{x,y} ,\mm_{x,y}^L) &= \cModepsilon{\varepsilon}{p}(\Chain{}_{x,y}, \mathcal{F}_{x,y} ,\bar{\mm}_{x,y}^L)\\
            \cModepsilon{\varepsilon}{p}(\Chain{}_{x,y}, \Lip(\X) ,\mm_{x,y}^L) &= \cModepsilon{\varepsilon}{p}(\Chain{}_{x,y}, \Lip(\bar{\X}) ,\bar{\mm}_{x,y}^L),
        \end{aligned}
    \end{equation}
    where the right hand sides are computed on the metric measure space $(\bar{\X}, \bar{\sfd}, \bar{\mm})$. In particular 
    $$\lim_{\varepsilon \to 0} \cModepsilon{\varepsilon}{p}(\Chain{}_{x,y}, \mathcal{F}_{x,y} ,\mm_{x,y}^L) = \lim_{\varepsilon \to 0} \cModepsilon{\varepsilon}{p}(\Chain{}_{x,y}, \Lip(\X), \mm_{x,y}^L)$$
    holds true in every doubling metric measure space without the completeness assumption.
\end{remark}
%
%
\begin{corollary}
    Let $(\X,\sfd,\mm)$ be a doubling metric measure space. Let $x,y\in \X$. Then the conditions of Proposition \ref{prop:pt_PI_equivalent_conditions_A_p_connectedness} are equivalent to the following. There exists $c>0$, $L\ge 1$ such that
    $$\lim_{\varepsilon \to 0} \cModepsilon{\varepsilon}{p}(\Chain{}_{x,y}, \mathcal{F}_{x,y} ,\mm_{x,y}^L) = \lim_{\varepsilon \to 0} \cModepsilon{\varepsilon}{p}(\Chain{}_{x,y}, \Lip(\X), \mm_{x,y}^L) \ge c\sfd(x,y)^{1-p}.$$
\end{corollary}
\begin{proof}
    As noticed in the proof of Proposition \ref{prop:pt_PI_equivalent_conditions_A_p_connectedness} and Remark \ref{rmk:chain_modulus_PI}, the conditions hold if and only if they hold on the metric completion $(\bar{X},\bar{\sfd},\bar{\mm})$. Therefore the result follows by Proposition \ref{prop:modulus_chain_modulus_limit} and \cite[Proposition A.1]{CaputoCavallucci2024II}, see also the original \cite[Theorem 2]{Kei03}.
\end{proof}
\subsection{Energy of separating sets via chains}
\label{sec:separating_sets_chains}
In the case $p=1$ we can extend our characterizations in \cite{CaputoCavallucci2024II} to possibly non complete metric spaces.
We need the notion of chain width of a given set $A \subset \X$, which is
\begin{equation}
    \Chain{}\textup{-}\width_{x,y}(A):= \lim_{\varepsilon \to 0}\inf_{\sfc \in \Chain{\varepsilon}_{x,y}} \int_{\sfc} \chi_A.
\end{equation}
%

%

%
We also need to recall the notion of separating set and of Minkowski content. Given $x,y\in \X$ we say that a set $\Omega \subseteq \X$ is separating if it is closed, $x$ belongs to the interior of $\Omega$ and $y$ belongs to $\Omega^c$. The family of separating sets between $x$ and $y$ is denoted by $\SS_{\textup{top}}(x,y)$.

Given a subset $A$ of a metric measure space $(\X,\sfd,\mm)$, we define its Minkowski content by
$$\mm^+(A) := \limi_{r\to 0} \frac{\mm(B_r(A)\setminus A)}{r}.$$
  
The following theorem is the chain version of \cite[Theorem 1.4]{CaputoCavallucci2024II} and it is suited for the case $p=1$.
\begin{theorem}
\label{theo:1-PI_equiv_BMC}
    Let $(\X,\sfd,\mm)$ be a doubling metric measure space. Let $x,y \in \X$. Then the following conditions are quantitatively equivalent:
    \begin{itemize}
        \item[(i)] there exist $C >0$, $L \ge 1$ such that \eqref{eq:ptwise_PI_lip} holds for every $u \in \Lip(\X)$;
        \item[(ii)] there exist $C>0$, $L \ge 1$ such that for every $A \subseteq \X$ it holds 
    \begin{equation}
        \Chain{}\textup{-}\width_{x,y}(A) \le C \mm_{x,y}^L(A).
    \end{equation}
        \item[(iii)] there exist $c >0$, $L \ge 1$ such that for every $\Omega \in \SS_{\textup{top}}(x,y)$ it holds $(\mm_{x,y}^L)^+(\Omega) \ge c$.
    \end{itemize}
\end{theorem}
\begin{proof}
    If (i) holds then item (ii) of Proposition \ref{prop:pt_PI_equivalent_conditions_A_p_connectedness} holds. Applying it to the Borel function $g=\chi_A$, where $A\subseteq \X$ is Borel, we get that
    $$\Chain{}\textup{-}\width_{x,y}(A) \le \lim_{\varepsilon \to 0} \inf_{\substack{\sfc \in \Chain{\varepsilon}_{x,y} \\ \ell(\sfc) \le C\sfd(x,y)}} \int_{\sfc} \chi_A \le C\int \chi_A\,\d\mm_{x,y}^L = C\mm_{x,y}^L(A).$$
    This shows (ii).
    
    We now assume (ii) and we consider $\Omega \in \SS_{\rm top}(x,y)$. Let $0< r < \min\{\sfd(x,\partial\Omega), \sfd(y,\partial\Omega)\}$. For $\varepsilon < r$, let $\sfc = \{q_i\}_{i=0}^N \in \Chain{\varepsilon}_{x,y}$ and let  $\sfc' = (q_m,\ldots,q_M)$ be a maximal subchain such that $q_i \in B_r(\Omega)\setminus \Omega$ for every $i=m,\ldots,M$. Therefore we have
    $$\int_{\sfc}\chi_{B_r(\Omega)\setminus \Omega} \geq \int_{\sfc'}\chi_{B_r(\Omega)\setminus \Omega} \geq r-2\varepsilon,$$
    by maximality of $\sfc'$. By taking the limit for $\varepsilon$ going to zero we find
    \begin{equation}
        \label{eq:estimate_cwidth_below}
        \Chain{}\textup{-}\width_{x,y}(B_r(\Omega)\setminus \Omega) \geq \limsup_{\varepsilon \to 0} (r - 2\varepsilon) = r.
    \end{equation}
    Hence we compute
    $$(\mm_{x,y}^L)^+(\Omega) = \limi_{r \to 0}\frac{\mm_{x,y}^L(B_r(\Omega)\setminus \Omega)}{r} \stackrel{\eqref{eq:estimate_cwidth_below}}{\geq} \limi_{r \to 0}\frac{\mm_{x,y}^L(B_r(\Omega)\setminus \Omega)}{\Chain{}\textup{-}\width_{x,y}(B_r(\Omega)\setminus \Omega)} \geq \frac{1}{C}.$$
This proves (iii). 

It remains only to prove (iii) implies (i). This is the proof of the last implication in \cite[Thm.\ 6.1]{CaputoCavallucci2024} that we report for completeness. Let $u\colon \X \to \R$, $u\geq 0$ be a bounded Lipschitz function and let $x,y\in \X$. We can assume that $u(x) < u(y)$ otherwise there is nothing to prove. The sets $\Omega_t := \lbrace u \geq t \rbrace$ belong to $\SS_\text{top}(x,y)$ for all $t\in (u(x),u(y))$. So we can apply the coarea inequality for the Minkowski content (see \cite[Lemma 3.2]{AmbDiMarGig17}) with respect to the measure $\mm_{x,y}^L$ to get
    $$c\, \vert u(x) - u(y) \vert \leq \int_{u(x)}^{u(y)} (\mm_{x,y}^L)^{+}(\lbrace u \geq t \rbrace) \,\d t \leq \int_\X \lip\,u \,\d\mm_{x,y}^L.$$
    Therefore item (i) follows with $C = 1/c$ for Lipschitz, nonnegative, bounded functions. A standard approximation argument gives the same estimate for all Lipschitz functions.
\end{proof}
\begin{remark}
    Condition (iii) of Theorem \ref{theo:1-PI_equiv_BMC} is denoted by (BMC)$_{x,y}$ in \cite{CaputoCavallucci2024II}, meaning `big Minkowski content'. If it holds for every couple of points of $\X$ with same constants we say that $(\X,\sfd,\mm)$ satisfies property (BMC). This property has been studied in \cite{CaputoCavallucci2024}, where it is shown to be equivalent to the $1$-Poincaré inequality if $(\X,\sfd, \mm)$ is complete and doubling. Other properties regarding the boundary of separating sets have been studied there, and they are all equivalent to the $1$-Poincaré inequality in case of complete, doubling metric measure spaces. In the non complete case, this is not true. 
    This happens since a separating set can have empty boundary like in the following example. Let $\X=\mathbb{R}^2\setminus \{x=0\}$. We consider $A =(-1,0)$ and $B=(1,0)$ and $\Omega:=(-\infty, 0) \times \mathbb{R} \in {\rm SS}_{\rm top}(A,B)$. Here $\partial \Omega=\emptyset$. Moreover, $\X$ satisfies a $1$-Poincaré inequality for all couples $(u,\lip\,u)$ with $u\in \Lip(\X)$, because of Corollary \ref{cor:PI_lip=chain}.    
    However, with the notation of \cite[Theorem 6.1]{CaputoCavallucci2024}, 
    we have that the set $\Omega$ above does not satisfy (BH),(BH$_\textup{R}$), (BH$^e$), (BH$^e_\textup{R}$), (BP), (BP$_\text{R}$), (BC), (BAM), (BAM)$^e$ and (BAM)$^\pitchfork$. The last condition (BAM)$^\pitchfork$ can be expressed via chains in a way that it is equivalent to (BMC).
\end{remark}

\bibliographystyle{alpha}
\bibliography{biblio.bib}

\end{document}